\documentclass[11pt,a4paper]{article}
\usepackage[utf8]{inputenc}
\usepackage[english]{babel}
\usepackage{tocbibind}
\usepackage{amsthm}
\usepackage{tikz-cd}
\usepackage{amsfonts}
\usepackage{amssymb}
\usepackage[T1]{fontenc}
\usepackage{stix}
\usepackage{ragged2e}
\usepackage{textcomp}
\usepackage{amssymb,amscd,amsthm,amsfonts,amsmath,verbatim}

\usepackage[toc,page]{appendix}
\usepackage{eucal}
\usepackage[all]{xy}
\usepackage{amsthm}
\usepackage{tikz}
\usepackage{stackengine}
\usepackage{scalerel}
\usetikzlibrary{decorations.markings}
\usetikzlibrary{matrix,arrows}
\usepackage{bbm}
\usepackage{textcomp}
\usepackage{setspace} 
\usetikzlibrary{shapes,arrows}
\usepackage{pgf}
\usepackage{wrapfig,lipsum,booktabs}
\usepackage{lipsum}
\usepackage{calrsfs}
\DeclareMathAlphabet{\pazocal}{OMS}{zplm}{m}{n}
\newcommand{\Aa}{\mathcal{A}}
\newcommand{\Ab}{\pazocal{A}}
\newcommand{\Ba}{\mathcal{B}}

\newcommand{\Da}{\mathcal{D}}
\newcommand{\Db}{\pazocal{D}}
\newcommand{\Ea}{\mathcal{E}}
\newcommand{\Eb}{\pazocal{E}}
\newcommand{\Fa}{\mathcal{F}}
\newcommand{\Fb}{\pazocal{F}}
\newcommand{\Ga}{\mathcal{G}}
\newcommand{\Gb}{\pazocal{G}}
\newcommand{\Ib}{\pazocal{I}}

\newcommand{\Jb}{\pazocal{J}}

\newcommand{\La}{\mathcal{L}}
\newcommand{\Lb}{\pazocal{L}}

\newcommand{\Mb}{\pazocal{M}}

\newcommand{\Nb}{\pazocal{N}}

\newcommand{\Ob}{\pazocal{O}}

\newcommand{\Pb}{\pazocal{P}}

\newcommand{\Sb}{\pazocal{S}}

\newcommand{\Tb}{\pazocal{T}}

\newcommand{\Ub}{\pazocal{U}}

 \newtheorem{lem}{Lemma}[subsection]
 \newtheorem{prop}[lem]{Proposition}
  \newtheorem{coro}[lem]{Corollary}
 \newtheorem{theo}[lem]{Theorem}
  \newtheorem{defi}[lem]{Definition}
  \newtheorem{rem}[lem]{Remark}
 
 \newtheorem{exa}[lem]{Example}
 
  \newtheorem{roots}[lem]{}
  \newtheorem{UEA_AD}[lem]{}
  \newtheorem{char_alg}[lem]{}  
 \newcommand{\C }{{\mathbb C }}
 \newcommand{\Q }{{\mathbb Q }}
 \newcommand{\G }{{\mathbb G }}
 \newcommand{\T }{{\mathbb T }}
 \newcommand{\B }{{\mathbb B }}
 \newcommand{\Z }{{\mathbb Z }}
 \newcommand{\N }{{\mathbb N }}
 
 \usepackage{graphicx}
\newcommand{\simeqd}{\mathrel{\rotatebox[origin=c]{-90}{$\simeq$}}}

\usepackage{vmargin}
\setpapersize{A4}
\setmargins{2.5cm}       
{1.5cm}                        
{17cm}                      
{23.42cm}                    
{10pt}                           
{1cm}                           
{0pt}                             
{2cm}
\begin{document}
\def \bangle{ \atopwithdelims \langle \rangle}
\title{\textbf{ A Beilinson-Bernstein theorem for twisted arithmetic differential operators on the formal flag variety}}
\author{ Andr\'{e}s Sarrazola Alzate}
\date{}
\maketitle
\begin{abstract}
Let $\Q_p$ be the field of $p$-adic numbers and $\G$ a split connected reductive group scheme over $\Z_p$. In this work we will introduce a sheaf of twisted arithmetic differential operators on the formal flag variety of $\G$ , associated to a general character. In particular, we will generalize the results of \cite{Huyghe2}, concerning the $\Da^{\dag}$-affinity of the smooth formal flag variety of $\G$, of certain sheaves of $p$-adically complete twisted arithmetic differential operators associated to an algebraic character, and the results of \cite{HS2} concerning the calculation of the global sections.
\end{abstract}
\justify

\makeatletter
\@starttoc{toc}
\makeatother

\section{Introduction}
\justify
An important theorem in group theory is the so-called Beilinson-Bernstein theorem \cite{BB}. Let us briefly recall its statement. Let $G$ be a semi-simple complex algebraic group with Lie algebra $\mathfrak{g}_{\C}:=\text{Lie}(G)$. Let $\mathfrak{t}_{\C}\subset\mathfrak{g}_{\C}$ be a Cartan subalgebra and $\mathfrak{z}\subset\Ub(\mathfrak{g}_{\C})$ the center of the universal enveloping algebra of $\mathfrak{g}_{\C}$. For every character $\lambda\in\mathfrak{t}_{\C}^*$ we denote by $\mathfrak{m}_{\lambda}\subset\mathfrak{z}$ the corresponding maximal ideal determined by the Harish-Chandra homomorphism. We put $\Ub_{\lambda}:=\Ub(\mathfrak{g}_{\C})/\mathfrak{m}_{\lambda}$. The theorem states that if $X$ is the flag variety associated to $G$ and $\Db_{X,\lambda}$ is the sheaf of $\lambda$-twisted differential operators on $X$, then $\text{Mod}_{qc}(\Db_{X,\lambda})\simeq\text{Mod}(\Ub_{\lambda})$, provided that $\lambda$ is a dominant and regular character. Here $\text{Mod}_{qc}(\Db_{X,\lambda})$ denotes the category of $\Ob_{X}$-quasi-coherent $\Db_{X,\lambda}$-modules. Moreover, under this equivalence of categories coherent $\Db_{X,\lambda}$-modules correspond to finitely generated $\Ub_{\lambda}$-modules \cite[théorème principal]{BB}. \\
The Beilinson-Bernstein theorem has been proved independently by A. Beilinson and J. Bernstein in \cite{BB} and, by J-L. Brylinski and M. Kashiwara in \cite{BK}. This result is an essential tool in the proof of Kazhdan-Lusztig's multiplicity conjecture \cite{KL}. In mixed characteristic, an important progress has been achieved by C. Huyghe in \cite{Huyghe1, Huyghe2} and Huyghe-Schmidt in \cite{HS2}. They use Berthelot's arithmetic differential operators \cite{Berthelot1} to prove an arithmetic version of the Beilinson-Bernstein theorem for the smooth formal flag variety over $\Z_p$\footnote{In reality over a discrete valuation ring.}. In this setting, the global sections of these operators equal a crystalline version of the classical distribution algebra $\text{Dist}(G)$ of the group scheme $G$.
\justify
Let $\Q_p$ be the field of $p$-adic numbers. Throughout this paper $\G$ will denote a split connected reductive group scheme over $\Z_p$, $\B\subset\G$ a Borel subgroup and $\T\subset\B$ a split maximal torus of $\G$. We will also denote by  $X=\G/\B$ the smooth flag $\Z_p$-scheme associated to $\G$. In this work, we introduce sheaves of \textit{twisted differential operators} on the formal flag scheme $\mathfrak{X}$ and we show an arithmetic Beilinson-Bernstein correspondence. Here the twist is respect to a morphism of $\Z_p$-algebras $\lambda:\text{Dist}(\T)\rightarrow\Z_p$, where $\text{Dist}(\T)$ denotes the distribution algebra in the sense of \cite{DG}. Those sheaves are denoted by $\Da^{\dag}_{\mathfrak{X},\lambda}$. In particular, there exists a basis $\Sb$ of $\mathfrak{X}$ consisting of affine open subsets, such that for every $\mathfrak{U}\in\Sb$ we have
\begin{eqnarray*}
\Da^{\dag}_{\mathfrak{X},\lambda}|_{\mathfrak{U}}\simeq \Da^{\dag}_{\mathfrak{U}}.
\end{eqnarray*} 
In other words, locally we recover the sheaves of Berthelots' differential operators, \cite{Berthelot1} (of course, this clarifies why they are called twisted differential operators).
\justify
To calculate their global sections, we remark for the reader  that actually we dispose of a description of the distribution algebra $\text{Dist}(\G)$ as an inductive limit of filtered noetherian $\Z_p$-algebras $\text{Dist}(\G)=\varinjlim_{m\in\N}D^{(m)}(\G)$, such that for every $m\in\N$ we have $D^{(m)}(\G)\otimes_{\Z_p}\Q_p=\Ub(\text{Lie}(\G)\otimes_{\Z_p}\Q_p)$ the universal enveloping algebra of Lie$(\G)\otimes_{\Z_p}\Q_p$. For instance $D^{(0)}(\G)=\Ub(\text{Lie}(\G))$. In particular, every character $\lambda:\text{Dist}(\T)\rightarrow\Z_p$ induces, via tensor product with $\Q_p$, an infinitesimal character $\chi_{\lambda}$. The last one allows us to define $\hat{D}^{(m)}(\G)_{\lambda}$ as the $p$-adic completion of the central reduction $D^{(m)}(\G)/(D^{(m)}(\G)\cap\text{ker}(\chi_{\lambda+\rho}))$, and we denote by $D^{\dag}(\G)_{\lambda}$ the limit of the inductive system $\hat{D}^{(m)}(\G)_{\lambda}\otimes_{\Z_p}\Q_p\rightarrow \hat{D}^{(m')}(\G)_{\lambda}\otimes_{\Z_p}\Q_p$. In this paper we show the following result
\justify
\textbf{Theorem.} \textit{ Let us suppose that $\lambda : \text{Dist}(\T)\rightarrow\Z_p$ is a character of $\text{Dist}(\T)$ such that $\lambda + \rho \in \mathfrak{t}^*_\Q$ \footnote{Here the character $\lambda$ is induced via tensor product with $\Q_p$ and the correspondence (\ref{Iso_chars_Z_p}).} ($\mathfrak{t}_\Q := \text{Lie}(\T)\otimes_{\Z_p}\Q_p$) is a dominant and regular character of $\mathfrak{t}^*_{\Q}$. Then the global sections functor induces an equivalence between the categories of coherent $\Da^{\dag}_{\mathfrak{X},\lambda}$-modules and finitely presented $D^{\dag}(\G)_{\lambda}$-modules.}
\justify
This theorem is based on a refined version for the sheaves (of \textit{level m twisted differential operators} which we will define later) $\hat{\Da}^{(m)}_{\mathfrak{X},\lambda,\Q}$. As in the classical case, the inverse functor is determined by the \textit{localization functor} 
\begin{eqnarray*}
\Lb oc_{\mathfrak{X},\lambda}^{\dag}(\bullet):=\Da^{\dag}_{\mathfrak{X},\lambda}\otimes_{D^{\dag}(\G)_{\lambda}}(\bullet),
\end{eqnarray*}
with a completely analogous definition for every $m\in\N$. 
\justify
The first section of this work is devoted to fix some important arithmetic definitions. They are necessary to define an integral model  (in the sense of definition \ref{defi I.models}) of the usual sheaf of twisted differential operators on the smooth flag variety $X_\Q:=X\times_{\text{Spec}(\Z_p)}\text{Spec}(\Q_p)$, associated to the split connected reductive algebraic group $\G_\Q:=\G\times_{\text{Spec}(\Z_p)}\text{Spec}(\Q_p)$. One of the most important is the algebra of distributions of level $m$, which is denoted by $D^{(m)}(\G)$ and it is introduced in subsection 2.5. This is a filtered noetherian $\Z_p$-algebra which plays a fundamental roll in the rest of our work.
\justify
Inspired by the works \cite{AW}, \cite{BMR} and \cite{BB1},  in the first part of the third section we construct our \textit{level m twisted arithmetic differential operators} on the formal flag scheme $\mathfrak{X}$. To do this, we denote by $\mathfrak{t}$ the commutative $\Z_p$-Lie algebra of the maximal torus $\T$ and by $\mathfrak{t}_\Q:=\mathfrak{t}\otimes_{\Z_p}\Q_p$. They are Cartan subalgebras of $\mathfrak{g}$ and $\mathfrak{g}_\Q$, respectively. Let \textbf{N} be the unipotent radical subgroup of the Borel subgroup $\B$ and let us define
\begin{eqnarray*}
\widetilde{X}:=\G/\textbf{N},\;\;\; X:=\G/\B
\end{eqnarray*}
the basic affine space and the flag scheme of $\G$. These are smooth and separated schemes over $\Z_p$, and $\widetilde{X}$ is endow with commuting  $(\G,\T)$-actions making the canonical projection $\xi:\widetilde{X}\rightarrow X$ a locally trivial $\T$-torsor for the Zariski topology on $X$. Following \cite{BB1},  the right $\T$-action on $\widetilde{X}$ allows to  define the \textit{level m relative enveloping algebra} of the torsor $\xi$ as the sheaf of $\T$-invariants of $\xi_*\Db_{\widetilde{X}}^{(m)}$:
\begin{eqnarray*}
\widetilde{\Db^{(m)}}:=\left(\xi_*\Db_{\widetilde{X}}^{(m)}\right)^{\T}.
\end{eqnarray*}
As we will explain later this is a sheaf of $D^{(m)}(\T)$-modules and we will show that over an affine open subset $U\subset X$ that trivialises the torsor, the sheaf $\widetilde{\Db^{(m)}}$ may be described as the tensor product $\Db^{(m)}_{X}|_{U}\otimes_{\Z_p}D^{(m)}(\T)$ (this is the arithmetic analogue of \cite[page 180]{BB1}).
\justify
Two fundamental properties of the distribution algebra $\text{Dist}(\T)$  are: it constitutes an integral model of the universal enveloping algebra $\Ub(\mathfrak{t}_\Q)$ and  $\text{Dist}(\T)=\varinjlim_{m}D^{(m)}(\T)$. These properties allow us to introduce the \textit{central reduction} as follows. First of all, we say that an $\Z_p$-algebra map $\lambda:\text{Dist}(\T)\rightarrow\Z_p$ is a \textit{character} of $\text{Dist}(\T)$  (cf. \ref{UEA_AD}). By the properties just stated, it induces a character of the Cartan subalgebra $\mathfrak{t}_\Q$ \footnote{In order to soft the notation through this work we will always denote by $\lambda$ the character $\lambda\otimes 1_{\Q_p}$. This should not cause any confusion to the reader.}
\begin{eqnarray*}
(\lambda=)\;\lambda\otimes1 :\mathfrak{t}_\Q\hookrightarrow \Ub(\mathfrak{t}_\Q)=\text{Dist}(\T)\otimes_{\Z_p}\Q_p\rightarrow \Q_p.
\end{eqnarray*}
Now, let us consider the ring $\Z_p$ as a $D^{(m)}(\T)$-module via $\lambda$. We can define the sheaf of \textit{level m twisted arithmetic differential operators} on the flag scheme $X$ by
\begin{eqnarray}\label{defi twist}
\Db_{X,\lambda}^{(m)}:=\widetilde{\Db^{(m)}}\otimes_{D^{(m)}(\T)}\Z_p.
\end{eqnarray}
This is a sheaf of $\Z_p$-algebras which is an integral model of $\Db_{\lambda}$.\footnote{Tensoring with $\Q_p$ and restricting to the generic fiber $X_\Q\hookrightarrow X$ equals $\Db_{\lambda}$, \cite[page 170]{BB1}.}
\justify
The second part of the third section is dedicated to explore some finiteness properties of the cohomology of coherent $\Db^{(m)}_{X,\lambda}$-modules. Notably important is the case when the character $\lambda + \rho\in \mathfrak{t}_\Q^*$ is dominant and regular. Under this assumption, the cohomology groups of every coherent $\Db^{(m)}_{X,\lambda}$-module have the nice property of being of finite $p$-torsion. This is one of the central results in this work.
\justify
In section 4 we will consider the $p$-adic completion of (\ref{defi twist}). It will be denoted by $\widehat{\Da}^{(m)}_{\mathfrak{X},\lambda}$ and the sheaf $\widehat{\Da}^{(m)}_{\mathfrak{X},\lambda,\Q}:=\widehat{\Da}^{(m)}_{\mathfrak{X},\lambda}\otimes_{\Z_p}\Q_p$ is our sheaf of \textit{level m twisted arithmetic differential operators} on the formal flag scheme $\mathfrak{X}$. 
\justify
Let $Z(\mathfrak{g}_\Q)$ be the center of the universal enveloping algebra $\Ub(\mathfrak{g}_\Q)$. From now on, we will assume that $\lambda + \rho\in\mathfrak{t}^*_\Q$ is dominant and regular, and that $\chi_{\lambda}:Z(\mathfrak{g}_\Q)\rightarrow \Q_p$ is the central character induced by $\lambda$ via the classical Harish-Chandra homomorphism. We will show that if Ker$(\chi_{\lambda + \rho})_{\Z_p}:=D^{(m)}(\G)\cap\text{Ker}(\chi_{\lambda+\rho})$ and $\widehat{D}^{(m)}(\G)_{\lambda}$ denotes the $p$-adic completion of the central reduction  $D^{(m)}(\G)_{\lambda}:=D^{(m)}(\G)/D^{(m)}(\G)\text{Ker}(\chi_{\lambda+\rho})_{\Z_p}$, then we have an isomorphism of complete $\Z_p$-algebras 
\begin{eqnarray*}
\widehat{D}^{(m)}(\G)_{\lambda}\otimes_{\Z_p}\Q_p \xrightarrow{\simeq} H^0\left(\mathfrak{X},\widehat{\Da}^{(m)}_{\mathfrak{X},\lambda,\Q}\right).
\end{eqnarray*}
Finally, in subsection 4.3 we will introduce the localization functor $\La oc^{(m)}_{\mathfrak{X},\lambda}$, from the category of finitely generated $\widehat{D}^{(m)}(\G)_{\lambda}\otimes_{\Z_p}\Q_p$-modules to the category of coherent $\widehat{\Da}^{(m)}_{\mathfrak{X},\lambda,\Q}$-modules as the sheaf associated to the presheaf defined by
\begin{eqnarray*}
\Ub\subseteq \mathfrak{X}\mapsto \widehat{\Da}^{(m)}_{\mathfrak{X},\lambda,\Q}(\Ub)\otimes_{\widehat{D}^{(m)}(\G)_{\lambda}\otimes_{\Z_p}\Q_p}E,
\end{eqnarray*}
where $E$ is a finitely generated $\widehat{D}^{(m)}(\G)_{\lambda}\otimes_{\Z_p}\Q_p$-module. We will show
\justify
 \textbf{ theorem:}
Let us suppose that $\lambda:\text{Dist}(\T)\rightarrow\Z_p$ is a character of $\text{Dist}(\T)$ such that $\lambda + \rho \in \mathfrak{t}^*_\Q$ is a dominant and regular character of $\mathfrak{t}_\Q$.
\begin{itemize}
\item[(i)] The functors $\La oc^{(m)}_{\mathfrak{X},\lambda}$ and $H^{0}(\mathfrak{X},\bullet)$ are quasi-inverse equivalences of categories between the abelian categories of finitely generated (left) $\widehat{D}^{(m)}(\G)_{\lambda}\otimes_{\Z_p}\Q_p$-modules and coherent  $\widehat{\Da}^{(m)}_{\mathfrak{X},\lambda,\Q}$-modules.
\item[(ii)] The functor $\La oc^{(m)}_{\mathfrak{X},\lambda}$ is an exact functor.
\end{itemize}
\justify
Section 5 of this work is devoted to treat the problem of passing to the inductive limit. In fact, if $m\le m'$ the sheaves $\hat{\Da}^{(m)}_{\mathfrak{X},\lambda,\Q}$ and $\hat{\Da}^{(m')}_{\mathfrak{X},\lambda,\Q}$ are related via a natural map $\hat{\Da}^{(m)}_{\mathfrak{X},\lambda,\Q}\rightarrow\hat{\Da}^{(m')}_{\mathfrak{X},\lambda,\Q}$ and we can define $\Da^{\dag}_{\mathfrak{X},\lambda}$ to be the inductive limit. Similarly, we dispose of canonical maps $\widehat{D}^{(m)}(\G)_{\lambda}\otimes_{\Z_p}\Q_p\rightarrow\widehat{D}^{(m')}(\G)_{\lambda}\otimes_{\Z_p}\Q_p$ and we denote by $D^{\dag}(\G)_{\lambda}$ its limit. Introducing the localization functor $\La oc^{\dag}_{\mathfrak{X},\lambda}$ exactly as we have made before, we get an analogue of the previous theorem for the sheaves $\Da^{\dag}_{\mathfrak{X},\lambda}$.
\justify
The work developed by Huyghe in \cite{Huyghe1} and by D. Patel, T. Schmidt and M. Strauch in \cite{PSS1}, \cite{PSS2} and \cite{HPSS}, shows that the Beilinson-Berstein theorem is an important tool in the following localization theorem \cite[theorem 5.3.8]{HPSS}: if $\mathfrak{X}$ denotes the formal flag scheme of a split connected reductive group $\G$, then the theorem provides an equivalence of categories between the category of admissible locally analytic $\G(\Q_p)$- representations (with trivial character!) \cite{PT} and a category of coadmissible equivariant arithmetic $\Da$-modules (on the family of formal models of the rigid analytic flag variety of $\G$). Our motivation is to study this localization in the twisted case (in a work still in progress, we have proved the affinity of an admissible formal blow-up of $\mathfrak{X}$ for the sheaf of twisted arithmetic differential operators with a congruence level $k$ \footnote{The point here is that an admissible formal blow-up is not necessarily smooth and therefore, we can not use immediately Berthelot's differential operators. For the definition of these differential operators the reader is invited to visit  \cite{HSS}.}. This is a fundamental result to achieve the previous localization theorem \cite{HPSS}). 
\justify
The case of a finite extension $L$ of $\Q_p$ seems to present several technical problems (the reader can take a look to \ref{UEA_AD} to see one of this obstacles). We will treat this case in a future work. 
\justify
\textbf{Acknowledgements:} The present article contains a part of the author PhD thesis written at the Universities of Strasbourg and Rennes 1 under the supervision of C. Huyghe and T. Schmidt. Both have always been very patient and attentive supervisors. For this, I express my deep gratitude to them. 
\justify
\textbf{Notation:} Throughout this work $p$ will denote a prime number and $\Z_p$ the ring of $p$-adic integers. Furthermore, if $X$ is an arbitrary noetherian scheme over $\Z_p$ and $j\in\N$, then we will denote by $X_j:=X\times_{\text{Spec}(\Z_p)}\text{Spec}(\Z_p/p^{j+1})$ the reduction modulo $p^{j+1}$, and by
\begin{eqnarray*}
\mathfrak{X}= \varinjlim_{j}X_j
\end{eqnarray*}
the formal completion of $X$ along the special fiber. Moreover, if  $\Eb$ is a sheaf of $\Z_p$-modules on $X$ then its $p$-adic completion $\Ea:=\varprojlim_{j}\Eb/p^{j+1}\Eb$ will be considered as a sheaf on $\mathfrak{X}$. Finally, the base change of a sheaf of $\Z_p$-modules on $X$ (resp. on $\mathfrak{X}$) to $\Q_p$ will always be denoted by the subscript $\Q$. For instance $\Eb_{\Q}:=\Eb\otimes_{\Z_p}\Q_p$ (resp. $\Ea_{\Q}:=\Ea\otimes_{\Z_p}\Q_p$).
\newpage
\section{Arithmetic definitions}
In this section we will describe the arithmetic objects on which the definitions and constructions of our work are based. We will give their functorial constructions and we will enunciate  their most remarkable properties. For a more detailed approach, the reader is invited to take a look to the references \cite{Berthelot2}, \cite{Huyghe1}, \cite{HS1} and \cite{Berthelot1}.
\subsection{Partial divided power structures of level m}
Let $p\in\Z$ be a prime number. In this subsection $\Z_{(p)}$ denotes the localization of  $\Z$ with respect to the prime ideal $(p)$.
\justify
We start recalling the following definition \cite[Definition 3.1]{Berthelot2}.

\begin{defi}\label{PD-structure}
Let $A$ be a commutative ring and $I\subset A$ an ideal. By a structure of divided powers on I we mean a collection of maps $\gamma_{i}:I\rightarrow A$ for all integers $i\ge 0$, such that
\begin{itemize}
\item[(i)] For all $x\in I$, $\gamma_{0}(x)=1$, $\gamma_{1}(x)=x$ and $\gamma_{i}(x)\in I$ if $i\ge 2$.
\item[(ii)] For $x,y\in I$ and $k\ge 1$ we have $\gamma_{k}(x+y)=\sum_{i+j=k}\gamma_i(x)\gamma_{j}(y)$.
\item[(iii)] For $a\in A$ and $x\in I$ we have $\gamma_k(ax)=a^k\gamma_k(x)$.
\item[(iv)] For $x\in I$ we have $\gamma_{i}(x)\gamma_{j}(x)=((i,j))\gamma_{i+j}(x)$, where $((i,j)):=(i+j)!(i!)^{-1}(j!)^{-1}$.
\item[(v)] We have $\gamma_{p}(\gamma_{q}(x))=C_{p,q}\gamma_{pq}(x)$, where $C_{p,q}:=(pq)!(p!)^{-1}(q!)^{-p}$.
\end{itemize}
\end{defi}
\justify
Throughout this paper we will use the terminology: "$(I,\gamma)$ is a PD-ideal", "$(A,I, \gamma)$ is a PD-ring" and "$\gamma$ is a PD-structure on $I$". Moreover, we say that $\phi: (A,I,\gamma)\rightarrow (B,J,\delta)$ is a \textit{PD-homomorphism} if $\phi: A\rightarrow B$ is a homorphism of rings such that $\phi(I)\subset J$ and $\delta_k\circ\phi|_I=\phi\circ\gamma_k$, for every $k\ge 0$.

\begin{exa}\label{Example}\cite[Section 3, Examples 3.2 (3)]{Berthelot2}
Let $L|\Q_p$ be a finite extension of $\Q_p$, such that $\mathfrak{o}$ is its valuation ring and $\varpi$ is a uniformizing parameter. Let us write $p=u\varpi^{e}$, with $u$ a unit and $e$ a positive integer (called the absolute ramification index of $\mathfrak{o}$). Then $\gamma_{k}(x):= x^{k}/k!$ defines a PD-structure on $(\varpi)$ if and only if $e\le p-1$.   In particular, we dispose of a PD-structure on $(p)\subset \Z_{p}$. We let $x^{[k]}:=\gamma_{k}(x)$ and we denote by $((p), [\; ])$ this PD-ideal.
\end{exa}
\justify
Let us fix a positive integer $m\in\Z$. For the next terminology we will always suppose that $(A,I,\gamma)$ is a $\Z_{(p)}$-PD-algebra whose PD-structure is compatible (in the sense of \cite[subsection 1.2]{Berthelot1}) with the PD-structure induced by $((p), [\; ])$ (we recall to the reader that the PD-structure $((p), [\;])$ always \textit{extends} to a PD-structure on any $\Z_{(p)}$-algebra \cite[proposition 3.15]{Berthelot2} ). We will also denote by $I^{(p^m)}$ the ideal generated by $x^{p^m}$, with $x\in I$.

\begin{defi}
Let $m$ be a positive integer. Let $A$ be a $\Z_{(p)}$-algebra  and $I\subset A$ an ideal. We call a m-PD-structure on $I$ a PD-ideal $(J,\gamma)\subset A$ such that $I^{(p^m)}+pI\subset J$. 
\end{defi}
\justify
We will say that $(I,J,\gamma)$ is a m-PD-ideal of $A$. Moreover, we say that $\phi: (A,I,J,\gamma)\rightarrow (A',I',J',\gamma')$ is a m-PD-morphism if $\phi:A\rightarrow A'$ is a ring morphism such that $\phi(I)\subset I'$, and such that $\phi: (A,J,\gamma)\rightarrow (A',J',\gamma')$ is a PD-morphism. \\
For every $k\in\N$ we denote by $k=p^mq+r$ the Euclidean division of $k$ by $p^m$, and for every $x\in I$ we define $x^{\{k\}_{(m)}}:=x^{r}(\gamma_{q}(x^{p^m}))$. We remark to the reader that the relation $q!\gamma_{q}(x)=x^q$ (which is an easy consequence of (i) and (iv) of definition \ref{PD-structure}) implies that $q!x^{\{k\}_{(m)}}=x^{k}$.\\
On the other side, the m-PD-structure $(I,J,\gamma)$ allows us to define an increasing filtration $(I^{\{n\}})_{n\in\N}$ on the ring $A$ which is finer that the $I$-adic filtration and called the \textit{m-PD-filtration}. It is characterized by the following conditions \cite[1.3]{Berthelot3}:
\begin{itemize}
\item[(i)] $I^{\{0\}}=A$, $I^{\{1\}}=I$.
\item[(ii)] For every $n\ge 1$, $x\in I^{\{n\}}$ and $k\ge 0$ we have $x^{\{k\}}\in I^{\{kn\}}$.
\item[(iii)] For every $n\ge 0$, $(J+pA)\cap I^{\{n\}}$ is a PD-subideal of $(J+pA)$.
\end{itemize}
\begin{prop}\label{m-PD-structure} \cite[proposition 1.4.1]{Berthelot2}
Let us suppose that $R$ is a $\Z_{(p)}$-algebra endowed with a m-PD-structure $(\mathfrak{a}, \mathfrak{b},\alpha)$. Let $A$ be a $R$-algebra and $I\subset A$ an ideal. There exists an $R$-algebra $P_{(m)}(I)$, an ideal $\overline{I}\subset P_{(m)}(I)$ endowed with a m-PD-structure $(\tilde{I}, [\; ])$ compatible with $(\mathfrak{b},\alpha)$, and a ring homomorphism $\phi:A\rightarrow P_{(m)}(I)$ such that $\phi(I)\subset \overline{I}$. Moreover, $(P_{(m)}(I),\overline{I},\tilde{I},[\;],\phi)$ satisfies the following universal property: for every $R$-homomorphism $f:A\rightarrow A'$ sending $I$ to an ideal $I'$ which is endowed with a m-PD-structure $(J',\gamma')$ compatible with $(\mathfrak{b},\alpha)$, there exists a unique m-PD-morphism $g:(P_{(m)}(I),\overline{I},\tilde{I}, [\;])\rightarrow (A',I',J',\gamma')$ such that $g\circ\phi=f$.
\end{prop}

\begin{defi}
Under the hypothesis of the preceding proposition, we call the $R$-algebra $P_{(m)}(I)$, endowed with the m-PD-ideal $(\bar{I},\tilde{I},[\; ])$, the m-PD-envelope of $(A,I)$. 
\end{defi}
\justify
Finally, if we endow $P^{n}_{(m)}(I):=P_{(m)}(I)/\overline{I}^{\{n+1\}}$ with the quotient m-PD-structure \cite[1.3.4]{Berthelot1} we have
\begin{coro}\label{order n mPD}\cite[Corollary 1.4.2]{Berthelot1}
Under the hypothesis of the preceding proposition, there exists an $R$- algebra $P^{n}_{(m)}(I)$ endowed with a m-PD-structure $(\overline{I},\tilde{I}, [\; ])$ compatible with $(\mathfrak{b},\alpha)$ and such that $\overline{I}^{\{n+1\}}=0$. Moreover, there exists an $R$-homomorphism $\phi_n: A\rightarrow P^{n}_{(m)}(I)$ such that $\phi(I)\subset \overline{I}$, and universal for the $R$-homomorphisms $A\rightarrow (A',I',J',\gamma')$ sending $I$ into a m-PD-ideal $I'$ compatible with $(\mathfrak{b},\alpha)$ and such that $I'^{\{n+1\}}=0$.
\end{coro}

\subsection{Arithmetic differential operators}\label{ADP}
Let us suppose that $\Z_p$ is endowed with the m-PD-structure defined in example \ref{Example} (cf. \cite[Subsection 1.3, example (i)]{Berthelot1}). Let $X$ be a smooth $\Z_p$-scheme, and $\Ib\subset\Ob_{X}$ a quasi-coherent ideal. The presheaves
\begin{eqnarray*}
U\mapsto  P_{(m)}(\Gamma(U,\Ib)) \;\;\;\text{and}\;\;\; U\mapsto  P^{n}_{(m)}(\Gamma(U,\Ib))
\end{eqnarray*} 
are sheaves of quasi-coherent $\Ob_{X}$-modules which we denote by $\Pb_{(m)}(\Ib)$ and $\Pb^{n}_{(m)}(\Ib)$. In a completely analogous way, we can define a canonical ideal $\overline{\Ib}$ of $\Pb_{(m)}(\Ib)$, a sub-PD-ideal $(\tilde{\Ib}, [\; ])\subset \bar{\Ib}$, and the sequence of ideals $(\bar{\Ib}^{\{n\}})_{n\in\N}$ defining the m-PD-filtration. Those are also quasi-coherent sheaves \cite[subsection 1.4]{Berthelot1}.\\
\medskip
Now, let us consider the diagonal embedding $\Delta: X\hookrightarrow X\times_{\Z_p} X$ and let $W\subset X\times_{\Z_p}X$ be an open subset such that $X\subset W$ is a closed subset, defined by a quasi-coherent sheaf $\Ib\subset\Ob_{W}$. For every $n\in\N$, the algebra $\Pb^{n}_{X,(m)}:=\Pb^{n}_{(m)}(\Ib)$ is quasi-coherent and its support is contained in $X$. In particular, it is independent of the open subset $W$ \cite[2.1]{Berthelot1}. Moreover, by proposition \ref{m-PD-structure} the projections $p_1,p_2: X\times_{\Z_p}X\rightarrow X$ induce two morphisms $d_1,d_2:\Ob_{X}\rightarrow \Pb^{n}_{X,(m)}$ endowing $ \Pb^{n}_{X,(m)}$ of a \textit{left} and a \textit{right} structure of $\Ob_{X}$-algebra, respectively. 

\begin{defi}\label{Diff order n}
Let $m,n$ be positive integers. The sheaf of differential operators of level $m$ and order less or equal to n on $X$ is defined by
\begin{eqnarray*}
\Db^{(m)}_{X,n}:=\mathcal{H}om_{\Ob_{X}}(\Pb^n_{X,(m)},\Ob_{X}).
\end{eqnarray*}
\end{defi}
\justify
If $n\le n'$ corollary \ref{order n mPD} gives us a canonical surjection $\Pb^{n'}_{X,(m)}\rightarrow\Pb^{n}_{X,(m)}$ which induces the injection $\Db^{(m)}_{X,n}\hookrightarrow\Db_{X,n'}^{(m)}$ and the sheaf of \textit{differential operators of level} $m$ is defined by
\begin{eqnarray*}
\Db_{X}^{(m)}:=\bigcup_{n\in\N}\Db^{(m)}_{X,n}.
\end{eqnarray*}
We remark for the reader that by definition $\Db_{X}^{(m)}$ is endowed with a natural filtration called the \textit{order filtration}, and like the sheaves $\Pb^{n}_{X,(m)}$, the sheaves $\Db_{X,n}^{(m)}$ are endowed with two natural structures of $\Ob_{X}$-modules. Moreover, the sheaf $\Db^{(m)}_{X}$ acts on $\Ob_{X}$: if $P\in\Db^{(m)}_{X,n}$, then this action is given by the composition $\Ob_{X}\xrightarrow{d_1} \Pb^{n}_{X,(m)}\xrightarrow{P} \Ob_{X}$.\\
Finally, let us give a local description of $\Db^{(m)}_{X,n}$. Let $U$ be a smooth open affine subset of $X$ endowed with a family of local coordinates $x_1,\;.\;.\;.\;,x_N$. Let $dx_1,\;.\;.\;.\;,dx_N$ be a basis of $\Omega_{X}(U)$ and $\partial_{x_1},\;.\;.\;.\;,\partial_{x_N}$ the dual basis of $\Tb_X(U)$ (as usual, $\Tb_{X}$ and $\Omega_X$ denote the tangent and cotangent sheaf on $X$, respectively). Let $\underline{k}\in\N^N$. Let us denote by $|\underline{k}|=\sum_{i=1}^{N}k_i$ and   $\partial_i^{[k_i]}=\partial_{x_i}/k_i!$  for every $1\le i\le N$. Then, using multi-index notation, we have $\underline{\partial}^{[\underline{k}]}=\prod_{i=1}^{N}\partial_i^{[k_i]}$ and $\underline{\partial}^{<\underline{k}>}=q_{\underline{k}}!\underline{\partial}^{[\underline{k}]}$. In this case, the sheaf $\Db_{X,n}^{(m)}$ has the following description on $U$
\begin{eqnarray}\label{locally_Berthelot}
\Db_{X,n}^{(m)}(U)=\bigg\{\sum_{|\underline{k}|\le n}a_{\underline{k}}\underline{\partial}^{<\underline{k}>}\;|\; a_{\underline{k}}\in\Ob_{X}(U)\;\text{and}\;\underline{k}\in\N^{N}\bigg\}.
\end{eqnarray}

\subsection{Symmetric algebra of finite level}
In this subsection we will focus on introducing the constructions in \cite{Huyghe1}. Let $X$ be a $\Z_p$-scheme, $\Lb$ a locally free $\Ob_X$-module of finite rank, $\textbf{S}_X(\Lb)$ the symmetric algebra associated to $\Lb$ and $\Ib$ the ideal of homogeneous elements of degree 1. Using the notation of section 2.1 we define
\begin{eqnarray}
\Gamma_{X,(m)}(\Lb):=\Pb_{\textbf{S}_{X}(\Lb),(m)}(\Ib)\;\;\;\;\text{and}\;\;\;\; \Gamma_{X,(m)}^n(\Lb):= \Gamma_{X,(m)}(\Lb)/\bar{\Ib}^{\{n+1\}}.
\end{eqnarray}
Those algebras are graded \cite[Proposition 1.3.3]{Huyghe1}, and if $\eta_1,...,\eta_N$ is a local basis of $\Lb$, we have 
\begin{eqnarray*}
\Gamma_{X,(m)}^{n}(\Lb)=\displaystyle\bigoplus_{|\underline{l}|\le n}\Ob_{X}\underline{\eta}^{\{\underline{l}\}}.
\end{eqnarray*}
As before $\underline{\eta}^{\{\underline{l}\}}=\prod_{i=1}^{N}\eta_i^{\{l_i\}}$ and $q_i!\eta_i^{\{l_i\}}=\eta^{l_i}$. We define by duality 
\begin{eqnarray*}
\text{Sym}^{(m)}(\Lb):=\displaystyle\bigcup_{k\in\N}\mathcal{H}om_{\Ob_X}\left(\Gamma_{X,(m)}^k(\Lb^{\vee}), \Ob_{X}\right),
\end{eqnarray*}
By \cite[Propositions 1.3.1, 1.3.3 and 1.3.6]{Huyghe1} we know that $\text{Sym}^{(m)}(\Lb)=\oplus_{n\in\N}\text{Sym}^{(m)}_n(\Lb)$ is a commutative graded algebra with noetherian sections over any open affine subset. Moreover, locally  over a basis $\eta_1,...,\eta_N$ we have the following description
\begin{eqnarray*}
\text{Sym}^{(m)}_n(\Lb)=\displaystyle\bigoplus_{|\underline{l}|=n}\Ob_{X}\underline{\eta}^{<\underline{l}>},\;\;\;\text{where}\;\;\; \frac{l_i !}{q_i!}\eta_i^{<l_i>}=\eta_i^{l_i}.
\end{eqnarray*}

\begin{rem}
By \cite[A.10]{Berthelot2} we have that $\textbf{S}^{(0)}_X(\Lb)$ is the symmetric algebra of $\Lb$, which justifies the terminology.
\end{rem}
\justify
We end this subsection by  remarking the following results from \cite{Huyghe1}. Let $\Ib$ be the kernel of the comorphism $\Delta^{\sharp}$ of the diagonal embedding  $\Delta: X\rightarrow X\times_{\text{Spec}(\Z_p)}X$. In \cite[Proposition 1.3.7.3]{Huyghe1} Huyghe shows that the graded algebra associated to the m-PD-adic filtration of $\Pb_{X,(m)}$ it is identified with the graded m-PD-algebra $\Gamma_{X,(m)}(\Ib/\Ib^2)=\Gamma_{X,(m)}(\Omega_{X}^1)$. More exactly, we have canonical isomorphisms $\Gamma^{n}_{X,(m)}:=\Gamma^n_{X,(m)}(\Omega_{X}^1) \xrightarrow{\simeq}gr_{\bullet}(\Pb^n_{X,(m)})$ which, by definition, induce a graded isomorphism of algebras
\begin{eqnarray}\label{graded}
\text{Sym}^{(m)}(\Tb_{X})\xrightarrow{\simeq} gr_{\bullet}\Db^{(m)}_{X}.
\end{eqnarray}
\subsection{Arithmetic distribution algebra of finite level}\label{ADA}
As in the introduction, let us consider $\G$ a split connected reductive group scheme over $\Z_p$ and $m\in\N$ fixed. We propose to give a description of the algebra of distributions of level $m$ introduced in \cite{HS1}. Let $I$ denote the kernel of the surjective morphism of $\Z_p$-algebras $\epsilon_\G:\Z_p[\G]\rightarrow \Z_p$, given by the identity element of $\G$. We know that $I/I^2$ is a free $\Z_p=\Z_p[\G]/I$-module of finite rank. Let $t_1,\;.\;.\;.\;,t_l\in I$ such that modulo $I^2$, these elements form a basis of $I/I^2$. The $m$-divided power enveloping of $(\Z_p[\G],I)$ (proposition \ref{m-PD-structure}) denoted by $P_{(m)}(\G)$, is a free $\Z_p$-module with basis the elements $\underline{t}^{\{\underline{k}\}}=t_{1}^{\{k_1\}}\;.\;.\;.\;t_{l}^{\{k_l\}}$,
where $q_{i}!t_{i}^{\{k_i\}}=t_i^{k_i}$, for every $k_i=p^mq_i+r_i$ and $0\le r_i<p^m$. These algebras are endowed with a decreasing filtration by ideals $\overline{I}^{\{n\}}$ (subsection 2.1), such that $\overline{I}^{\{n\}}=\oplus_{|\underline{k}|\ge n}\Z_p\;\underline{t}^{\{\underline{k}\}}$.
The quotients $P^n_{(m)}(\G):=P_{(m)}(\G)/\overline{I}^{\{n+1\}}$ are therefore $\Z_p$-modules generated by the elements $\underline{t}^{\{\underline{k}\}}$ with $|\underline{k}|\le n$ \cite[Proposition 1.5.3 (ii)]{Berthelot1}. Moreover, there exists an isomorphism of $\Z_p$-modules 
\begin{eqnarray*}
P^n_{(m)}(\G)\simeq \displaystyle\bigoplus_{|\underline{k}|\le n} \Z_p\;\underline{t}^{\{\underline{k}\}}.
\end{eqnarray*}
Corollary \ref{order n mPD} gives us for any two integers $n$, $n'$ such that $n\le n'$ a canonical surjection $\pi^{n',n}:P^{n'}_{(m)}(\G)\rightarrow P^n_{(m)}(\G)$. Moreover, for every $m'\ge m$, the universal property of the divided powers gives us a unique morphism of filtered $\Z_p$-algebras $\psi_{m,m'}:P_{(m')}(\G)\rightarrow P_{(m)}(\G)$ which induces a homomorphism of $\Z_p$-algebras $\psi^{n}_{m,m'}:P^n_{(m')}(\G)\rightarrow P^n_{(m)}(\G)$. The module of distributions of level $m$ and order $n$ is $D_n^{(m)}(\G):=\text{Hom}(P^n_{(m)}(\G),\Z_p)$. \textit{The algebra of distributions of level m} is
\begin{eqnarray*}
D^{(m)}(\G):=\varinjlim_{n} D^{(m)}_{n}(\G),
\end{eqnarray*}
where the limit is formed respect to the maps  $\text{Hom}_{\Z_p}(\pi^{n',n},\Z_p)$. The multiplication is defined as follows. By universal property (Corollary \ref{order n mPD}) there exists a canonical application $\delta^{n,n'}:P^{n+n'}_{(m)}(\G)\rightarrow P^{n}_{(m)}\G)\otimes_{\Z_p}P^{n'}_{(m)}(\G)$. If $(u,v)\in D^{(m)}_n(\G)\times D^{(m)}_{n'}(\G)$, we define $u.v$ as the composition
\begin{eqnarray*}
u.v: P^{n+n'}_{(m)}(\G)\xrightarrow{\delta^{n,n'}} P^{n}_{(m)}(\G)\otimes_{\Z_p}P^{n'}_{(m)}(\G)\xrightarrow{u\otimes v}\Z_p.
\end{eqnarray*}
Let us denote by $\mathfrak{g}:=\text{Hom}_{\Z_p}(I/I^2,\Z_p)$
the Lie  algebra of $\G$. This is a free $\Z_p$-module with basis $\xi_1,\;.\;.\;.\;,\xi_l$  defined as the dual basis of the elements $t_1,\;.\;.\;.\;, t_l$. Moreover, if for every multi-index $\underline{k}\in\N^l$, $|\underline{k}|\le n$, we denote by $\underline{\xi}^{<\underline{k}>}$ the dual of the element $\underline{t}^{\{\underline{k}\}}\in P^n_{(m)}(\G)$, then $D^{(m)}_n(\G)$ is a free $\Z_p$-module of finite rank with a basis given by the elements $\underline{\xi}^{<\underline{k}>}$ with $|\underline{k}|\le n$ \cite[proposition 4.1.6]{HS1}. 

\begin{rem}\footnote{This remark exemplifies the local situation when $X=\text{Spec}(A)$ with $A$ a $\Z_{p}$-algebra \cite[Subsection 1.3.1]{Huyghe2}.}\label{Sym-Locally}
Let A be an $\Z_p$-algebra and $E$ a free $A$-module of finite rank with base  $(x_1,...,x_N)$. Let $(y_1, ... ,y_{N})$ be the dual base of  $E^{\vee}:=\text{Hom}_{A}(E,A)$. As in the preceding subsection, let $\textbf{S}(E^{\vee})$ be the symmetric algebra and $\textbf{I}(E^{\vee})$ the augmentation ideal. Let $\Gamma_{(m)}(E^{\vee})$  be the $m$-PD-envelope of $(\textbf{S}(E^{\vee}),\;\textbf{I}(E^{\vee}))$.  We put $\Gamma^n_{(m)}(E^{\vee}):=\Gamma_{(m)}(E^{\vee})/\overline{I}^{\{n+1\}}$. These are free $A$-modules with base $y_1^{\{k_1\}}...y_N^{\{k_N\}}$ with $\sum k_i\le n$ \cite[1.1.2]{Huyghe1}. Let $\{\underline{x}^{<\underline{k}>}\}_{|\underline{k}|\le n}$ be the dual base of Hom$_A(\Gamma_{(m)}^n(E^{\vee}),A)$. We define
\begin{eqnarray*}
Sym^{(m)}(E):=\displaystyle\bigcup_{n\in\N}\text{Hom}_{A}\left(\Gamma^{n}_{(m)}(E^{\vee}),A\right).
\end{eqnarray*}
This is a free $A$-module with a base given by all the $\underline{x}^{<\underline{k}>}$. The inclusion $\text{Sym}^{(m)}(E)\subseteq \text{Sym}^{(m)}(E)\otimes_{\Z_p}\Q_p$ gives  the relation
\begin{eqnarray}\label{Rel_base_change}
x_i^{<k_i>} = \displaystyle\frac{k_i !}{q_i !} x^{k_i}.
\end{eqnarray}
Moreover, it also has a structure of algebra defined as follows. By \cite[Proposition 1.3.1]{Huyghe1} there exists an application $\Delta_{n,n'}:\Gamma_{(m)}^{n+n'}(E^{\vee})\rightarrow\Gamma^{n}_{(m)}(E^{\vee})\otimes_{A}\Gamma^{n'}_{(m)}(E^{\vee})$, which allows to define the product of $u\in Hom_{A}(\Gamma^{n}_{(m)}(E^{\vee}),A)$ and $v\in Hom_{A}(\Gamma^{n'}_{(m)}(E^{\vee}),A)$  by the composition 
\begin{eqnarray*}
u.v: \Gamma^{n+n'}_{(m)}(E^{\vee})\xrightarrow{\Delta_{n,n'}}\Gamma^{n}_{(m)}(E^{\vee})\otimes_{A}\Gamma^{n'}_{(m)}(E^{\vee})\xrightarrow{u\otimes v} A.
\end{eqnarray*}
This maps endows $Sym^{(m)}(E)$  of a structure of a graded noetherian $\Z_p$-algebra \cite[Propositions 1.3.1, 1.3.3 and 1.3.6]{Huyghe1}.
\end{rem}
\justify
We have the following important properties \cite[Proposition 4.1.15]{HS1}.

\begin{prop}\label{gr.Dist algebra}
\begin{itemize}
\item[(i)] There exists a canonical isomorphism of graded $\Z_p$-algebras $gr_{\bullet}(D^{(m)}(\G))\simeq Sym^{(m)}(\mathfrak{g})$.
\item[(ii)] The $\Z_p$-algebras $gr_{\bullet}(D^{(m)}(\G))$ and $D^{(m)}(\G)$ are noetherian. 
\end{itemize}
\end{prop}

\subsection{Integral models}
We start this subsection with the following definition \cite[subsection 3.4]{Berthelot1}.
\begin{defi}\label{defi I.models}
Let $A$ be a $\Q_p$-algebra (resp. a sheaf of $\Q_p$-algebras). We say that a $\Z_p$-subalgebra $A_0$ (resp. a subsheaf of $\Z_p$-algebras) is an integral model of $A$ if $A_0\otimes_{\Z_p}\Q_p= A$.
\end{defi} 
\justify
Let us recall that throughout this paper $\mathfrak{g}$ denotes the Lie algebra of a connected reductive group scheme $\G$ and $\Ub(\mathfrak{g})$ its universal enveloping algebra. As we have remarked in the introduction, if $\mathfrak{g}_\Q$ denotes the Lie algebra of the algebraic group  $\G_\Q=\G\times_{\text{Spec}(\Z_p)}\text{Spec}(\Q_p)$ and $\Ub(\mathfrak{g}_\Q)$ its universal enveloping algebra, then it is known that $\Ub(\mathfrak{g})$ is an integral model of $\Ub(\mathfrak{g}_\Q)$.  Moreover, the algebra of distributions of level $m$, introduced in the preceding subsection, is also an integral model of $\Ub(\mathfrak{g}_{\Q})$ \cite[subsection 4.1]{HS1}. This latest example will be specially important in this work.
\justify
In the following discussion we will assume that $X$ is a smooth $\Z_p$-scheme endowed with a right $\G$-action.

\begin{prop}\label{morp.HS}
The right $\G$-action induces a canonical homomorphism of filtered $\Z_p$-algebras 
\begin{eqnarray*}
\Phi^{(m)}: D^{(m)}(\G)\rightarrow H^{0}(X,\Db^{(m)}_{X}).
\end{eqnarray*}
\end{prop}
\begin{proof}
The reader can find the proof of this proposition in \cite[Proposition 4.4.1 (ii)]{HS1}, we will briefly discuss the construction of $\Phi^{(m)}$.  The central idea in the construction is that if $\rho:X\times_{\Z_p}\G\rightarrow X$ denotes the action, then the comorphism $\rho^{\natural}:\Ob_{X}\rightarrow\Ob_{X}\otimes_{\Z_p}\Z_p[\G]$ induces a morphism 
\begin{eqnarray*}
\rho^{(n)}_{m}:\Pb^{n}_{X,(m)}\rightarrow\Ob_{X}\otimes_{\Z_p}P^{n}_{(m)}(\G)
\end{eqnarray*}
for every $n\in\N$. Those applications are compatible when varying $n$. Let $u\in D^{(m)}_n(\G)$ we define $\Phi^{(m)}(u)$ by 
\begin{eqnarray*}
\Phi^{(m)}(u):\Pb^{n}_{X,(m)}\xrightarrow{\rho^{(n)}_m} \Ob_{X}\otimes_{\Z_p}P^{n}_{(m)}(\G)\xrightarrow{id\otimes u}\Ob_{X}.
\end{eqnarray*}
Again, those applications are compatible when varying $n$ and we get the morphism of the proposition.
\end{proof}
\begin{rem}  \label{Invarian global sections of group}
\begin{itemize}
\item[(i)] If $X$ is endowed with a left $\G$-action, then it turns out that $\Phi^{(m)}$ is an anti-homomorphism.
\item[(ii)] In \cite[Theorem 4.4.8.3]{HS1} Huyghe and Schmidt have shown that if $X=\G$ and we consider the right (resp. left) regular action, then the morphism of the preceding proposition is in fact a canonical filtered isomorphism (resp. an anti-isomorphism) between $D^{(m)}(\G)$ and  $H^0(\G, \Db^{(m)}_{\G})^{\G}$, the $\Z_p$-submodule of (left) $\G$-invariant global sections (cf. definition \ref{T-invariants}). This isomorphism induces a bijection between $D_n^{(m)}(\G)$ and $H^{0}(\G,\Db^{(m)}_{\G,n})^{\mathbb{G}}$, and it is compatible when varying $m$.  
\end{itemize}
\end{rem}
\justify
Let us define $\Ab_{X}^{(m)}:=\Ob_{X}\otimes_{\Z_p}D^{(m)}(\G)$, and let us remark that we can endow this sheaf with the skew ring multiplication (see (\ref{smash}) below) coming from the action of $D^{(m)}(\G)$ on $\Ob_{X}$ via the morphism $\Phi^{(m)}$ of proposition \ref{morp.HS}.  This multiplication defines over $\Ab_X^{(m)}$ a structure of a sheaf of associative $\Z_p$-algebras, such that it becomes an integral model of the sheaf of $\Q_p$-algebras $\Ub^\circ:=\Ob_{X_\Q}\otimes_{\Q_p}\Ub(\mathfrak{g}_\Q)$. To see this, let us recall how the multiplicative structure of the sheaf $\Ub^\circ$ is defined (cf. \cite[subsection 5.1]{PSS1} or \cite[section 2]{Milicic1}).\\
Differentiating the right action of $\G_\Q$ on $X_\Q$ we get a morphism of $\Q_p$-Lie algebras
\begin{eqnarray}\label{Action_Lie}
\tau:\ \mathfrak{g}_\Q\rightarrow H^0(X_\Q,\Tb_{X_\Q}).
\end{eqnarray}
This implies that $\mathfrak{g}_{\Q}$ acts on $\Ob_{X_{\Q}}$ by derivations and we can endow $\Ub^{\circ}$ with the skew ring multiplication
\begin{eqnarray}\label{smash}
(f\otimes\eta)(g\otimes\zeta)=f\tau(\eta)g\otimes\zeta + fg\otimes\eta\zeta
\end{eqnarray}
for $\eta\in\mathfrak{g}_L$, $\zeta\in\Ub(\mathfrak{g}_\Q)$ and $f,g\in\Ob_{X_\Q}$. With this product, the sheaf $\Ub^{\circ}$ becomes a sheaf of associative algebras and given that $D^{(m)}(\G)$ is an integral model of the universal enveloping algebra $\Ub(\mathfrak{g}_\Q)$, then $\Ab_{X}^{(m)}$ is also a sheaf of associative $\Z_p$-algebras being a subsheaf of $\Ub^{\circ}$. We have the following result from \cite[Corollary 4.4.6]{HS1}.

\begin{prop}\label{prop 1.4.1} 
\begin{itemize}
\item[(i)]  The sheaf $\Ab_{X}^{(m)}$ is a locally free $\Ob_X$-module.
\item[(ii)]  There exist a unique structure over $\Ab_{X}^{(m)}$ of filtered $\Ob_{X}$-ring and a canonical isomorphism of graded $\Ob_{X}$-algebras $gr (\Ab_{X}^{(m)})= \Ob_{X}\otimes_{\Z_p}\text{Sym}^{(m)}(\mathfrak{g})$.
\item[(iii)]  The sheaf $\Ab_{X}^{(m)}$ (resp. $gr(\Ab_{X}^{(m)})$) is a coherent sheaf of $\Ob_{X}$-rings (resp. a coherent sheaf of $\Ob_{X}$-algebras), with noetherian sections over open affine subsets of $X$. 
\end{itemize}
\end{prop}

\begin{rem}
If we take the tensor product with $\Q_p$ in proposition \ref{morp.HS}, then by functoriality we get the following commutative diagram
\begin{eqnarray*}
\begin{tikzcd}
D^{(m)}(\G) \arrow[r,"\Phi^{(m)}"] \arrow[d,hook] & H^{0}(X,\Db^{(m)}_X)\arrow[d, hook]\\
\Ub(\mathfrak{g}_\Q) \arrow[r,"\Psi_{X_\Q}"]  & H^{0}(X_\Q,\Db_{X_\Q})
\end{tikzcd}
\end{eqnarray*}
where $\Psi_{X_\Q}$ is the morphism induced by (\ref{Action_Lie}) and it is called the operator-representation \cite{BB2}.
\end{rem}

\section{Integral twisted arithmetic differential operators} 
\subsection{Torsors} Let us suppose that $\mathbb{T}$ is a smooth affine algebraic group over $\Z_p$ with Lie algebra denoted by $\mathfrak{t}$, and that $\widetilde{X}$ and $X$ are smooth separated schemes over $\Z_p$, such that $\widetilde{X}$  is endowed with a right $\mathbb{T}$-action $\sigma: \widetilde{X}\times_{\text{Spec}(\Z_p)}\mathbb{T}\rightarrow \widetilde{X}$. We will also assume that $\T$ acts trivially on $X$. \footnote{For example if $\T\subset \B$ and $X=\G/\B$ is the flag variety.}
\justify
We say that a morphism $\xi:\widetilde{X}\rightarrow X$ is a $\mathbb{T}$-\textit{torsor} for the Zariski topology, if $\xi$ is a faithfully flat morphism such that the diagram
\begin{eqnarray*}
\begin{tikzcd}
\widetilde{X}\times_{\text{Spec}(\mathfrak{o)}}\T \arrow[r,"\sigma "] \arrow[d, "p_1"]
& \widetilde{X} \arrow[d, "\xi "]\\
\widetilde{X} \arrow[r, "\xi "]
& X
\end{tikzcd}
\end{eqnarray*}
is commutative and the map (induced by the previous diagram)
\begin{eqnarray}\label{Cartesian_diag}
\widetilde{X}\times_{\text{Spec}(\mathfrak{o})}\mathbb{T}\rightarrow\widetilde{X}\times_{X}\widetilde{X};\;\;\; (x,h)\mapsto (x,xh)
\end{eqnarray}
is an isomorphism. 
\justify
Let $U\subset X$ be an affine open subset and  $\text{pr}_{1}:U\times_{\text{Spec}(\Z_p)}\T\rightarrow U$ the first projection. We say that $U$ \textit{trivializes the torsor} $\xi$ if there is a $\mathbb{T}$-equivariant isomorphism $\alpha_U:U\times_{\text{Spec}(\Z_p)}\mathbb{T}\xrightarrow{\simeq}\xi^{-1}(U)$, where $\mathbb{T}$ acts on $U\times_{\text{Spec}(\Z_p)}\mathbb{T}$ by right translations on the second factor, and such that $\text{pr}_1=\xi|_{\xi^{-1}(U)}\;\circ\; \alpha_U$.

\begin{rem}\label{rem 5.1}
As $X$ is separated, the set $\Sb$ of open affine subschemes $U$ of $X$ that trivialises the torsor and such that $\Ob_X(U)$ is a finitely generated $\Z_p$-algebra, it is stable under intersections. Moreover, if $U\in \Sb$ and $W$ is an open affine subscheme of $U$, then $W\in\Sb$.
\end{rem}

\begin{defi} \label{locally trivial}
We say that $\xi$ is locally trivially for the Zariski topology if $X$ can be covered by opens in $\Sb$.  
\end{defi}
 
\begin{lem}\label{triviality}
Let $\xi:\widetilde{X}\rightarrow X$ be a locally trivial $\T$-torsor and let $\Mb$ be a quasi-coherent $\Ob_{\widetilde{X}}$-module. Then $\text{R}^{1}\xi_{*}\Mb=0$.
\end{lem}
\begin{proof}
We recall for the reader that R$^{1}\xi_{*}\Mb$ is the sheaf associated to the presheaf \cite[chapter III, prop. 8.1]{Hartshorne1}
\begin{eqnarray*}
U\mapsto H^{1}(\xi^{-1}(U),\Mb).
\end{eqnarray*}
As $\xi$ is locally trivial, the set $\Sb$ of affine open subsets of $X$ that trivializes the torsor is a base for the Zariski topology of $X$. Moreover, if $U\in\Sb$ then by definition $\xi^{-1}(U)$ is an affine open subset of $X$ and given that $\Mb$ is a quasi-coherent $\Ob_{\widetilde{X}}$-module, we can conclude that $H^{1}(\xi^{-1}(U),\Mb)=0$.
\end{proof}

\subsection{$\T$-equivariant sheaves and sheaves of $\T$-invariant sections}
Let us denote by $m:\T\times_{\text{Spec}(\Z_p)}\T\rightarrow\T$ the group law of $\T$ and by 
\begin{eqnarray*}
p_1:\widetilde{X}\times_{\text{Spec}(\Z_p)}\T\rightarrow \widetilde{X}\;\;\text{and}\;\; p_{1,2}:\widetilde{X}\times_{\text{Spec}(\Z_p)}\T\times_{\text{Spec}(\Z_p)}\T\rightarrow\widetilde{X}\times_{\text{Spec}(\Z_p)}\T
\end{eqnarray*}
the respective projections. We will also denote by
\begin{eqnarray*}
f_1,\; f_2,\; f_3:\widetilde{X}\times_{\text{Spec}(\Z_p)}\T\times_{\text{Spec}(\Z_p)}\T\rightarrow\widetilde{X}
\end{eqnarray*}
the morphisms defined by $f_1(x,t_1,t_2)=x$, $f_2(x,t_1,t_2)= xt_1$ and $f_3(x,t_1,t_2)= xt_1t_2$. Following \cite[chapter 0, section 3]{MF}, we say that a sheaf $\Mb$ of $\Ob_{\widetilde{X}}$-modules is $\T$-\textit{equivariant} if there exists an isomorphism 
\begin{eqnarray} \label{EquivarianceDatum}
p_{1}^{*}\Mb\xrightarrow{\Psi} \sigma^{*}\Mb
\end{eqnarray}
such that the following diagram is commutative (cocycle condition \cite[(9.10.10)]{HTT})

\begin{eqnarray}\label{CocycleCondition}
\begin{tikzcd} [row sep=huge]
(id_{\widetilde{X}}\times m)^{*}p_1^*\Mb=f_1^*\Mb=p_{1,2}^*p_1^*\Mb \arrow[rd, "(id_{\widetilde{X}}\times m)^*\Psi"] \arrow[r, "p_{1,2}^{*}\Psi"] & p_{1,2}^*\sigma^*\Mb=f_2^*\Mb=(\sigma\times id_{\T})^{*}p_1^{*}\Mb \arrow[d, "(\sigma\times id_{\T})^{*}\Psi"] \\
& (id_{\widetilde{X}}\times m)^*\sigma^*\Mb = f_3^*\Mb=(\sigma\times id_{\T})^*\sigma^{*}\Mb.
\end{tikzcd}
\end{eqnarray}
We will need the following lemmas. First of all, we recall for the reader  that the category of $\T$-equivariant quasi-coherent $\Ob_{\widetilde{X}}$-modules is an abelian category \cite[Lemma 2.17]{Hashimoto}. In particular it is complete and cocomplete.

\begin{lem}\label{Eq.Filtration}
Let $\Mb$ be an $\Ob_{\widetilde{X}}$-module filtered by a family $(\Mb_n)_{n\in\N}$ of $\T$-equivariant $\Ob_{\widetilde{X}}$-modules such that the inclusions
 $\Mb_n\hookrightarrow\Mb_{n+1}$ are $\T$-equivariant. Then $\Mb$ is also $\T$-equivariant.
\end{lem}
\justify
In fact, under the hypothesis on the lemma, the exactness of  the functors $\sigma^{*}$ and $p_1^*$ allows us to conclude that $\sigma^{*}(\Mb)$ and $p_1^*(\Mb)$ are endowed with canonical filtrations $(\sigma^{*}(\Mb_n))_{n\in\N}$ and $(p_1^{*}(\Mb))_{n\in\N}$, respectively. Since the components of this filtration have compatible $\T$-equivariant structures we can conclude that $\Mb$ is also $\T$-equivariant via a filtered isomorphism.
  
\begin{lem}\label{Eq.Dual}
Let $(\Lb,\Psi)$ be a $\T$-equivariant locally free $\Ob_{\widetilde{X}}$-module of finite rank. Then its dual $\Lb^{\vee}$ is also $\T$-equivariant.
\end{lem}
\begin{proof}
Given that $\Lb$ is a locally free sheaf of finite rank, we dispose of a canonical and functorial isomorphism
\begin{eqnarray*}
\sigma^{*}\Lb^{\vee}\xrightarrow{\simeq} \mathcal{H}om_{\Ob_{\widetilde{X}\times_{\Z_p}\T}}(\sigma^{*}\Lb,\Ob_{\widetilde{X}\times_{\Z_p}\T}).
\end{eqnarray*}
This implies that $(\Psi^{-1})^{\vee}$, the dual of $\Psi^{-1}$, defines the $\T$-equivariant structure on $\Lb^{\vee}$. 
\end{proof}
\justify
Since the category of $\T$-equivariant quasi-coherent $\Ob_{\widetilde{X}}$-modules is an abelian category, the cokernel of a $\T$-equivariant morphism $\Mb\rightarrow\Nb$ between two $\T$-equivariant quasi-coherent $\Ob_{\widetilde{X}}$-modules is again a  $\T$-equivariant quasi-coherent $\Ob_{\widetilde{X}}$-module. As we will demonstrate below, we can give an independent proof of this result.

\begin{lem}\label{Eq.sequence}
Let $(\Mb,\Phi_1)$ and $(\Nb,\Phi_2)$ be $\T$-equivariant quasi-coherent $\Ob_{\widetilde{X}}$-modules. Let $\Lb$ be a quasi-coherent $\Ob_{\widetilde{X}}$-module such that 
\begin{eqnarray*}
0\rightarrow \Mb\xrightarrow{\phi}\Nb\xrightarrow{\psi}\Lb\rightarrow 0
\end{eqnarray*}
is an exact sequence which induces a commutative diagram
\begin{eqnarray}\label{Eq_3_factor}
\begin{tikzcd}
0 \arrow[r] 
& p_1^*(\Mb) \arrow[r, "p_1^*(\phi)"] \arrow[d, "\Phi_1"] 
& p_1^*(\Nb) \arrow[r, "p_1^*(\psi)"] \arrow[d,"\Phi_2"]
& p_1^*(\Lb) \arrow[r] 
& 0 \\
0 \arrow[r]
& \sigma^{*}(\Mb) \arrow[r, "\sigma^*(\phi)"]
& \sigma^{*}(\Nb) \arrow[r, "\sigma^*(\psi)"]
& \sigma^{*}(\Lb) \arrow[r]
& 0.
\end{tikzcd}
\end{eqnarray}
Then $\Lb$ is also $\T$-equivariant.
\end{lem}
\begin{proof}
To define the right vertical morphism, we consider a basis $\Aa$ of $\widetilde{X}$ consisting of affine open subsets, which can be assumed stable under intersection because $\widetilde{X}$ is separated. Let us fix $U\in\Aa$. As (\ref{Eq_3_factor}) is a diagram of quasi-coherent $\Ob_{\tilde{X}\times_{\Z_p}\T}$-modules, the horizontal short exact sequences remains exact when we take local sections on $U\times_{\Z_p}\T$. To soft the notation we will assume that $M(\sigma):=\sigma^{*}(\Mb)(U\times_{\Z_p}\T)$ ($N(\sigma):=\sigma^*(\Nb)(U\times_{\Z_p}\T)$ and $L(\sigma):=\sigma^*(\Lb)(U\times_{\Z_p}\T)$) and $M(p_1):=p_1^{*}(\Mb)(U\times_{\Z_p}\T)$ ($N(p_1):=p_1^*(\Nb)(U\times_{\Z_p}\T)$ and $L(p_1):=p_1^*(\Lb)(U\times_{\Z_p}\T)$). Also, we will suppose that $\phi_1:=p_1^{*}(\phi)(U\times_{\Z_p}\T)$ (resp. $\psi_{1}:=p_1^{*}(\psi)(U\times_{\Z_p}\T)$), $\phi_2:=\sigma^*(\phi)(U\times_{\Z_p}\T)$ (resp. $\psi_2:=\sigma^*(\psi)(U\times_{\Z_p}\T)$), $\Phi_{1,U}:=\Phi_{1}(U\times_{\Z_p}\T)$ and $\Phi_{2,U}:=\Phi_{2}(U\times_{\Z_p}\T)$. In such a way that we have the following commutative diagram
 
\begin{eqnarray*}
\begin{tikzcd}
0 \arrow[r]
              & M(p_1) \arrow[r,"\phi_1"] \arrow[d,"\Phi_{1,U}"]
              & N(p_1) \arrow[r,"\psi_1"] \arrow[d,"\Phi_{2,U}"]
              & L(p_1) \arrow[r]
              & 0\\
0 \arrow[r]
              & M(\sigma) \arrow[r,"\phi_2"] 
              & N(\sigma) \arrow[r,"\psi_2"]
              & L(\sigma) \arrow[r]
              & 0              
\end{tikzcd}
\end{eqnarray*}
Let $x\in L(p_1)$. By surjectivity of $\psi_1$ we can find $y_1\in N(p_1)$ such that $\psi_1(y_1)=x$. We define then $\Phi_{U}(x):= \psi_2(\Phi_{2,U}(y_1)) \in L(\sigma)$. Let us see that $\Phi_{U}$ is well-defined, this means that it does not depend of the choice of $y_1\in N(p_1)$. Let $y_2\in N(p_1)$ such that $\psi_1(y_2)=x$. We want to see $\Phi_{U}(x):= \psi_2(\Phi_{2,U}(y_1))= \psi_{2}(\Phi_{2,U}(y_2))$. Let $y:=y_1-y_2 \in N(p_1)$. By definition $y\in\text{Ker}(\psi_1)=\text{Im}(\phi_1)$, therefore there exists $z\in M(p_1)$ such that $\phi_{1}(z)=y$. Let $z'=\Phi_{1,U}(z)\in M(\sigma)$. By commutativity of the diagram we have 
\begin{eqnarray*}
\phi_2(z')=\phi_2(\Phi_{1,U}(z))= \Phi_{2,U}(\phi_1(z))= \Phi_{2,U}(y) = \Phi_{2,U}(y_1)-\Phi_{2,U}(y_2)
\end{eqnarray*}
and therefore 
\begin{eqnarray*}
0=\psi_{2}(\phi_{2}(z'))=\psi_2(\Phi_{2,U}(y_1))-\psi_2(\Phi_{2,U}(y_2))= \Phi_{U}(x)-\psi_2(\Phi_{2,U}(y_2)).
\end{eqnarray*}
Moreover, it is straightforward to see that $\Phi_U$ is in fact a morphism of $\Ob_{\tilde{X}\times_{\Z_p}\T}(U\times_{\Z_p}\T)$-modules, which by the well-known \textit{five lemma} becomes an isomorphism. From the preceding reasoning, and the quasi-coherence of the sheaves involved, for every $U\in\Aa$ we get an isomorphism  $\tilde{\Phi}_{U}:p_1^*(\Nb)|_{U\times\T}\rightarrow\sigma^{*}(\Lb)|_{U\times\T}$ of sheaves of $\Ob_{U\times\T}$-modules.\\
To complete the construction of the right vertical isomorphism we need to globalize the preceding reasoning. This means that if we chose $U,V\in\Aa$ then we have $\tilde{\Phi}_{U}|_{U\cap V}=\tilde{\Phi}_{V}|_{U\cap V}$. As $\Aa$ is stable under intersections we can construct, in the same way as before, an isomorphism $\tilde{\Phi}_{U\cap V}$ over $U\cap V$. Let us see that $\tilde{\Phi}_{U}|_{U\cap V}=\tilde{\Phi}_{U\cap V}$ (the reader can fallow the same reasoning to proof that $\tilde{\Phi}_{V}|_{U\cap V}=\tilde{\Phi}_{U\cap V}$). We consider the following cube 
\begin{eqnarray*}
\begin{tikzcd}[row sep=scriptsize, column sep=scriptsize]
& p_1^{*}(\Nb)(U\times_{\Z_p}\T)  \arrow[dl, "res"] \arrow[rr, "p_1^*(\psi)(U\times\T)"] \arrow[dd, pos=0.7, "\Phi_{2}(U\times\T)"] & & p_1^*(\Lb)(U\times_{\Z_p}\T) \arrow[dl, "res"] \arrow[dd, "\Phi_{U}"] \\
p_1^{*}(\Nb)(U\cap V\times_{\Z_p}\T) \arrow[rr, crossing over, "p_1^*(\psi)(U\cap V\times\T)", near end] \arrow[dd, "\Phi_2(U\cap V)"] & & p_1^{*}(\Lb)(U\cap V\times_{\Z_p}\T) \\
& \sigma^{*}(\Nb)(U\times_{\Z_p}\T) \arrow[dl, "res"] \arrow[rr, "\sigma^{*}(\psi)(U\times\T)" near start] & & \sigma^*(\Lb)(U\times_{\Z_p}\T) \arrow[dl, "res"] \\
\sigma^*(\Nb)(U\cap V\times_{\Z_p}\T) \arrow[rr, "\sigma^*(\psi)(U\cap V\times\T)"] & & \sigma^*(\Lb)(U\cap V\times_{\Z_p}\T) \arrow[from=uu, crossing over, "\Phi_{U\cap V}", pos=0.7]\\
\end{tikzcd}
\end{eqnarray*}
Except for the right lateral face, all the other faces form, by construction, commutative diagrams which implies that also the right lateral face forms a commutative diagram. This shows that $\tilde{\Phi}_{U}|_{U\cap V}=\tilde{\Phi}_{U\cap V}$. We have constructed an isomorphism $\Phi:p_1^*(\Lb)\rightarrow\sigma^{*}(\Lb)$ of quasi-coherent $\Ob_{\tilde{X}\times_{\Z_p}\T}$-modules. Let us show that $\Psi$ defines a $\T$-equivariant structure. To do that we consider the following diagram   
\begin{eqnarray*}
\begin{tikzcd}[row sep=huge, column sep=huge]
  & (id_{\widetilde{X}}\times m)^{*}p_1^*\Nb \arrow[r] \arrow[d, "(id_{\widetilde{X}}\times m)^*p_1^*\psi"] \arrow[dl] & (\sigma\times id_{\T})^{*}p_1^{*}\Nb \arrow [dll, crossing over] \arrow[d] \\
 (id_{\widetilde{X}}\times m)^*\sigma^*\Nb = (\sigma\times id_{\T})^*\sigma^{*}\Nb \arrow[d] &  (id_{\widetilde{X}}\times m)^{*}p_1^*\Lb \arrow[r, "p_{1,2}^*\Phi"] \arrow[dl, "(id_{\widetilde{X}}\times m)^*\Phi"] &  (\sigma\times id_{\T})^{*}p_1^{*}\Lb \arrow[dll, "(\sigma\times id_{\T})^*\Phi", shift left=0.9ex] \\
    (id_{\widetilde{X}}\times m)^*\sigma^*\Lb = (\sigma\times id_{\T})^*\sigma^{*}\Lb
\end{tikzcd}
\end{eqnarray*}
The triangle on the top is commutative and by functoriality the lateral faces of the prism are also commutative. Its is straightforward to show from this that 
\begin{eqnarray*}
(id_{\widetilde{X}}\times m)^*\Phi\; \circ\; (id_{\widetilde{X}}\times m)^*p_1^*\psi= (\sigma\times id_{\T})^*\Phi\; \circ\;  p_{1,2}^*\Phi\; \circ\; (id_{\widetilde{X}}\times m)^*p_1^*\psi,
\end{eqnarray*}
and as $(id_{\widetilde{X}}\times m)^*p_1^*\psi$ is surjective we can conclude that the triangle on the bottom is also commutative. Therefore $\Psi$ defines a $\T$-equivariant structure for $\Lb$. 
\end{proof}
\justify
Let $(\Mb,\Psi)$ be a $\T$-equivariant $\Ob_{\widetilde{X}}$-module. By the Künneth formula \cite[Theorem 6.7.8]{Grothendieck3} we have a canonical isomorphism
\begin{eqnarray*}
H^0(\widetilde{X}\times_{\Z_p}\mathbb{T},p_1^*\Mb)\simeq H^{0}(\widetilde{X},\Mb)\otimes_{\Z_p}\Z_p[\mathbb{T}]
\end{eqnarray*}
which composing with the application 
\begin{eqnarray*}
H^{0}(\widetilde{X},\Mb)\longrightarrow H^0(\widetilde{X}\times_{\Z_p}\mathbb{T},\sigma^*\Mb)\xrightarrow{H^0(\Psi)} H^{0}(\widetilde{X}\times_{\Z_p}\mathbb{T},p_1^*\Mb)
\end{eqnarray*}
(the first application is induced via the canonical map $\Mb\rightarrow\sigma_*\sigma^*\Mb$) gives us a morphism 
\begin{eqnarray*}
\Delta: H^0(\widetilde{X}, \Mb)\rightarrow H^{0}(\widetilde{X},\Mb)\otimes_{\Z_p}\Z_p[\mathbb{T}],
\end{eqnarray*}
defining a structure of $\mathbb{T}$-module on $H^{0}(\widetilde{X},\Mb)$. The co-module relations are given by the cocycle condition \cite[chapter 0, definition 1.6]{MF}.
\begin{defi}\label{T-invariants}
The $\mathbb{T}$-invariant elements of $H^0(\widetilde{X},\Mb)$ are the elements $P\in H^{0}(\widetilde{X},\Mb)$ such that $\Delta(P)=P\otimes 1$. This subspace will be denoted by $H^0(\widetilde{X},\Mb)^{\mathbb{T}}$. 
\end{defi}
\justify
Now, let us suppose that $\widetilde{X}$ can be covered by a family $\widetilde{\Sb}$ of affine open subsets, which are stable under finite intersection and invariant under the right action of $\T$. This means that for every $\tilde{U}\in\widetilde{\Sb}$ the morphism $\sigma$, inducing the right $\T$-action on $\widetilde{X}$, induces a right $\T$-action  $\tilde{\sigma}:=\sigma|_{\tilde{U}\times_{\Z_p}\T}: \tilde{U}\times_{\Z_p}\T\rightarrow \tilde{U}$ on $\tilde{U}$. By pulling back $\Psi$ under the inclusion $\tilde{U}\times_{\Z_p}\T \hookrightarrow \widetilde{X}\times_{\Z_p}\T$ we get an isomorphism $\Psi|_{\tilde{U}}:\tilde{\sigma}^*\Mb_{\tilde{U}}\rightarrow p_1^{*}\Mb_{\tilde{U}}$ which satisfies the respective cocycle condition (\ref{CocycleCondition}), and, as before, we obtain a comodule map
\begin{eqnarray*}
\Delta_{\tilde{U}}: \Gamma(\tilde{U},\Mb)\rightarrow \Gamma(\tilde{U},\Mb)\otimes_{\Z_p}\Z_p[\T].
\end{eqnarray*}
As in definition \ref{T-invariants}, we can define the $\Z_p$-submodule of $\T$-\textit{invariant sections} on $\tilde{U}$ by 
\begin{eqnarray}\label{T-inv on an open}
\Gamma(\tilde{U},\Mb)^{\T}=\left\{P\in \Gamma(\tilde{U},\Mb)|\;\; \Delta_{\tilde{U}}(P)=P\otimes 1\right\}.
\end{eqnarray}
Finally, let us suppose that $\Mb$ is also quasi-coherent. By \cite[Chapter II, Corollary 5.5]{Hartshorne1} on every affine open subset $\tilde{U}\in\widetilde{\Sb}$ we can define a subsheaf 
\begin{eqnarray*}
\left(\Mb|_{\tilde{U}}\right)^{\T}:= \widetilde{\Gamma(\tilde{U},\Mb)^{\T}}\subset \widetilde{\Gamma(\tilde{U},\Mb)}=\Mb|_{\tilde{U}}.
\end{eqnarray*}
By definition, and giving that $\widetilde{\Sb}$ was supposed to be stable under finite intersections, the preceding sheaves glue together to a subsheaf $\left(\Mb\right)^{\T}\subset \Mb$ which does not depend on the covering $\widetilde{\Sb}$. We sum up the preceding construction in the next definition. 
\begin{defi}\label{Sheaf of T-inv sections}
Let $\widetilde{X}$ be a smooth separated $\Z_p$-scheme endowed with a right $\T$-action, and covered by a family of affine open subsets $\widetilde{\Sb}$ stable under finite intersections and the $\T$-action. For every $\T$-equivariant quasi-coherent $\Ob_{\widetilde{X}}$-module $\Mb$, the subsheaf $(\Mb)^{\T}$ is called the subsheaf of $\T$-invariant sections of $\Mb$. 
\end{defi}
\justify
As an application of the preceding construction let us point out that if $\xi:\widetilde{X}\rightarrow X$ is a locally trivial $\T$-torsor, then we actually dispose of a subsheaf of $\T$-invariant sections of the direct image sheaf $\xi_{*}\Mb$, with $\Mb$ a $\T$-equivariant quasi-coherent $\Ob_{\widetilde{X}}$-module. In fact, if $\Sb$ denotes the collection of all affine open subsets that trivialises the torsor $\xi$, then for every $U\in\Sb$ we know that $\xi^{-1}(U)$ is stable under the right $\T$-action and, as in (\ref{T-inv on an open}), we can define 
\begin{eqnarray}\label{T-inv on an open for xi}
\left((\xi_*\Mb)(U)\right)^{\T}=\left(\Mb(\xi^{-1}(U))\right)^{\T}\subset \Mb(\xi^{-1}(U)).
\end{eqnarray}
As $\widetilde{X}$ is noetherian $\xi_*\Mb$ is quasi-coherent and therefore from (\ref{T-inv on an open for xi}) we have a subsheaf
\begin{eqnarray*}
\left((\xi_*\Mb)|_{U}\right)^{\T}:= \widetilde{\left((\xi_*\Mb)(U)\right)^{\T}} \subset \widetilde{\left(\xi_*\Mb\right)(U)} =\left(\xi_*\Mb\right)|_{U}.
\end{eqnarray*} 
Since $\Sb$ is stable under finite intersections, those sheaves glue together to define a subsheaf 
\begin{eqnarray}\label{T-invariants for  xi}
 \left(\xi_*\Mb\right)^{\T}\subset \xi_*\Mb.
\end{eqnarray}
\justify
For the rest of this subsection we will always suppose that $\xi:\widetilde{X}\rightarrow X$ is a locally trivial $\T$-torsor.

\begin{lem}\label{xi} The morphism $\xi:\widetilde{X}\rightarrow X$ induces an isomorphism $\xi^{\natural}:\Ob_{X}\rightarrow \left(\xi_*\Ob_{\widetilde{X}}\right)^{\T}$.
\end{lem}
\begin{proof}
As this is a local problem, we can take $U\in\Sb$ and suppose that $\xi: \xi^{-1}(U)=U\times_{\text{Spec}(\Z_p)}\T\rightarrow U$ is the first projection. Since rational cohomolgy  commutes with direct limits \cite[Part I, Lemma 4.17]{Jantzen} and $\Ob_{X}(U)$ is a direct limit of free $\Z_p$-modules, we can conclude that $\left(\xi_{*}\Ob_{\widetilde{X}}\right)^{\T}(U)=\left(\Ob_{X}(U)\otimes_{\Z_p}\Z_p[\T]\right)^{\T}=\Ob_{X}(U)$.
\end{proof}
\subsection{Relative enveloping algebras of finite level}
Let us fix a positive integer $m\in\Z$. As in the preceding subsections $\widetilde{X}$ and $X$ will denote smooth separated $\Z_p$-schemes, and $\T$ a smooth affine commutative group scheme over $\Z_p$. We will also assume that $\xi:\widetilde{X}\rightarrow X$ is a locally trivial $\T$-torsor. We start this subsection recalling the construction of the $\T$-equivariant structures of the sheaf of level $m$ differential operators $\Db^{(m)}_{\widetilde{X}}$.
\justify
Let $p_1:\widetilde{X}\times_{\text{Spec}(\Z_p)}\T\rightarrow\widetilde{X}$ and $p_{2}:\widetilde{X}\times_{\text{Spec}(\Z_p)}\T\rightarrow\T$ be the projections. For every $n\in\N$ the universal property of the m-PD-envelopes (proposition \ref{order n mPD}) gives us two canonical morphisms
\begin{eqnarray*}
d^np_1:p_1^*\Pb^{n}_{\widetilde{X},(m)}\rightarrow\Pb^{n}_{\widetilde{X}\times_{\text{Spec}(\Z_p)}\T,(m)}\;\;\;\; \text{and}\;\;\;\; d^np_2:  p_2^{*}\Pb^{n}_{\T,(m)}\rightarrow \Pb^{n}_{\widetilde{X}\times_{\text{Spec}(\Z_p)}\T,(m)}.
\end{eqnarray*}
Let $\Jb$ be the m-PD-ideal of the m-PD-algebra $\Pb^{n}_{\T,(m)}$. We have a canonical m-PD-morphism 
\begin{eqnarray*}
s:\Pb^{n}_{\widetilde{X}\times_{\text{Spec}(\Z_p)}\T,(m)}\rightarrow \Pb^{n}_{\widetilde{X}\times_{\text{Spec}(\Z_p)}\T,(m)}/p_2^*\Jb
\end{eqnarray*}
and $\rho:=s\circ d^{n}p_1$ is a m-PD-isomorphism. Then we dispose of a canonical section of $d^{n}{p_1}$, named\label{Can_sect} $q^n_1:=\rho^{-1}\circ s$ \cite[(14)]{HS1}. On the other hand, by functoriality, we obtain a morphism $d^{n}\sigma:\sigma^{*}\Pb^{n}_{\widetilde{X},(m)}\rightarrow \Pb^{n}_{{\widetilde{X}\times_{\Z_p}\T,(m)}}$ and the $\T$-equivariant structure for $\Pb^{n}_{\widetilde{X}}$ is defined by \label{Inv_D} $\Phi^{n}_{(m)}:= q^n_1\circ d^n\sigma$ \cite[Proposition 3.4.1]{HS1}. Definition \ref{Diff order n} and lemma \ref{Eq.Dual} allow us to conclude that for every $n\in\N$ the sheaf $\Db^{(m)}_{\widetilde{X},n}$ is $\T$-equivariant and the inclusions $\Db^{(m)}_{\widetilde{X},n}\hookrightarrow \Db^{(m)}_{\widetilde{X},n+1}$ are $\T$-equivariant morphisms. In particular, by lemma \ref{Eq.Filtration}, the sheaf of level $m$ differential operators is $\T$-equivariant.

\begin{rem}\label{T-Inv_Sym}
(Notation as at the end of subsection 2.3) Following the preceding lines of reasoning we can also show that, for every $n\in\N$, there exists a $m$-PD-morphism
\begin{eqnarray*}
q_1^{'n}: \Gamma^{n}_{\tilde{X}\times\T,(m)}\rightarrow p_1^*\Gamma^n_{\tilde{X},(m)}
\end{eqnarray*}
which is a section of the canonical $m$-PD-morphism $p_1^*\Gamma^n_{\tilde{X},(m)}\rightarrow \Gamma^{n}_{\tilde{X}\times\T,(m)}$ induced by $p_1$ \cite[Subsection 2.2.2]{HS1}. Let $\Gamma^n(\sigma):\sigma^*\Gamma^{n}_{\tilde{X},(m)}\rightarrow\Gamma^n_{\tilde{X}\times\T,(m)}$ be the canonical $m$-PD-morphism induced by $\sigma$. Then $\Phi^{'n}:= q_1^{'n}\circ\Gamma^{n}(\sigma)$ is a $\T$-equivariant structure for $\Gamma^{n}_{\tilde{X},(m)}$. As before, this implies that $Sym^{(m)}(\Tb_{\tilde{X}})$ is $\T$-equivariant.
\end{rem}

\begin{rem}
Although it is well-known that the tangent sheaf $\Tb_{\widetilde{X}}$ is a $\T$-equivariant quasi-coherent sheaf, we point out to the reader that this can be proven using the preceding discussion. In fact, as $\Pb^0_{\widetilde{X},(m)}=\Ob_{\widetilde{X}}$ and $\Pb^{1}_{\widetilde{X},(m)}=\Ob_{\widetilde{X}}\oplus\Omega_{\widetilde{X}}^1$ \cite[Subsection 2.2.3 (18)]{HS1}, we can apply the lemma \ref{Eq.sequence} to the sequence
\begin{eqnarray*}
0\rightarrow \Ob_{\widetilde{X}}\rightarrow \Ob_{\widetilde{X}}\oplus\Omega^1_{\widetilde{X}}\rightarrow \Omega^1_{\widetilde{X}}\rightarrow 0
\end{eqnarray*}
and lemma \ref{Eq.Dual} gives us the $\T$-equivariance of $\Tb_{\widetilde{X}}$. In particular, we dispose of the sheaves $(\Tb_{\T})^{\T}$ and $(\xi_*\Tb_{\widetilde{X}})^{\T}$.
\end{rem}
\justify
Let us recall the following discussion from \cite[4.4]{AW}. Let us suppose $U\in\Sb$ and $\tau\in \Tb_{\widetilde{X}}(\xi^{-1}(U))^{\T}$. This assumption in particular implies that $\tau$ is a $\T$-invariant vector field on $\xi^{-1}(U)$ and therefore a $\T$-invariant endomorphism of $\Ob_{\widetilde{X}}(\xi^{-1}(U))$. Hence it preserves $\Ob_{X}(\xi^{-1}(U))^{\T}$ and by lemma \ref{xi} it induces a vector field $\nu(\tau)\in\Tb_{X}(U)$. We get then a map of $\Ob_{X}$-modules
\begin{eqnarray*}
\nu: \left(\xi_{*}\Tb_{\widetilde{X}}\right)^{\T}\rightarrow\Tb_{X}.
\end{eqnarray*}
\justify
On the other hand, differentiating the right $\T$-action on $\widetilde{X}$ we obtain a $\Z_p$-linear Lie homomorphism $\mathfrak{t}\rightarrow\Tb_{\widetilde{X}}$, which induces a map of $\Ob_{X}$-modules
\begin{eqnarray*}
\mathfrak{t}\otimes_{\Z_p}\Ob_{X}\rightarrow  \left(\xi_{*}\Tb_{\widetilde{X}}\right)^{\T}.
\end{eqnarray*}
\justify
We get a complex of $\Ob_{X}$-modules $\mathfrak{t}\otimes_{\Z_p}\Ob_{X}\rightarrow \left(\xi_{*}\Tb_{\widetilde{X}}\right)^{\T}\xrightarrow{\nu}\Tb_{X}$ which is functorial in $\widetilde{X}$ \cite[subsection 4.4]{AW}.

\begin{lem}\footnote{The reasoning is as in \cite[Lemma 4.4]{AW}.}\label{Exact for tangent}
If $\xi:\widetilde{X}\rightarrow X$ is a locally trivial $\T$-torsor, then the preceding complex is in fact a short exact sequence 
\begin{eqnarray*}
0\rightarrow \mathfrak{t}\otimes_{\Z_p}\Ob_{X}\rightarrow \left(\xi_{*}\Tb_{\widetilde{X}}\right)^{\T}\xrightarrow{\nu}\Tb_{X}\rightarrow 0.
\end{eqnarray*}
\end{lem}
\justify
Before starting the proof we recall for the reader the following relations, which  come from the $\T$-equivariant structure of $\Tb_{\T}$ \cite[Lemma 2 ]{BB1}  
\begin{eqnarray}\label{Sections of T_T}
H^{0}(\T,\Tb_{\T})=\Z_p[\T]\otimes_{\Z_p}\mathfrak{t}\;\;\;\;\text{and}\;\;\;\; H^{0}(\T,\Tb_{\T})^{\T}=\mathfrak{t}.
\end{eqnarray}
\justify
Moreover, by \cite[Section II, exercise 8.3]{Hartshorne1} we also dispose of the local description
\begin{eqnarray}\label{local sections of T_tide}
\Tb_{U\times_{\Z_p}\T} = \left(\Tb_{U}\otimes_{\Z_p}\Ob_{\T}\right)\oplus \left(\Ob_{U}\otimes_{\Z_p}\Tb_{\T}\right).
\end{eqnarray}

\begin{proof}[Proof of lemma \ref{Exact for tangent}]
As the sheaves in the sequence are quasi-coherent it is enough to check exactness over an affine open subset $U\in\Sb$. First of all, since $\Tb_{X}(U)$ is a locally free $\Ob_{X}(U)$-module and $\Ob_{X}(U)$ is a flat $\Z_p$-algebra, we can conclude that $\Tb_{X}(U)$ is an inductive limit of free $\Z_p$-modules. Therefore 
\begin{eqnarray*}
(\Tb_{X}(U)\otimes_{\Z_p}\Z_p[\T])^{\T}=\Tb_{X}(U)\otimes_{\Z_p}(\Z_p[\T])^{\T}=\Tb_{X}(U).
\end{eqnarray*}
This relation, together with (\ref{Sections of T_T}) and (\ref{local sections of T_tide}), allow us to conclude that 
\begin{eqnarray}\label{Locally_Split_2}
\Tb_{\widetilde{X}}(\xi^{-1}(U))^{\T}=\Tb_{X}(U)\oplus\left(\Ob_{X}(U)\otimes_{\Z_p}\mathfrak{t}\right).
\end{eqnarray} 
\end{proof}

\begin{rem}
The preceding lemma shows that $\left(\xi_*\Tb_{\tilde{X}}\right)^{\T}$ is a locally free $\Ob_X$-module of finite rank. In particular  $Sym^{(m)}((\xi_*\Tb_{\tilde{X}})^{\T})$ is well-defined.
\end{rem}

\begin{defi}
Let $\xi:\widetilde{X}\rightarrow X$ be a locally trivial $\T$-torsor. Following \cite[page 180]{BB1} we define the level $m$ relative enveloping algebra of the torsor to be the sheaf of $\T$-invariants of $\xi_{*}\Db_{\widetilde{X}}^{(m)}$:
\begin{eqnarray*}
\widetilde{\Db^{(m)}}:=\left(\xi_{*}\Db^{(m)}_{\widetilde{X}}\right)^{\T}.
\end{eqnarray*}
\end{defi}
\justify
The preceding sheaf is endowed with a canonical filtration 
\begin{eqnarray}\label{Fil.Invariants}
\text{Fil}_{d}\left(\widetilde{\Db^{(m)}}\right) = \left(\xi_{*}\Db_{\widetilde{X},d}\right)^{\T}.
\end{eqnarray}
\justify
\textbf{Tensor product filtration.} Let $\Ab$ be a filtered sheaf of commutative rings on a topological space $Y$ \cite[A: III. 2]{Bjork}. Let $\Mb$ and $\Nb$ be filtered $\Ab$-modules \cite[A: III. 2.5]{Bjork}. The sheaf of $\Ab$-modules $\Mb\otimes_{\Ab}\Nb$ carries a natural filtration called \textit{the tensor product filtration} and it is defined as follows.  Let $n\in \N$ fix. For every $U\subset Y$ we let $F_n(\Mb(U)\otimes_{\Ab(U)}\Nb(U))$ be the abelian subgroup of $\Mb(U)\otimes_{\Ab(U)}\Nb(U)$ generated  by elements of type $x\otimes y$ with $x\in\Mb_{l}(U)$, $y\in\Nb_s(U)$, and such that $l+s\le n$. This process defines a presheaf on $Y$ and we let $F_n(\Mb\otimes_{\Ab}\Nb)$ be its sheafification. The sheaf $\Mb\otimes_{\Ab}\Nb$ becomes therefore a filtered sheaf of $\Ab$-modules
\begin{eqnarray*}
F_0(\Mb\otimes_{\Ab}\Nb)\subseteq ...\subseteq F_n(\Mb\otimes_{\Ab}\Nb)\subseteq ... \subseteq \Mb\otimes_{\Ab}\Nb.
\end{eqnarray*} 
\justify
Furthermore, for every open subset $U\subset Y$ we have a canonical map 
\begin{eqnarray*}
gr_{\bullet}\left(\Mb(U)\right)\otimes_{gr_{\bullet}\left(\Ab(U)\right)}gr_{\bullet}\left(\Nb(U)\right)\rightarrow gr_{\bullet}(\Mb(U)\otimes_{\Ab(U)}\Nb(U))
\end{eqnarray*}
by putting $x_{(l)}\otimes y_ {(s)}\rightarrow (x\otimes y)_{l+s}$, where $x\in F_l\Mb(U)-F_{l-1}\Mb(U)$, $y\in F_s\Nb(U)-F_{s-1}\Nb(U)$, and $x_{(l)}:=x+F_{l-1}\Mb(U)$ and $y_{(s)}:=y+F_{s-1}\Nb(U)$. Moreover, these morphisms are compatible under restrictions and therefore, by the universal property of the sheafification, we get a morphism of graded sheaves
\begin{eqnarray}\label{Gr_morph}
gr_{\bullet}(\Mb)\otimes_{gr_{\bullet}(\Ab)}gr_{\bullet}(\Nb)\rightarrow gr_{\bullet}\left(\Mb\otimes_{\Ab}\Nb\right)
\end{eqnarray}
which is surjective by \cite[Section I, 6.13]{Huishi}.

\begin{prop}\label{locally tilde}
For any $U\in \Sb$ there exists an isomorphism of sheaves of filtered $\Z_p$-algebras
\begin{eqnarray*}
\widetilde{\Db^{(m)}}|_{U}\xrightarrow{\simeq} \Db^{(m)}_{X}|_{U}\otimes_{\Z_p}D^{(m)}(\T).
\end{eqnarray*}
\end{prop}
\justify
Before starting the proof of the proposition let us consider the following facts. Let $n\in\N$ fix and $i\le n$. For the next few lines we will suppose that $X$ and $Z$ are smooth $\Z_p$-schemes and that $Y=X\times_{\text{Spec}(\Z_p)}Z$. Let $p_1$ and $p_2$ be the projections. By following \cite{HS1} we have defined in page \pageref{Can_sect} two canonical applications
\begin{eqnarray*}
q_1^i:\Pb^{i}_{Y,(m)}\rightarrow p_1^*\Pb^{i}_{X,(m)}\;\;\;\text{and}\;\;\; q_2^{n-1}:\Pb^{n-i}_{Y,(m)}\rightarrow p_2^*\Pb^{n-i}_{Z,(m)}.
\end{eqnarray*}
Locally, if $(t_1,...,t_{N})$ and $(t'_1,...,t'_{N'})$ are coordinated systems on $X$ and $Z$, respectively, then we obtain a coordinated system on $Y$ by putting $(p_1^*(t_1),... ,p_1^{*}(t_N), p_2^{*}(t'_1),...,p_2^*(t'_{N'})$. We have 
\begin{eqnarray*}
\;\;\;\;\;\;\Pb^{i}_{Y,(m)}\simeq \displaystyle\bigoplus_{|\underline{l_1}|+|\underline{l_2}|\le i}\Ob_{Y}p_1^*(\underline{\tau}^{\{\underline{l_1}\}})p_{2}^*(\underline{\tau'}^{\{\underline{l_2}\}})\;\;\;(\tau_i:=1\otimes t_i-t_i\otimes 1\;\text{and}\;\tau'_i:=1\otimes t'_1-t'_i\otimes 1).
\end{eqnarray*} 
In this case \cite[subsection 2.2.2]{HS1}
\begin{eqnarray*}
q_{1}^i\left(\sum_{\underline{l_1},\underline{l_2}}a_{\underline{l_1},\underline{l_2}}p_{1}^*(\underline{\tau}^{\{\underline{l_1}\}})p_2^{*}(\underline{\tau'}^{\{\underline{l_2}\}})\right)=\sum_{\underline{l_1}}a_{\underline{l_1},0}p_{1}^*(\underline{\tau}^{\{\underline{l_1}\}})
\end{eqnarray*}
(with a similar description for $q_{2}^{n-i}$) and we have an isomorphism
\begin{eqnarray}\label{Filtration_P}
\Pb_{Y,(m)}^{n}\xrightarrow{\simeq} \displaystyle\bigoplus_{0\le i\le n}p_{1}^*\Pb^{i}_{X,(m)}\otimes_{\Ob_{Y}}p_2^*\Pb^{n-i}_{Z,(m)}.
\end{eqnarray}
Moreover, since $\Pb^{i}_{X,(m)}$ and $\Pb^{n-i}_{Z,(m)}$ are locally free $\Ob$-modules of finite rank, taking duals in (\ref{Filtration_P}) we get a canonical isomorphism
\begin{eqnarray}\label{Filtration_D}
\Db^{(m)}_{Y,n}\xrightarrow{\simeq}\displaystyle\bigoplus_{0\le i \le n}p_1^*\Db^{(m)}_{X,i}\otimes_{\Ob_{Y}}p_2^{*}\Db^{(m)}_{Z,n-i}
\end{eqnarray} 
\begin{proof}[Proof of proposition \ref{locally tilde}]
Let $U\in\Sb$ and let $\xi^{-1}(U)\simeq U\times_{\text{Spec}(\Z_p)}\T$ be a trivialization of $\xi$ over $U$.  We obtain the following isomorphisms of filtered $\Z_p$-algebras
\begin{eqnarray*}
\left(\xi_*\Db_{\widetilde{X}}^{(m)}\right)^{\T}(U)=\Db^{(m)}_{\widetilde{X}}(\xi^{-1}(U))^{\T}
\simeq  \Db^{(m)}_{U\times\T}(U\times\T)^{\T}
\simeq  \Db^{(m)}_{X}(U)\otimes_{\Z_p}H^{0}(\T,\Db_{\T}^{(m)})^{\T}
\simeq   \Db^{(m)}_{X}(U)\otimes_{\Z_p} D^{(m)}(\T)
\end{eqnarray*}
where the first isomorphism follows from the fact that $U$ trivializes the $\T$-torsor $\xi$, the second isomorphism becomes from (\ref{Filtration_D}) and the Kunneth formula \cite[Theorem 6.7.8]{Grothendieck3}), and the third isomorphism is given by $(ii)$ in remark \ref{Invarian global sections of group}. Since the previous isomorphisms are compatible with restrictions to open affine subsets contained in $U$, we obtain the desired isomorphism of sheaves of filtered $\Z_p$-algebras.  
\end{proof}

\begin{prop}\label{gr tilde}
If $\xi$ is a locally trivial $\T$-torsor, then there exists a canonical and graded isomorphism
\begin{eqnarray*}
Sym^{(m)}\left(\left(\xi_{*}\Tb_{\widetilde{X}}\right)^{\T}\right)\xrightarrow{\simeq} gr_{\bullet}\left(\widetilde{\Db^{(m)}}\right).
\end{eqnarray*}
\end{prop}
\begin{proof}
We will divide the proof into two cases. We will first consider the case $m=0$ and then we will generalize for all $m\in\Z_{>0}$.
\justify
\textit{Case 1.} Let us suppose that $m=0$. By the remark given after the proposition 1.2.2 in \cite{Huyghe1} we know that if $\Lb$ is a locally free $\Ob_{\widetilde{X}}$-module of finite rank, then $\text{Sym}^{(0)}(\Lb)= \textbf{S}(\Lb)$ is the symmetric algebra of $\Lb$. By (\ref{graded}), which is true for every $m\in\N$ (cf. \cite[Proposition 1.3.7.3]{Huyghe1}), we have a canonical isomorphism of graded $\Ob_{\widetilde{X}}$-algebras
\begin{eqnarray*}
\text{\textbf{S}}(\Tb_{\widetilde{X}})\xrightarrow{\simeq} \text{gr}_{\bullet}\left(\Db^{(0)}_{\widetilde{X}}\right).
\end{eqnarray*}
Applying the direct image functor $\xi_*$ to the preceding isomorphism and then taking $\T$-invariants sections (both functors being exact by lemma \ref{triviality} and the fact that $\T$ is diagonalisable \cite[Part I, Lemma 4.3 (b)]{Jantzen})
we get an isomorphism
\begin{eqnarray*}
\left(\xi_*\text{\textbf{S}}\left(\Tb_{\widetilde{X}}\right)\right)^{\T}\xrightarrow{\simeq} \text{gr}_{\bullet}\left(\left(\xi_*\Db^{(0)}_{\widetilde{X}}\right)^{\T}\right).
\end{eqnarray*}
We remark for the reader that the left-hand side of the previous isomorphism is well defined by remark \ref{T-Inv_Sym}. To complete the proof of the first case, we need to show that $(\xi_*\;\text{\textbf{S}}(\Tb_{\widetilde{X}}))^{\T}=\text{\textbf{S}}\left((\xi_*\Tb_{\widetilde{X}})^{\T}\right)$. To do that, we start by considering the canonical map of $\Ob_X$-modules 
\begin{eqnarray*}
\left(\xi_*\Tb_{\widetilde{X}}\right)^{\T}\rightarrow \left(\xi_*\text{\textbf{S}}\left(\Tb_{\widetilde{X}}\right)\right)^{\T}
\end{eqnarray*}
which induces, by universal property of $\text{\textbf{S}}(\bullet)$, a canonical morphism of graded $\Ob_{X}$-algebras
\begin{eqnarray*}
\text{\textbf{S}}\left(\left(\xi_{*}\Tb_{\widetilde{X}}\right)^{\T}\right)\xrightarrow{\varphi} \left(\xi_{*}\text{\textbf{S}}\left(\Tb_{\widetilde{X}}\right)\right)^{\T}.
\end{eqnarray*}
Let us see that $\varphi$ is indeed an isomorphism. Let us take $U\in\Sb$. We have a commutative diagram
\begin{eqnarray*}
\begin{tikzcd}
\widetilde{U}:=\xi^{-1}(U) \arrow{rr}{\simeq} \arrow[swap]{dr}{\xi}& & U\times_{\Z_p}\T \arrow{dl}{\text{p}_1}\\
& U & 
\end{tikzcd}
\end{eqnarray*}
which tells us that (cf. \cite[Section II, exercise 8.3]{Hartshorne1}) 
\begin{eqnarray}\label{T_over_U_tilde}
\Tb_{\widetilde{U}} = \xi^{*}\Tb_{U}\oplus \text{p}_2^* \Tb_{\T} = \xi^* \Tb_{U}\oplus \left(\Ob_{\widetilde{U}}\otimes_{\Z_p}\mathfrak{t}\right).
\end{eqnarray}
By (\ref{Locally_Split_2}), we have 
\begin{eqnarray*}
\text{\textbf{S}}\left(\left(\xi_*\Tb_{\widetilde{X}}\right)^{\T}\right)(U) = \text{\textbf{S}}\left(\left(\xi_*\Tb_{\widetilde{X}} (U)\right)^{\T}\right) = \text{\textbf{S}}\left(\Tb_{U}(U)\oplus\left(\Ob_{U}(U)\otimes_{\Z_p}\mathfrak{t}\right)\right).  
\end{eqnarray*}
On the other hand, by (\ref{T_over_U_tilde}) we have the following relation
\begin{eqnarray}\label{Rel_1}
\text{\textbf{S}}(\Tb_{\widetilde{U}}) = \text{\textbf{S}}\left(\xi^*\Tb_{U}\right)\otimes_{\Ob_{\widetilde{U}}} \text{\textbf{S}}\left(\Ob_{\widetilde{U}}\otimes_{\Z_p}\mathfrak{t}\right) = \xi^*\text{\textbf{S}}\left(\Tb_{U}\right)\otimes_{\Ob_{\widetilde{U}}}\text{\textbf{S}}\left(\Ob_{\widetilde{U}}\otimes_{\Z_p}\mathfrak{t}\right)
\end{eqnarray}
which implies, by the projection formula \cite[Chapter II, section 5, exercise 5.1 (d)]{Hartshorne1} that

\begin{align}\label{Rel_2}
\xi_*\text{\textbf{S}}(\Tb_{\widetilde{U}}) & = \xi_*\left( \xi^*\text{\textbf{S}}\left(\Tb_{U}\right)\otimes_{\Ob_{\widetilde{U}}}\text{\textbf{S}}\left(\Ob_{\widetilde{U}}\otimes_{\Z_p}\mathfrak{t}\right)\right) \notag \\
& = \text{\textbf{S}}\left(\Tb_{U}\right)\otimes_{\Ob_{U}}\xi_*\text{\textbf{S}}\left(\Ob_{\widetilde{U}}\otimes_{\Z_p}\mathfrak{t}\right).
\end{align}
\justify
Taking $\T$-invariants and sections on $U$ we get
\begin{eqnarray*}
\left(\xi_*\text{\textbf{S}}(\Tb_{\widetilde{U}})\right)^{\T}(U) = \text{\textbf{S}}\left(\Tb_{U}(U)\right)\otimes_{\Ob_{U}(U)} \text{\textbf{S}}(\Ob_{U}(U)\otimes_{\Z_p}\mathfrak{t}).
\end{eqnarray*}
Summing up, we have the following commutative diagram
\begin{eqnarray*}
\begin{tikzcd}
\text{\textbf{S}}\left(\left(\xi_*\Tb_{\widetilde{X}}\right)^{\T}\right)(U)  \arrow[r, "\varphi_{U}"] \arrow[d, "\simeqd "]
& \left(\xi_*\text{\textbf{S}}(\Tb_{\widetilde{X}})\right)^{\T}(U) \arrow[d, "\simeqd "]\\
\text{\textbf{S}}\left(\Tb_{U}(U)\oplus\left(\Ob_{U}(U)\otimes_{\Z_p}\mathfrak{t}\right)\right)  \arrow[r, "\simeq"]
& \text{\textbf{S}}\left(\Tb_{U}(U)\right)\otimes_{\Ob_{U}(U)} \text{\textbf{S}}(\Ob_{U}(U)\otimes_{\Z_p}\mathfrak{t})
\end{tikzcd}
\end{eqnarray*}
which ends the proof of the first case because $\Sb$ is a base for the Zariski topology of $X$.
\justify
\textit{Case 2.} Let  us suppose now that $m\in\Z_{>0}$. Exactly as we have done at the beginning of case 1, applying $\xi_*$ to the isomorphism (\ref{graded}) and then taking $\T$-invariant sections, we get a canonical isomorphism of graded $\Ob_{X}$-algebras
\begin{eqnarray*}
\left(\xi_*\text{Sym}^{(m)}\left(\Tb_{\widetilde{X}}\right)\right)^{\T}\xrightarrow{\simeq}\text{gr}_{\bullet}\left(\left(\xi_*\Db^{(m)}_{\widetilde{X}}\right)^{\T}\right).
\end{eqnarray*}
We want to see that the map $\varphi$, built in case 1, induces an isomorphism
\begin{eqnarray*}
\text{Sym}^{(m)}\left(\left(\xi_*\Tb_{\widetilde{X}}\right)^{\T}\right) \simeq \left(\xi_*\text{Sym}^{(m)}\left(\Tb_{\widetilde{X}}\right)\right)^{\T}.
\end{eqnarray*}
To do that, we take $U\in\Sb$ and we begin by noticing that analogously to case 1 the relation (\ref{Locally_Split_2}) gives us
\begin{eqnarray}\label{Sym_1}
\text{Sym}^{(m)}\left(\left(\xi_*\Tb_{\widetilde{X}}\right)^{\T}\right)(U) = \text{Sym}^{(m)}\left(\Tb_{U}(U)\oplus\left(\Ob_{U}(U)\otimes_{\Z_p}\mathfrak{t}\right)\right).
\end{eqnarray}
Moreover, the relation (\ref{Locally_Split_2}) and \cite[Proposition 1.3.5]{Huyghe1} give us
\begin{eqnarray*}
\text{Sym}^{(m)}\left(\Tb_{\widetilde{U}}\right) = \text{Sym}^{(m)}\left(\xi^*\Tb_{U}\right)\otimes_{\Ob_{\widetilde{U}}}\text{Sym}^{(m)}\left(\Ob_{\widetilde{U}}\otimes_{\Z_p}\mathfrak{t}\right)
\end{eqnarray*}
which, following the same arguments that in (\ref{Rel_1}) and (\ref{Rel_2}), implies that
\begin{eqnarray}\label{Sym_2}
\left(\xi_*\text{Sym}^{(m)}\left(\Tb_{\widetilde{X}}\right)\right)^{\T}(U) = \text{Sym}^{(m)}\left(\Tb_{U}(U)\right)\otimes_{\Ob_{U}(U)}\text{Sym}^{(m)}\left(\Ob_{U}(U)\otimes_{\Z_p}\mathfrak{t}\right).
\end{eqnarray}
\justify
Again, by \cite[Proposition 1.3.5]{Huyghe1}, we have that (\ref{Sym_1}) and (\ref{Sym_2}) are canonically isomorphic, so in order to globalize this map, which we denote by $\varphi^{(m)}_{U}$, we need to check that the following diagram is commutative 
\begin{eqnarray*}
\begin{tikzcd}
 \text{Sym}^{(m)}\left(\Tb_{U}(U)\oplus\left(\Ob_{U}(U)\otimes_{\Z_p}\mathfrak{t}\right)\right) \arrow[r, hook] \arrow[d, "\varphi^{(m)}_{U}"]
& \text{\textbf{S}}\left(\Tb_{U}(U)\oplus\left(\Ob_{U}(U)\otimes_{\Z_p}\mathfrak{t}\right)\right)\otimes_{\Z_p}\Q_p \arrow [d, "\varphi_{U}\otimes_{\Z_p}1_{\Q_p}"]\\
 \text{Sym}^{(m)}\left(\Tb_{U}(U)\right)\otimes_{\Ob_{U}(U)}\text{Sym}^{(m)}\left(\Ob_{U}(U)\otimes_{\Z_p}\mathfrak{t}\right) \arrow[r, hook]
&  \text{\textbf{S}}\left(\Tb_{U}(U)\right)\otimes_{\Ob_{U}(U)}\text{\textbf{S}}\left(\Ob_{U}(U)\otimes_{\Z_p}\mathfrak{t}\right)\otimes_{\Z_p}\Q_p.
\end{tikzcd}
\end{eqnarray*}
Shrinking $U$ if necessary, we can suppose that $U$ is endowed with a set of local coordinates $x_1,\; ...,\; x_N$, in such a way that if $\Tb_{U}(U)$ is generated as $\Ob_{U}(U)$-module by the derivations $\partial_{x_1},\; ...,\;\partial_{x_N}$, and if $\zeta_1,\; ...,\; \zeta_l$ denotes a $\Z_p$-basis of $\mathfrak{t}$, then $ \text{Sym}^{(m)}\left(\Tb_{U}(U)\oplus\left(\Ob_{U}(U)\otimes_{\Z_p}\mathfrak{t}\right)\right)$ is generated (as $\Ob_{U}(U)$-module) by all the elements of the form $\underline{\partial}^{<\underline{k}>}\cdot\underline{\zeta}^{<\underline{v}>}$ \footnote{Here we use the multi-index notation introduced in sections \ref{ADP} and \ref{ADA}.}. In particular, 
\begin{eqnarray*}
\varphi^{(m)}_{U}\left(\underline{\partial}^{<\underline{k}>}\cdot\underline{\zeta}^{<\underline{v}>}\right) = \varphi_U\otimes_{\Z_p}1_{\Q_p}\left(\underline{\partial}^{<\underline{k}>}\cdot\underline{\zeta}^{<\underline{v}>}\right) = \displaystyle\frac{\underline{k}!}{q_{\underline{k}}!}\displaystyle\frac{\underline{v}!}{q_{\underline{v}}!} \; \underline{\partial}^{\underline{k}}\otimes_{\Q_p} \underline{\zeta}^{\underline{v}}.
\end{eqnarray*}
which shows that the preceding diagram is commutative. This ends the proof of the proposition.
\end{proof}

\subsection{Affine algebraic groups and homogeneous spaces}

Let us suppose that $\G$ is a  split connected reductive group scheme over $\Z_p$ and $\T$ is a split maximal torus in $\G$. As we know, the Lie algebra $\mathfrak{g}=\text{Lie}(\G)$ is a $\T$-module via the adjoint representation \cite[I, 7.18]{Jantzen} and the decomposition into weight spaces has the form
\begin{eqnarray*}
\text{Lie}(\G)= \text{Lie}(\T)\oplus \displaystyle\bigoplus_{\alpha\in \Lambda} (\text{Lie} (\G))_{\alpha}.
\end{eqnarray*}
Here $\Lambda$ is the subset of $X(\T)=\text{Hom}(\T,\G_{m})$ of non-zero weights of $\text{Lie}(\G)$, this means the roots of $\G$ with respect to $\T$.
\justify
For each $\alpha\in \Lambda$ there exists a homomorphism $x_\alpha: \G_a\rightarrow\G$ satisfying 
\begin{eqnarray}\label{Normal}
t\;x_{\alpha}(a)\;t^{-1}=x_{\alpha}(\alpha(t)\;a),
\end{eqnarray}
for any $\Z_p$-algebra $A$ and all $t\in\T(A)$, and such that the tangent map $dx_{\alpha}:\text{Lie}(\G_a)\rightarrow(\text{Lie}(\G))_{\alpha}$ is an isomorphism \cite[II, 1.2]{Jantzen}. This homomorphism defines a functor $A\mapsto x_{\alpha}(\G_{a}(A))$ which is a closed subgroup of $\G$ and it is denoted by $U_{\alpha}$. By definition we have  $\text{Lie}(U_{\alpha})=(\text{Lie}(\G))_{\alpha}$ and by  (\ref{Normal}) it is clear that $\T$ normalises $U_\alpha$.
\justify
Now, let us choose a positive system $\Lambda^+\subset \Lambda$. It is known that $\Lambda^+$ and $-\Lambda^+$ are unipotent and closed subsets of $\Lambda$\footnote{$\Lambda^+\cap(-\Lambda^{+})=\emptyset$ and $(\N\alpha +\N\beta)\cap\Lambda\subset\Lambda^{+}$ for any $\alpha,\beta\in\Lambda^{+}$.} \cite[II, 1.7]{Jantzen}. Let $\textbf{N}$ be the closed subgroup of $\G$ generated by all $U_{\alpha}$ with $\alpha\in \Lambda^+$. As we have remarked $\T$ normalises $\textbf{N}$. We set 
\begin{eqnarray}\label{Borel}
\B=\textbf{N}\rtimes\T
\end{eqnarray}
a Borel subgroup of $\G$. With this terminology $\textbf{N}$ is called the \textit{unipotent radical} of $\B$. We put
\begin{eqnarray*}
\widetilde{X}:= \G/\textbf{N},\;\; X:=\G/\B
\end{eqnarray*} 
for the corresponding quotients (the basic affine space and the flag scheme of $\G$ \cite[subsection 4.7]{AW}). Since $\Z_p$ is in particular a Dedekind domain, these are smooth and separated schemes over $\Z_p$ \cite[Lemma 4.7 (a)]{AW}.

\begin{rem}\label{Right_actions}
For technical reasons (cf. Proposition \ref{morp.HS}) in this work we will suppose that the group $\G$, and the schemes $\widetilde{X}$ and $X$ are endowed with the right regular $\G$-action. This means that for any $\Z_p$-algebra $A$ and $g_0,g\in \G(A)$ we have 
\begin{eqnarray*}
g_0\bullet g = g^{-1}g_0,\;\;\; g_0\;\textbf{N}(A)\bullet g = g^{-1}g_0\; \textbf{N}(A) \;\;\text{and}\;\; g_0\;\B(A)\bullet g=g^{-1}g_0\;\B(A).
\end{eqnarray*}
Under this actions, the canonical projections $\G\rightarrow \widetilde{X}$ and $\G\rightarrow X$ are clearly $\G$-equivariant.
\end{rem}
\justify
Now, as $\T$ normalises $\textbf{N}$ we have $\textbf{N}\T = \T\textbf{N}$ and therefore
\begin{eqnarray*}
(g\textbf{N}(A)).t\subset g\T(A)\textbf{N}(A),
\end{eqnarray*}
for any $\Z_p$-algebra $A$, $g\in\G(A)$ and $t\in\T(A)$. This defines a right $\T$-action on $\widetilde{X}$ which clearly commutes with the right regular $\G$-action (cf. Remark \ref{Right_actions}). Moreover, this right $\T$-action makes the canonical projection $\xi:\widetilde{X}\rightarrow X$
a $\T$-torsor for the Zariski topology of $X$. To see this we recall first that from (\ref{Borel}) the \textit{abstract Cartan group} $\textbf{H}:=\B/\textbf{N}$ is canonical isomorphic to $\T$. Let us consider the covering of $X$ given by the open subschemes $U_w$, $w\in W:=N_{\G}(\T)/\T$ (the Weyl group) where
\begin{eqnarray*}
U_w:= \text{image of}\; w\textbf{N}\B
\end{eqnarray*}
under the canonical projection $\G\rightarrow X$ \cite[Part II, chapter 1, 1.9 (7)]{Jantzen}. For every $w\in W$ we can find a morphism $\pi_{w}:U_w\rightarrow\G$ splitting the projection map $\G\rightarrow X$ \cite[Part II, chapter 1, 1.10 (1) and (2)]{Jantzen}. This map gives a map $\overline{\pi}_w:U_{w}\rightarrow \widetilde{X}$ such that $\xi\circ\overline{\pi}_{w}=id_{U_w}$. The map $(u,b\textbf{N})\mapsto\pi_{w}(u)b\textbf{N}$ is the required $\T$-invariant isomorphism $U_w\times\T\xrightarrow{\simeq} U_w\times\textbf{H}\xrightarrow{\simeq}\xi^{-1}(U_w)$. Now we can apply \cite[Chaper III, Proposition 4.1 (b)]{Milne}. As in definition \ref{locally trivial}  we denote by $\Sb$ the set of all affine open subsets of $X$ that trivialise the torsor $\xi$. This forms a base for the Zariski topology of $X$.

\subsection{Relative enveloping algebras of finite level on homogeneous spaces}\label{Rel_env_alg_on_hs}
\justify
In this section we adopt the notation of the preceding section. In particular, we recall for the reader that the set $\Sb$, of all affine open subsets of $X$ that trivialise the torsor $\xi$ forms a base for the Zariski topology of $X$.  Let us recall that by proposition \ref{morp.HS} and remark \ref{Right_actions} the right regular $\G$-action on $\widetilde{X}$ (introduced in remark \ref{Right_actions}) induces a homomorphism $\Phi^{(m)}: D^{(m)}(\G)\rightarrow H^{0}\left(\widetilde{X},\Db^{(m)}_{\widetilde{X}}\right)$ which equals the operator-representation (notation in subsection 2.5)
\begin{eqnarray*}
\Psi_{\widetilde{X}_\Q}: \Ub(\mathfrak{g}_\Q)\rightarrow H^{0}\left(\widetilde{X}_{\Q},\Db_{\widetilde{X}_\Q}\right)
\end{eqnarray*}
if we tensor with $\Q_p$ ($\Db_{\widetilde{X}_\Q}$ denotes the usual sheaf of differential operators on $\widetilde{X}_\Q$). Let us consider the base change $\T_\Q:=\T\times_{\text{Spec}(\Z_p)}\text{Spec}(\Q_p)$. We know by \cite[Part I, 2.10 (3)]{Jantzen} that 
\begin{eqnarray} \label{BaseChange}
H^0\left(\widetilde{X}_\Q,\Db_{\widetilde{X}_\Q}\right)^{\T_\Q}=H^0\left(\widetilde{X},\Db^{(m)}_{\widetilde{X}}\right)^{\T}\otimes_{\Z_p}\Q_p.
\end{eqnarray}
Given that the right regular $\G$-action on $\widetilde{X}$ commutes with the right action of the torus $\T$, the vector fields by which $\mathfrak{g}_\Q$ acts on $\widetilde{X}_\Q$ must be invariant under the $\T_\Q$-action \cite[Lemma 4.5]{BB2}. This means that the \textit{operator-representation} $\Psi_{\widetilde{X}_\Q}$ satisfies 
\begin{eqnarray*}
\Psi_{\widetilde{X}_\Q}(\mathfrak{g}_\Q)\subset H^0(\widetilde{X}_\Q,\Db_{\widetilde{X}_\Q})^{\T_\Q}.
\end{eqnarray*}
The relation (\ref{BaseChange}) tells us that for every $x\in D^{(m)}(\G)$ there exists $k(x)\in\N$ (a natural number that depends of $x$) such that  
\begin{eqnarray*}
p^{k(x)}\Phi^{(m)}(x)\subset H^{0}(\widetilde{X},\Db^{(m)}_{\widetilde{X}})^{\T}.
\end{eqnarray*}
Since the $\T$-action on $H^{0}(\widetilde{X},\Db^{(m)}_{\widetilde{X}})$ is $\Z_p$-linear, for every $\Z_p$-algebra $A$ and every $t\in\T(A)$ we have \label{Reasoning}
\begin{eqnarray}\label{omega_mult}
p^{k(x)}\Phi^{(m)}(x)=t.\left(p^{k(x)}\Phi^{(m)}(x)\right)=p^{k(x)}\left(t.\Phi^{(m)}(x)\right).
\end{eqnarray}
Since $\widetilde{X}$ is a smooth $\Z_p$-scheme, the local description (\ref{locally_Berthelot}) tells us that the sheaf $\Db^{(m)}_{\tilde{X}}$ is $p$-torsion free. In particular, $H^{0}(\widetilde{X},\Db^{(m)}_{\widetilde{X}})$ is also $p$-torsion free and therefore, from (\ref{omega_mult}), we have that $\Phi^{(m)}$ induces the filtered morphism 
\begin{eqnarray*}\label{Can.Morp1}
\Phi^{(m)}:D^{(m)}(\G)\rightarrow H^{0}\left(\widetilde{X},\Db^{(m)}_{\widetilde{X}}\right)^{\T}= H^{0}\left(X,\xi_*\Db_{\widetilde{X}}^{(m)}\right)^{\T}.
\end{eqnarray*} 
From the preceding reasoning we have an $\Ob_{X}$-morphism of sheaves of filtered $\Z_p$-algebras
\begin{eqnarray}\label{Can.Morp2}
\Phi^{(m)}_{X}: \Ab^{(m)}_{X}\rightarrow \widetilde{\Db^{(m)}}.
\end{eqnarray}
\justify
The sheaf $\Ab_{X}^{(m)}:=\Ob_{X}\otimes_{\Z_p}D^{(m)}(\G)$ of associative $\Z_p$-algebras has been introduced in the subsection 2.5. We recall for the reader that this is an integral model of the sheaf $\Ub^{\circ}:=\Ob_{X_\Q}\otimes_{\Q}\Ub(\mathfrak{g}_\Q)$. \\
To twist the sheaves $\widetilde{\Db^{(m)}}$ introduced in the subsection 3.3 we consider the classical distribution algebra as in \cite[Chapter II, 4.6.1]{DG}. To define it, we suppose that $\varepsilon: \text{Spec}(\Z_p)\rightarrow\T$ is the identity of $\T$ and we take $J:=\{f\in \Z_p[\T]|\; f(\varepsilon)=0\}$. Then $\Z_p[\T]=\Z_p\oplus J$. We put
\begin{eqnarray*}
\text{Dist}_{n}(\T):=\left(Z_p[\T]/J^{n+1}\right)^*= \text{Hom}_{\Z_p}(\Z_p[\T]/J^{n+1},\Z_p)\subset\Z_p[\T]^*
\end{eqnarray*}
the space of distributions of order $n$, and then $\text{Dist}(\T):= \varinjlim_{n\in\N} \text{Dist}_{n}(\T)$. Moreover, if $\Delta_{\T}:\Z_p[\T]\rightarrow\Z_p[\T]\otimes_{\Z_p}\Z_p[\T]$ denotes the comorphism associated to the multiplication of $\T$, $\varepsilon_{\T}:\Z_p[\T]\rightarrow\Z_p$ is the counit associated to the identity element and $i^*_{\T}:\Z_p[\T]\rightarrow\Z_p[\T]$ is the coinverse (these maps defining a structure of Hopf algebra on $\Z_p[\T]$), then the product
\begin{eqnarray*}
uv:\Z_p[\T]\xrightarrow{\Delta_{\T}}\Z_p[\T]\otimes_{\Z_p}\Z_p[\T]\xrightarrow{u\otimes v}\Z_p\;\;\;\;\;\;\;\;u,v\in \Z_p[\T]^*
\end{eqnarray*}
defines a structure of algebra on $\Z_p[\T]^*$ and $\text{Dist}(\T)$ is a subalgebra with $\text{Dist}_{n}(\T).\text{Dist}_m(\T)\subset\text{Dist}_{m+n}(\T)$ \cite[Part I, 7.7]{Jantzen}. Furthermore, $\text{Dist}_n(\T)\simeq\Z_p\oplus(J/J^{n+1})^{*}$.

\begin{prop}\cite[Subsection 4.1]{HS1}
\begin{itemize}
\item[(i)]The applications $Hom_{\Z_p}(\psi_{m,m'},\Z_p): D^{(m)}(\T)\rightarrow D^{(m')}(\T)$, with $\psi_{m,m'}$ as in subsection 2.4, induce an isomorphism of filtered $\Z_p$-algebras $\varinjlim_{m\in\N} D^{(m)}(\T) \xrightarrow{\simeq} Dist(\T)$.
\item[(ii)] The distribution algebra $\text{Dist}(\T)$ is an integral model of $\Ub(\mathfrak{t}_\Q)$,  this means that $\text{Dist}(\T)\otimes_{\Z_p} \Q_p=\Ub(\mathfrak{t}_\Q)$.
\end{itemize}
\end{prop}

\begin{exa}\label{example T_G_m}
Let us suppose that $\T=\G_m=\text{Spec}(\Z_p[T,T^{-1}])$. In this case $J$ is generated by $T-1$ and the residue classes of\; $1,\; T-1,\; ...,\; (T-1)^{n}$ form a basis of $\Z_p[\T]/J^{n+1}$. Let $\delta_{n}\in\text{Dist}(\T)$ such that $\delta_n((T-1)^i)=\delta_{n,i}$ (the Kronecker delta). By \cite[Part I, 7.8]{Jantzen} all $\delta_n$ with $n\in\N$ form a basis of $\text{Dist}(\T)$ and they satisfy the relation 
\begin{eqnarray}\label{Binomial}
n!\delta_{n}=\delta_1(\delta_1-1)...(\delta_1-n+1).
\end{eqnarray} 
Therefore $\text{Dist}(\T)\otimes_{\Z_p}\Q_p=\Q_p[\delta_1]$. Since $\mathfrak{t}=(J/J^2)^*$ we can conclude that $\text{Dist}(\T)\otimes_{\Z_p}\Q_p=\Ub(\mathfrak{t}_\Q)$.
\end{exa}
\justify
The preceding proposition in particular implies that every morphism of $\Z_p$-algebras $\lambda: \text{Dist}(\T)\rightarrow Z_p$ induces for every $m\in\N$ a morphism of $\Z_p$-algebras $\lambda:D^{(m)}(\T)\rightarrow\Z_p$.
\begin{UEA_AD}\label{UEA_AD}
Let us clarify the mysterious characters $\lambda:\text{Dist}(\T)\rightarrow\Z_p$. Let us suppose first that $\T=\G_m=\text{Spec}(\Z_p[T,T^{-1}])$. By the preceding example we know that the set of distributions $\{\delta_n\}_{n\in\N}$, where $\delta_n((T-1)^i)=0$ if $i< n$ and $\delta_{n}((T-1)^n)=1$, is a basis for $\text{Dist}(\T)$. Moreover, $\text{Dist}(\T)\otimes_{\Z_p}\Q_p=\Q_p[\delta_1]$. Now, let us take $\lambda\in\mathfrak{t}^{*}$, which induces a morphism of algebras $\lambda:\Ub(\mathfrak{t})\rightarrow\Z_p$. Taking the tensor product with $\Q_p$ and using the canonical isomorphism $\text{Dist}(\T)_{\Q}\simeq\Ub(\mathfrak{t}_{\Q})$ we obtain a character $\lambda:\text{Dist}(\T)\rightarrow\Q_p$ (of course, here we assume Dist$(\T)\subset\text{Dist}(\T)_{\Q}$). To see that its image is contained in $\Z_p$, we need to check that $\lambda(\delta_{n})\in\Z_p$. By (\ref{Binomial}) we have
\begin{eqnarray*}
\lambda(\delta_n)=\lambda\left({\delta_1\choose n}\right) = {\lambda(\delta_1)\choose n}\in\Z_p.
\end{eqnarray*}
Where we have used the fact that binomial coefficients extend to functions from $\Z_p$ to $\Z_p$ and the fact that $\delta_1\in\mathfrak{t}$. In the case of an arbitrary split maximal torus $\T=\G_m\times_{\text{Spec}(\Z_p)} ... \times_{\text{Spec}(\Z_p)} \G_m$ (n-times), the reader can follow the same reasoning using the canonical isomorphism $\text{Dist}(\T)=\text{Dist}(\G_m)\otimes_{\Z_p}...\otimes_{\Z_p}\text{Dist}(\G_m)$ (n-times) \cite[Part I, 7.9 (3)]{Jantzen}. We have therefore a correspondence between the characters of $\mathfrak{t}$ (the Lie algebra of a split maximal torus $\T\subset\G$) and the characters of the distribution algebra studied in this text
\begin{eqnarray}\label{Iso_chars_Z_p}
\text{Hom}_{\Z_p-\text{mods}}\left(\mathfrak{t},\Z_p\right)\xrightarrow{\simeq} \text{Hom}_{\Z_p-\text{alg}}\left(\text{Dist}(\T),\Z_p\right).
\end{eqnarray}
\end{UEA_AD}
\begin{roots}\label{roots} Let us consider the positive system $\Lambda^{+}\subset\Lambda \subset X(\T)$ ($X(\T)$ the group of algebraic characters) associated to the Borel subgroup scheme $\B\subset\G$ defined in the preceding subsection. The Weyl subgroup $W=N_{\G}(\T)/\T$ acts naturally on the space $\mathfrak{t}^*_\Q:=\text{Hom}_{\Q_p-\text{mod}}(\mathfrak{t}\otimes_{\Z_p}\Q_p,\Q_p)$, and via differentiation $d: X(\T)\hookrightarrow \mathfrak{t}^*$ we may view $X(\T)$ as a subgroup of $\mathfrak{t}^*$ in such a way that $X^*(\T)\otimes_{\Z_p}\Q_p=\mathfrak{t}^*_\Q$. Let $\rho=\frac{1}{2}\sum_{\alpha\in\Lambda^{+}}\alpha$ be the so-called Weyl vector. Let $\check{\alpha}$ be a coroot of $\alpha\in\Lambda$ viewed as an element of $\mathfrak{t}_\Q$. An arbitrary weight $\lambda\in\mathfrak{t}^*_\Q$ is called \textit{dominant} if $\lambda(\check{\alpha})\ge 0$ for all $\alpha\in\Lambda^{+}$. The weight $\lambda$ is called \textit{regular} if its stabilizer under the $W$-action is trivial.
\end{roots}
\justify
We recall for the reader that $D^{(m)}(\T)$ is an integral model of the universal enveloping algebra $\Ub(\mathfrak{t}_\Q)$.
 
\begin{defi}\label{character of the dist. Alg}  
We say that a morphism of $\Z_p$-algebras $\lambda:\text{Dist}(\T)\rightarrow \Z_p$ (resp. the induced morphism $\lambda:D^{(m)}(\T)\rightarrow\Z_p$) is a character of the distribution algebra $Dist(\T)$ (resp. a character of the level $m$ distribution algebra $D^{(m)}(\T)$). We say that a character $\lambda:\text{Dist}(\T)\rightarrow\Z_p$ (resp. a character $\lambda:D^{(m)}(\T)\rightarrow\Z_p$) is a dominant and regular character if the $\Q_p$-linear map induced by tensoring with $\Q_p$ is a dominant and regular character of $\mathfrak{t}_\Q$.
\end{defi}
\justify
Let  $\lambda: \text{Dist}(\T)\rightarrow\Z_p$  be a character of the distribution algebra of  $\T$. We can consider the ring $\Z_p$ as a $D^{(m)}(\T)$-module  via $\lambda$.
\justify
The reader can easily verify the following elementary lemma. 
\begin{lem}
Let $A$ be a $\Q_p$-algebra and $A_0\subset A$ a $\Z_p$-subalgebra such that $A_0\otimes_{\Z_p}\Q_p=A$. If $Z(A)$ denotes the center of $A$ (resp. $Z(A_0)$ denotes the center of $A_0$), then $Z(A)=Z(A_0)\otimes_{\Z_p}\Q_p$.
\end{lem}
\justify
Let us consider $\Db_{\tilde{X}_\Q}$ the usual sheaf of differential operators \cite{Grothendieck4, Berthelot2} on $\tilde{X}_\Q:=\tilde{X}\times_{\text{Spec}(\Z_p)}\text{Spec}(\Q_p)$ (resp. $X_\Q:=X\times_{\text{Spec}(\Z_p)}\text{Spec}(\Q_p)$ and $\T_{\Q}=\T\times_{\text{Spec}(\Z_p)}\text{Spec}(\Q_p)$). By \cite[Part I, 2.10 (3)]{Jantzen} we have
\begin{eqnarray} \label{Base_Change}
H^{0}\left(X_\Q,(\xi\times_{\Z_p} id_{\Q_p})_*\Db_{\tilde{X}_\Q}\right)^{\T_\Q} = H^{0}(\tilde{X}_\Q,\Db_{\tilde{X}_\Q})^{\T_\Q} = H^{0}(\widetilde{X},\Db^{(m)}_{\widetilde{X}})^{\T}\otimes_{\Z_p}\Q_p.
\end{eqnarray}
On the other hand, we know by \ref{morp.HS} that the right $\T$-action on $\tilde{X}$ induces a canonical morphism of filtered $\Z_p$-algebras
\begin{eqnarray*}
\Phi^{(m)}_{\T}: D^{(m)}(\T)\rightarrow H^{0}(\tilde{X},\Db^{(m)}_{\tilde{X}})
\end{eqnarray*}
and by \cite[page 7]{BG} $\Phi^{(m)}_{\T}\otimes_{\Z_p}\Q_p$ factors through the center of $H^{0}(X_\Q,(\xi\times_{\Z_p} id_{\Q_p})_*\Db_{\tilde{X}_\Q})^{\T_\Q}$. By (\ref{Base_Change}) and the preceding lemma we have the following morphism
\begin{eqnarray*}
D^{(m)}(\T)\hookrightarrow \Ub(\mathfrak{t}_\Q)\xrightarrow{\Phi^{(m)}_{\T}\otimes_{\Z_p}\Q_p} Z\left(H^{0}(\tilde{X},\Db_{\tilde{X}}^{(m)})^{\T}\right)\otimes_{\Z_p}\Q_p
\end{eqnarray*}
(we recall for the reader that $\mathfrak{t}_\Q:=\text{Lie}(\T)\otimes_{\Z_p}\Q_p$ and that $D^{(m)}(\T)\otimes_{\Z_p}\Q_p=\Ub(\mathfrak{t}_\Q)$, for every $m\in\N$). Following the same lines of reasoning that in page \pageref{Reasoning} we can conclude that $\Phi^{(m)}_{\T}$ induces a morphism of filtered $\Z_p$-algebras 
\begin{eqnarray*}
\Phi^{(m)}_{\T}: D^{(m)}(\T)\rightarrow H^{0}(X, Z(\widetilde{\Db^{(m)}})).
\end{eqnarray*}
Here $Z(\widetilde{\Db^{(m)}})$ is the center of $\widetilde{\Db^{(m)}}$ and its filtration is the one induced by (\ref{Fil.Invariants}). We have the following definition.

\begin{defi} \label{TDO 2}
Let $\lambda: D^{(m)}(\T)\rightarrow\Z_p$ be an integral character. We define the sheaf of level $m$ integral twisted arithmetic differential operators $\Db^{(m)}_{X,\lambda}$ on the flag scheme $X$ by
\begin{eqnarray*}
\Db^{(m)}_{X,\lambda}:=\widetilde{\Db^{(m)}}\otimes_{D^{(m)}(\T)}\Z_p.
\end{eqnarray*} 
\end{defi}
\justify
If we endow $\Z_p$ with the \textit{trivial filtration} as a $D^{(m)}(\T)$-module, this means $0=:\text{F}_{-1}\Z_p$ and $\text{F}_{i}\Z_p:=\Z_p$ for all $i\ge 0$, then using (\ref{Fil.Invariants}) we can view $\Db^{(m)}_{X,\lambda}$ as a sheaf of filtered $\Z_p$-algebras, equipped with the tensor product filtration. 

\begin{prop}\label{why a TDO}
Let $U\in\Sb$. Then $\Db^{(m)}_{X,\lambda}|_{U}$ is isomorphic to $\Db^{(m)}_X|_U$ as a sheaf of filtered $\Z_p$-algebras.
\end{prop}
\begin{proof}
Let us recall that by proposition \ref{locally tilde} for every $U\in\Sb$ we have an isomorphism of filtered $\Z_p$-algebras
\begin{eqnarray*}
\widetilde{\Db^{(m)}}|_U\xrightarrow{\simeq} \Db^{(m)}_X|_U\otimes_{\Z_p} D^{(m)}(\T)
\end{eqnarray*}
which induces an isomorphism $\Db^{(m)}_{X,\lambda}|_{U}\xrightarrow{\simeq}\Db^{(m)}_{X}|_{U}$ of filtered $\Z_p$-algebras. 
\end{proof}
\begin{rem}
This proposition justifies the name of "twisted arithmetic differential operators".
\end{rem}
\justify
Let us recall that, as $X$ is a smooth $\Z_p$-scheme, the sheaf of Berthelot's differential operators $\Db^{(m)}_X$ is a sheaf of $\Ob_X$-rings with noetherian sections over all open affine subsets of $X$ \cite[corollary 2.2.5]{Berthelot1}. As $\xi$ is locally trivial, the family $\Sb$ forms a base for the Zariski topology of $X$ and therefore the preceding proposition and the same reasoning given in \cite[Proposition 2.2.2 (iii)]{HSS} imply the following meaningful result.

\begin{prop}\label{prop 3.2.10}
The sheaf $\Db_{X,\lambda}^{(m)}$ is a sheaf of $\Ob_X$-rings with noetherian sections over all open affine subsets of $X$.
\end{prop}

\begin{defi}\label{arithmetic TDO}
We will denote by 
\begin{eqnarray}
\widehat{\Da}^{(m)}_{\mathfrak{X},\lambda}:= \displaystyle \varprojlim_{j}\Db_{X,\lambda}^{(m)}/p^{j+1}\Db_{X,\lambda}^{(m)}
\end{eqnarray}
the $p$-adic completion of $\Db^{(m)}_{X,\lambda}$ and we consider it as a sheaf on $\mathfrak{X}$. Following the notation given at the beginning of this work, the sheaf $\widehat{\Da}^{(m)}_{\mathfrak{X},\lambda,\Q}$ will denote our sheaf of level m twisted arithmetic differential operators on the formal flag scheme $\mathfrak{X}$.
\end{defi}

\begin{prop}\label{prop coherence }
\begin{itemize}
\item[(i)] There exists a basis $\Ba$ of the topology of $\mathfrak{X}$, consisting of open affine subsets, such that for every $\mathfrak{U}\in\Ba$ the ring $\widehat{\Da}^{(m)}_{\mathfrak{X},\lambda}(\mathfrak{U})$ is twosided noetherian.
\item[(ii)] The sheaf of rings $\widehat{\Da}^{(m)}_{\mathfrak{X},\lambda,\Q}$ is coherent. 
\end{itemize}
\end{prop}

\begin{proof}
To show (i) we can take an open affine subset $U\in\Sb$ and to consider $\mathfrak{U}$ its formal completion along the special fiber. We have
\begin{eqnarray*}
H^{0}(\mathfrak{U},\widehat{\Da}^{(m)}_{ \mathfrak{X},\lambda})
\simeq \widehat{H^{0}(U,\Db^{(m)}_{X,\lambda})}
\simeq \widehat{H^{0}(U,\Db^{(m)}_{X})}
\simeq H^{0}(\mathfrak{U},\widehat{\Da}^{(m)}_{\mathfrak{X}})
\end{eqnarray*}
The first and third isomorphism are given by \cite[($0_I$, 3.2.6)]{Grothendieck1} and the second one arises from the preceding proposition. By \cite[3.2.3 (iv)]{Berthelot1} the ring $H^{0}(\mathfrak{U},\widehat{\Da}^{(m)}_{\mathfrak{X}})$ is twosided noetherian. Therefore, we can take $\Ba$ as the set of affine open subsets of $\mathfrak{X}$ contained in the $p$-adic completion of an affine open subset $U\in\Sb$. This proves (i). By \cite[proposition 3.3.4]{Berthelot1} we can conclude that (ii) is an immediately consequence of (i) because $H^{0}(\mathfrak{U},\widehat{\Da}_{\mathfrak{X},\lambda,\Q}^{(m)})=H^{0}(\mathfrak{U},\widehat{\Da}_{\mathfrak{X},\lambda}^{(m)})\otimes_{\Z_p}\Q_p$ \cite[(3.4.0.1)]{Berthelot1}.
\end{proof}
\justify
Using the morphism $\Phi^{(m)}_X$ defined in (\ref{Can.Morp2}) and the canonical projection from $\widetilde{\Db^{(m)}}$ onto $\Db^{(m)}_{X,\lambda}$ ( $\lambda$ is $\Z_p$-linear) we can define a canonical map 
\begin{eqnarray}\label{c. map}
\Phi^{(m)}_{X,\lambda}:\Ab^{(m)}_X\rightarrow \Db^{(m)}_{X,\lambda}.
\end{eqnarray}

\begin{prop}\label{canonical 2}
\begin{itemize}
\item[(i)] There exists a canonical isomorphism $Sym^{(m)}(\Tb_X)\simeq gr_{\bullet}(\Db^{(m)}_{X,\lambda})$.
\item[(ii)] The canonical morphism $\Phi^{(m)}_{X,\lambda}$ is surjective. 
\item[(iii)] The sheaf $\Db^{(m)}_{X,\lambda}$ is a coherent $\Ab^{(m)}_X$-module.
\end{itemize}
\end{prop}
\begin{proof}
By (\ref{Gr_morph}) we have a canonical map 
\begin{eqnarray*}
gr_{\bullet}\left(\widetilde{\Db^{(m)}}\right)\otimes_{gr_{\bullet}(D^{(m)}(\T))}gr_{\bullet}(\Z_p)\rightarrow gr_{\bullet}\left(\Db_{X,\lambda}^{(m)}\right).
\end{eqnarray*} 
By proposition \ref{gr tilde} we know that $gr_{\bullet}(\widetilde{\Db^{(m)}})\simeq Sym^{(m)}((\xi_{*}\Tb_{\widetilde{X}})^{\T})$. Moreover, by definition, we know that $gr_{\bullet}(\Z_p)=\Z_p$ as a $gr_{\bullet}(D^{(m)}(\T))\;(=Sym^{(m)}(\mathfrak{t})$ proposition \ref{gr.Dist algebra})-module. We obtain a morphism of sheaves of graded $\Z_p$-algebras
\begin{eqnarray*}
Sym^{(m)}\left((\xi_{*}\Tb_{\widetilde{X}})^{\T}\right)\otimes_{Sym^{(m)}(\mathfrak{t})}\Z_p\rightarrow gr_{\bullet}\left(\Db^{(m)}_{X,\lambda}\right)
\end{eqnarray*}
(the structure of $\text{Sym}^{(m)}(\mathfrak{t})$-module is guaranteed by (\ref{Sym_1})). Using the short exact sequence $0\rightarrow \Ob\otimes_{\Z_p}\mathfrak{t}\rightarrow \left(\xi_{*}\Tb_{\widetilde{X}}\right)^{\T}\xrightarrow{\nu} \Tb_{X}\rightarrow 0$ of lemma \ref{Exact for tangent} we see that 
\begin{eqnarray*}
Sym^{(m)}((\xi_{*}\Tb_{\widetilde{X}})^{\T})\otimes_{Sym^{(m)}(\mathfrak{t})}\Z_p\xrightarrow{Sym^{(m)}(\nu)\otimes 1}Sym^{(m)}(\Tb_{X}) 
\end{eqnarray*}
is an isomorphism and we get a canonical morphism of $\Z_p$-algebras
\begin{eqnarray*}
\varphi: Sym^{(m)}(\Tb_{X})\rightarrow gr_{\bullet}\left(\Db_{X,\lambda}^{(m)}\right)
\end{eqnarray*}
By proposition \ref{why a TDO}, we have a commutative diagram for any $U\in\Sb$ 
\begin{eqnarray*}
\begin{tikzcd}[column sep = small]
\text{Sym}^{(m)}\left(\Tb_{X}(U)\right) \arrow{rr}{\varphi_{U}} \arrow[swap]{dr}{}& &\text{gr}_{\bullet}\left(\Db_{X,\lambda}^{(m)}(U)\right) \arrow{dl}{}\\
& \text{gr}_{\bullet}\left(\Db_{X}^{(m)}(U)\right), & 
\end{tikzcd}
\end{eqnarray*}
\justify
here the left diagonal arrow is given by (\ref{graded}). As $\Sb$ is a basis for the Zariski topology of $X$ we can conclude that $\varphi$ is an isomorphism.
\justify 
 For the second claim we can calculate $gr_{\bullet}(\Phi^{(m)}_{X,\lambda})$. By the first part of the proof and proposition \ref{prop 1.4.1} this morphism is identified with 
\begin{eqnarray*}
\Ob_{X}\otimes_{\Z_p}Sym^{(m)}(\mathfrak{g})\rightarrow Sym^{(m)}(\Tb_X)
\end{eqnarray*}
which is surjective by \cite[Proposition 1.6.1]{Huyghe2}. Finally, item $(iii)$ follows from $(ii)$ and proposition \ref{prop 1.4.1}. 
\end{proof}

\begin{rem}\label{rem 5.1.18}
\begin{itemize}
\item[(a)] By construction $\Db^{(m)}_{X,\lambda,\Q}= \Db_{\lambda}$ is the sheaf of usual $\lambda+\rho$\footnote{Here the $\lambda$ on the right is the character of the $\Q_p$-Lie algebra $\mathfrak{t}_\Q$ induced via the correspondence (\ref{Iso_chars_Z_p}).}-twisted differential operators on the flag variety $X_\Q$ \cite[page 170]{BB1}.
\item[(b)] Let us recall that the regular right action of $\G$ on $\widetilde{X}$ (cf. \ref{Right_actions}) induces a natural map $\Phi_{\lambda}: \Ub(\mathfrak{g}_\Q)\rightarrow H^{0}(X_{\Q},\Db_{\lambda})$. This implies that if $\Phi^{(m)}_{\lambda}$ denotes the canonical map induced by $\Phi^{(m)}_{X,\lambda}$ by taking global sections, then $\Phi^{(m)}_{\lambda}\otimes_{\Z_p} \Q_p=\Phi_{\lambda}$ \cite[Page 170 and 186]{BB1}.
\end{itemize}
\end{rem}
\justify
The relation given in (\ref{Iso_chars_Z_p}) tells us that, as in the classical case (see for example  \cite{BB_J} or \cite{BB2}), the sheaf $\widetilde{\Db^{(m)}}:=\left(\xi_*\Db^{(m)}_{\widetilde{X}}\right)^{\T}$ can be regarded as a family of twisted differential operators on $X$ parametrized by $\mathfrak{t}^*:=\text{Hom}(\mathfrak{t},\Z_p)$. 

\begin{char_alg}\label{char_alg}
Before investigating the finiteness properties of these sheaves, let us study the particular case when $\lambda\in\mathfrak{t}^*$ is an algebraic character. This means, $\lambda$ is obtained by differentiating a character of $\T$\footnote{By abusing notation, we can rephrase this by $\lambda =\text{d}\lambda\in\mathfrak{t}^*$, with $\lambda\in\text{Hom}(\T,\G_m)$.}. As is shown in \cite[Part I, chapter 5, (5.8)]{Jantzen}, the character $\lambda\in\text{Hom}(\T,\G_m)$ induces an invertible sheaf $\Lb (\lambda)$ on $X$ and we can consider the sheaf of (integral) differential operators acting on $\Lb (\lambda)$
\begin{eqnarray*}
\Db^{(m)}_{X}(\lambda) := \Lb(\lambda)\otimes_{\Ob_{X}}\Db^{(m)}_{X}\otimes_{\Ob_{X}}\Lb(\lambda)^{\vee}.
\end{eqnarray*}
In fact the (left) action of $\Db^{(m)}_{X}(\lambda)$ on $\Lb (\lambda)$ is defined by
\begin{eqnarray*}
\left(t\otimes P\otimes t^{\vee}\right)\bullet s :=\left(P\bullet\left<t^{\vee},s \right>\right) t
\hspace{1 cm}
(s,t\in\Lb (\lambda),\;\text{and}\; t^{\vee}\in\Lb(\lambda)^{\vee}).
\end{eqnarray*}
\justify
Let us denote by $\widehat{\Da}^{(m)}_{\mathfrak{X}}(\lambda)$ its $p$-adic completion (which is considered as a sheaf on $\mathfrak{X}$) and by $\Da^{\dag}_{\mathfrak{X}}(\lambda)$ its inductive limit tensored with $\Q_p$\footnote{The inductive limit is formed by tensoring the maps \cite[(2.2.1.5)]{Berthelot1} on the left with $\Lb(\lambda)$ and on the right with $\Lb(\lambda)^{\vee}$.}. These sheaves have been studied in \cite{Huyghe2} and \cite{HS2}. As before, we put $\widehat{\Da}^{(m)}_{\mathfrak{X},\Q}(\lambda) :=\widehat{\Da}^{(m)}_{\mathfrak{X}}(\lambda)\otimes_{\Z_p}\Q_p$.
\end{char_alg}
\justify
 For the next result, we will abuse of the notation and we will suppose that $\lambda\in\text{Hom}(\T,\G_m)$ is an algebraic character of $\mathfrak{t}$ and that $\lambda ' :\text{Dist}(\T)\rightarrow\Z_p$ denotes the character of the distribution algebra $\text{Dist}(\T)$ induced by the bijection (\ref{Iso_chars_Z_p}).

\begin{prop}\label{HSS}
The sheaves $\widehat{\Da}^{(m)}_{\mathfrak{X},\Q}(\lambda)$ and $\widehat{\Da}^{(m)}_{\mathfrak{X},\lambda',\Q}$ are canonically isomorphic. 
\end{prop}
\justify
Before starting the proof let us recall the following facts\footnote{This facts are detailed in \cite{Huyghe2}.}. First of all, the order filtration of $\Db^{(m)}_{X}$ (definition \ref{Diff order n}) induces a filtration $\left(\Db^{(m)}_{X,d}(\lambda)\right)_{d\in\N}$ of $\Db^{(m)}_{X}(\lambda)$ by
\begin{eqnarray*}
\Db^{(m)}_{X,d}(\lambda) := \Lb (\lambda)\otimes_{\Ob_{X}}\Db^{(m)}_{X,d}\otimes_{\Ob_{X}}\Lb (\lambda)^{\vee},
\end{eqnarray*}
such that 
\begin{eqnarray*}
\text{gr}_{\bullet}\left(\Db^{(m)}_{X}(\lambda)\right)\simeq \text{Sym}^{(m)}\left(\Tb_{X}\right).
\end{eqnarray*}
It is easy to see that in order to achieve the preceding isomorphism we only need to use (\ref{graded}) and the commutativity of the symmetric algebra.
On the other hand, by the work developed by Huyghe-Schmidt in \cite{HS2} we have a canonical morphism of filtered $\Z_p$-algebras
\begin{eqnarray*}
\Phi^{(m)}:\Ab^{(m)}_{X} := \Ob_{X}\otimes_{\Z_p}D^{(m)}(\G) \rightarrow \Db^{(m)}_{X}(\lambda)
\end{eqnarray*}
which is in fact surjective thanks to the preceding isomorphism and the fact that $X$ is an homogeneous space (cf. \cite[Section 1.6]{Huyghe2}).

\begin{proof}[Proof of proposition \ref{HSS}]
By the preceding discussion and proposition \ref{canonical 2} we have two canonical surjective morphisms of filtered $\Z_p$-algebras
\begin{eqnarray*}
\Phi^{(m)}:  \Ab^{(m)}_{X}  \rightarrow
\Db^{(m)}_{X}(\lambda)  
\hspace{0.5 cm}
\text{and}
\hspace{0.5 cm}
\Phi^{(m)}_{X,\lambda'} : \Ab^{(m)}_{X} \rightarrow \Db^{(m)}_{X,\lambda '}.
\end{eqnarray*}
In particular, there exist two-sided ideals $\Jb^{(m)}_1$ and $\Jb^{(m)}_{2}$ of $\Ab^{(m)}_{X}$ such that
\begin{eqnarray*}
\Db^{(m)}_{X,\lambda '} = \Ab^{(m)}_{X}\big / \Jb^{(m)}_1 
\hspace{0.5 cm} 
\text{and}
\hspace{0.5 cm}
\Db^{(m)}_{X}(\lambda) = \Ab^{(m)}_{X}\big / \Jb^{(m)}_2.
\end{eqnarray*} 
Moreover, since 
\begin{eqnarray*}
\Db^{(m)}_{X,\lambda '}\otimes_{\Z_p}\Q_p = \Db^{(m)}_{X}(\lambda) \otimes_{\Z_p}\Q_p = \Db_{\lambda}
\end{eqnarray*}
we have $\Phi^{(m)}\otimes_{\Z_p}\Q_p = \Phi^{(m)}_{X,\lambda '}\otimes_{\Z_p}\Q_p$ and therefore  
\begin{eqnarray}\label{Equal_kernels}
 \Jb^{(m)}_{2,\Q} = \text{Ker} (\Phi^{(m)}\otimes_{\Z_p}\Q_p ) = \text{Ker}(\Phi^{(m)}_{X,\lambda'}\otimes_{\Z_p}\Q_p) = \Jb^{(m)}_{1,\Q}.
\end{eqnarray}
\textit{Claim:} There exist $N_1, N_2\in\N$ such that $p^{N_1}\Jb^{(m)}_1\subseteq \Jb^{(m)}_2$ and $p^{N_2}\Jb^{(m)}_{2}\subseteq \Jb^{(m)}_1$.
\begin{proof}
Let $U\subseteq X$ be an affine open subset. Since $\Ab^{(m)}_{X}$ has noetherian sections over all affine open subsets (proposition \ref{prop 1.4.1}) we can find $t_1,\; ...,\; t_l\in \Jb^{(m)}_{2}(U)$ such that $\frac{t_1}{1},\; ...,\; \frac{t_l}{1}\in \Jb^{(m)}_2(U)\otimes_{\Z_p}\Q_p$ generate $\Jb^{(m)}_{2}(U)\otimes_{\Z_p}\Q_p$ as $\Ab^{(m)}_{X}(U)\otimes_{\Z_p}\Q_p$-module. Let $s_1,\; ...,\; s_d\in\Jb^{(m)}_1(U)$ a set of generators of $\Jb^{(m)}_1(U)$ as $\Ab^{(m)}_{X}(U)$-module. By (\ref{Equal_kernels}) we can find $a_{11},\; ...,\; a_{dl}\in \Ab^{(m)}_{X}(U)\otimes_{\Z_p}\Q_p$ such that
\begin{eqnarray*}
\begin{pmatrix} 
\frac{s_1}{1} \\
\cdot \\
\cdot \\
\cdot \\
\frac{s_d}{1} 
\end{pmatrix}
=
\begin{pmatrix} 
a_{11}   & \cdot & \cdot & \cdot & a_{1l} \\
\cdot     & \cdot &         &         & \cdot \\
\cdot     &         & \cdot &         & \cdot \\
\cdot     &         &         & \cdot & \cdot \\
a_{d1}   & \cdot & \cdot & \cdot & a_{dl}
\end{pmatrix}
\begin{pmatrix}
\frac{t_1}{1} \\
\cdot \\
\cdot \\
\cdot \\
\frac{t_l}{1} 
\end{pmatrix},
\hspace{0.5 cm}
A := \begin{pmatrix} 
a_{11}   & \cdot & \cdot & \cdot & a_{1l} \\
\cdot     & \cdot &         &         & \cdot \\
\cdot     &         & \cdot &         & \cdot \\
\cdot     &         &         & \cdot & \cdot \\
a_{d1}   & \cdot & \cdot & \cdot & a_{dl}
\end{pmatrix}
\in \Mb_{d\times l }\left(\Ab^{(m)}_{X}(U)\otimes_{\Z_p}\Q_p\right).
\end{eqnarray*}
In particular, we can find $u\in \Z_{p}^{\times}$ and $N\in\N$ such that $up^{N} A \in \Mb_{d\times l}(\Ab^{(m)}_{X}(U))$, and therefore $p^{N}s_j\in \Jb^{(m)}_{2}(U)$, for all $1\le j \le d$. Now, as $X$ is compact, we can find a finite affine covering $X=\cup_{1\le i \le k} U_i$ and we take $N_1:=\text{max}\left\{N_i|\; 1\le i\le k\right\}$, such that $p^{N_1}\Jb^{(m)}_1\subseteq \Jb^{(m)}_{2}$. By applying the same reasoning, we can see that there exists $N_2\in\N$ such that $p^{N_2}\Jb^{(m)}_2 \subseteq \Jb^{(m)}_1$.
\end{proof}
\justify
Let us denote by $\widehat{\Ab}^{(m)}_{\mathfrak{X}}$, $\widehat{\Jb}^{(m)}_1$ and $\widehat{\Jb}^{(m)}_2$ the respective $p$-adic completions (considered as sheaves on $\mathfrak{X}$). By the preceding claim, we have the following diagram
\begin{eqnarray*}
\begin{tikzcd}[row sep= large]
\widehat{\Ab}^{(m)}_{\mathfrak{X},\Q}\big / \widehat{\Jb}^{(m)}_{1,\Q} = \widehat{\Ab}^{(m)}_{\mathfrak{X},\Q}\big /p^{N_1} \widehat{\Jb}^{(m)}_{1,\Q} \arrow{r}{} \arrow[swap]{dr}{id} & \widehat{\Ab}^{(m)}_{\mathfrak{X},\Q}\big / \widehat{\Jb}^{(m)}_{2,\Q} = \widehat{\Ab}^{(m)}_{\mathfrak{X},\Q}\big /p^{N_2} \widehat{\Jb}^{(m)}_{2,\Q} \arrow{d}{} \\
     & \widehat{\Ab}^{(m)}_{\mathfrak{X},\Q}\big / \widehat{\Jb}^{(m)}_{1,\Q}.
\end{tikzcd}
\end{eqnarray*}
Given that $\Ab^{(m)}_{X}$ has noetherian sections over all open affine subsets, the $p$-adic completion is in this case an exact functor and therefore 
\begin{eqnarray*}
\widehat{\Da}^{(m)}_{\mathfrak{X},\lambda ',\Q} \simeq \widehat{\Ab}^{(m)}_{\mathfrak{X},\Q}\big / \widehat{\Jb}^{(m)}_{1,\Q} \simeq \widehat{\Ab}^{(m)}_{\mathfrak{X},\Q}\big / \widehat{\Jb}^{(m)}_{2,\Q} = \widehat{\Da}^{(m)}_{\mathfrak{X},\Q}(\lambda).
\end{eqnarray*}
\end{proof}

\subsection{Finiteness properties}
Let $\lambda:\text{Dist}(\T)\rightarrow\Z_p$ be a character. In this section we start the study of the cohomological properties of coherent   $\Db^{(m)}_{X,\lambda}$-modules. We follow the arguments of \cite{Huyghe2} to show a technical important finiteness property about the $p$-torsion of the cohomology groups of coherent  $\Db^{(m)}_{X,\lambda}$-modules, when the character $\lambda + \rho\in\mathfrak{t}_\Q^*$ is dominant and regular (proposition \ref{prop 2.2.3}). To start with, let us recall the twist by the sheaf $\Ob(1)$. As $X$ is a projective  $\Z_p$-scheme, there exists a very ample invertible sheaf $\Ob(1)$ on $X$ \cite[chapter II, remark 5.16.1]{Hartshorne1}. Therefore, for any arbitrary $\Ob_{X}$-module $\Eb$ we can consider the twist
\begin{eqnarray*}
\Eb(r):=\Eb\otimes_{\Ob_{X}}\Ob(r),
\end{eqnarray*}
where $r\in\Z$ and  $\Ob(r)$ means the $r$-th tensor product of $\Ob(1)$ with itself. We recall to the reader that there exists $r_0\in\Z$, depending of $\Ob(1)$, such that for every $k\in\Z_{>0}$ and for every $s\ge r_0$, $H^{k}(X,\Ob(s))=0$ \cite[chapter II, theorem 5.2 (b)]{Hartshorne1}. 
\justify
We start the results of this section with the following proposition which states three important properties of coherent $\Ab_{X}^{(m)}$-modules \cite[proposition A.2.6.1]{HS1}. This is a key result in this work. Let  $\Eb$ be a coherent $\Ab_{X}^{(m)}$-module.
\begin{prop}\label{prop 1.4.2}
\begin{itemize}
\item[(i)] $H^0(X,\Ab_{X}^{(m)})=D^{(m)}(\G)$ is a noetherian $\Z_p$-algebra.
\item[(ii)] There exists a surjection of $\Ab_{X}^{(m)}$-modules $\left(\Ab_{X}^{(m)}(-r)\right)^{\oplus a}\rightarrow\Eb\rightarrow 0$ for suitable $r\in\Z$ and $a\in\N$.
\item[(iii)] For any $k\ge 0$ the group $H^k(X,\Eb)$ is a finitely generated $D^{(m)}(\G)$-module.
\end{itemize} 
\end{prop}
\justify
Inspired in proposition \ref{canonical 2}, in a first time we will be concentrated on coherent $\Ab_{X}^{(m)}$-modules. The next two results will play an important role when we consider formal completions.
\begin{lem}\label{lem 2.2.1}
For every coherent $\Ab_{X}^{(m)}$-module $\Eb$, there exists $r=r(\Eb)\in\Z$ such that $H^k(X,\Eb(s))=0$ for every $s\ge r$.
\end{lem}
\begin{proof}
Let us fix $r_0\in\Z$ such  that $H^{k}(X,\Ob(s))=0$ for every $k>0$ and $s\ge r_0$. We have,
\begin{eqnarray*}
H^k(X,\Ab_X^{(m)}(s))= H^k(X,\Ob(s))\otimes_{\Z_p}D^{(m)}(\G)=0.
\end{eqnarray*}
Now, by the second part of proposition \ref{prop 1.4.2} there exist $a_0\in\N$ and $s_0\in\Z$ together with an epimorphism of $\Ab_X^{(m)}$-modules 
\begin{eqnarray*}
\Eb_0:=\left(\Ab_X^{(m)}(s_0)\right)^{\oplus a_0}\rightarrow\Eb\rightarrow 0.
\end{eqnarray*} 
If $r\ge r_0-s_0$ we see that $H^{k}(X,\Eb_0(r))=0$. Now we way use this relation and the preceding proposition to follow word by word the reasoning given in \cite[Proposition 2.2.1]{Huyghe2} to conclude the lemma.
\end{proof}

\begin{lem}\label{lem 2.2.2}
For every coherent $\Db^{(m)}_{X,\lambda}$-module $\Eb$, there exist $r=r(\Eb)\in\Z$, a natural number $a\in\N$ and an epimorphism of $\Db^{(m)}_{X,\lambda}$-modules
\begin{eqnarray*}
\left(\Db^{(m)}_{X,\lambda}(-r)\right)^{\oplus a}\rightarrow \Eb\rightarrow 0.
\end{eqnarray*} 
\end{lem}
\begin{proof}
Using the epimorphism in proposition \ref{canonical 2} we can suppose that $\Eb$ is also a coherent $\Ab_{X}^{(m)}$-module. In this case, by the second part of proposition \ref{prop 1.4.2}, there exist $r=r(\Eb)\in\Z$, a natural number $a\in\N$ and an epimorphism of $\Ab_{X}^{(m)}$-modules 
\begin{eqnarray*}
\left(\Ab_{X}^{(m)}(-r)\right)^{\oplus a}\rightarrow\Eb\rightarrow 0.
\end{eqnarray*}
Taking the tensor product with $\Db^{(m)}_{X,\lambda}$ we get the desired epimorphism of $\Db^{(m)}_{X,\lambda}$-modules
\begin{eqnarray*}
\left(\Db^{(m)}_{X,\lambda}(-r)\right)^{\oplus a}\simeq\Db^{(m)}_{X,\lambda}\otimes_{\Ab_{X}^{(m)}}\left(\Ab_{X}^{(m)}(-r)\right)^{\oplus a}\rightarrow \Db^{(m)}_{X,\lambda}\otimes_{\Ab_{X}^{(m)}}\Eb\simeq\Eb\rightarrow 0.
\end{eqnarray*}
\end{proof}
\justify
We recall to the reader that the distribution algebra of level $m$, which has been denoted by $D^{(m)}(\G)$ in subsection \ref{ADA},  it is a filtered noetherian $\Z_p$-algebra. This finiteness property is essential in the following proposition which we (personal) consider as the heart of this paper\footnote{We follow the same argument given in \cite{Huyghe1}.}.
\justify
\textbf{\textsc{notation}}: \textit{In the sequel we will refer to a character  $\lambda\in\mathfrak{t}^*_\Q$ as an $\Q_p$-linear application induced, via base change, by a character $\lambda:\text{Dist}(\T)\rightarrow\Z_p$ of the distribution algebra of the torus $\T$.}

\begin{prop}\label{prop 2.2.3} Let us suppose that $\lambda + \rho \in\mathfrak{t}_\Q^*$ is a dominant and regular character (definition \ref{character of the dist. Alg}).
\begin{itemize} 
\item[(i)] Let us fix $r\in\Z$. For every positive integer $k\in\Z_{>0}$, the cohomology group $H^{k}(X,\Db^{(m)}_{X,\lambda}(r))$ has bounded $p$-torsion.
\item[(ii)] For every coherent $\Db^{(m)}_{X,\lambda}$-module $\Eb$, the cohomology group $H^{k}(X,\Eb)$ has bounded $p$-torsion for all $k> 0$.
\end{itemize}
\end{prop}
\begin{proof}
To show $(i)$, we recall that $\Db^{(m)}_{X,\lambda,\Q}=\Db_{\lambda}$ is the usual sheaf of twisted differential operators on the flag variety $X_\Q$ (remark \ref{rem 5.1.18}). As $\Db^{(m)}_{X,\lambda,\Q}(r)$ is a coherent $\Db_{\lambda}$-module, the classical Beilinson-Bernstein theorem \cite{BB} allows us to conclude that $H^k(X,\Db^{(m)}_{X,\lambda}(r))\otimes_{\Z_p}\Q_p=0$ for every positive integer $k\in\Z_{>0}$. This in particular implies that the sheaf $\Db^{(m)}_{X,\lambda}(r)$ has $p$-torsion cohomology groups $H^k(X,\Db^{(m)}_{X,\lambda}(r))$, for every $k> 0$ and $r\in\Z$.\\
Now, by proposition \ref{canonical 2}, we know that $\Db^{(m)}_{X,\lambda}(r)$ is in particular a coherent $\Ab_X^{(m)}$-module and hence, by the third part of the proposition \ref{prop 1.4.2} we get that for every $k\ge 0$ the cohomology groups $H^k(X,\Db^{(m)}_{X,\lambda}(r))$ are finitely generated $D^{(m)}(\G)$-modules. Consequently, of finite $p$-torsion for every  integer $0<k\le \text{dim}(X)$ and $r\in\Z$. 
\justify
To show $(ii)$ the reader can follow the same arguments exhibited in \cite[Corollary 2.2.2]{Huyghe2} or \cite[Proposition 4.1.19]{HPSS}.
\end{proof}

\section{Passing to formal completions}
Let us start by recalling the formal completion of the sheaves introduced in the subsection \ref{Rel_env_alg_on_hs}. Let $\lambda\in\mathfrak{t}_\Q^*$ be a character of the Cartan subalgebra $\mathfrak{t}_\Q$ (we recall for the reader that this means an $\Q_p$-linear map induced by a character of $\text{Dist}(\T)$). We have introduced the following sheaves of $p$-adically complete $\Z_p$-algebras on the formal $p$-adic completion $\mathfrak{X}$ of $X$
\begin{eqnarray*}
\widehat{\Da}^{(m)}_{\mathfrak{X},\lambda}:=\displaystyle\varprojlim_j \Db^{(m)}_{X,\lambda}/p^{j+1}\Db^{(m)}_{X,\lambda}.
\end{eqnarray*}
The sheaf $\widehat{\Da}^{(m)}_{\mathfrak{X},\lambda,\Q}$ is our sheaf of \textit{level m twisted arithmetic differential operators} on the smooth formal flag scheme $\mathfrak{X}$.
\subsection{Cohomological properties}
From now on, we will suppose that  $\lambda + \rho\in\mathfrak{t}_\Q^*$ is a dominant and regular character of $\mathfrak{t}_\Q$ (to see \ref{roots} and definition \ref{character of the dist. Alg}). Our objective in this subsection is to prove an analogue of proposition \ref{prop 2.2.3} for coherent $\widehat{\Da}_{\mathfrak{X},\lambda}^{(m)}$-modules and to conclude that $H^0(\mathfrak{X},\bullet)$ is an exact functor over the category of coherent $\widehat{\Da}^{(m)}_{\mathfrak{X},\lambda,\Q}$-modules. 

\begin{prop}\label{prop 3.1.1}
Let $\Eb$ be a coherent $\Db^{(m)}_{X,\lambda}$-module and $\widehat{\Eb}:=\varprojlim_j\Eb/p^{j+1}\Eb$ its $p$-adic completion, which we consider as a sheaf on $\mathfrak{X}$.
\begin{itemize}
\item[(i)] For all $k\ge 0$ one has $H^{k}(\mathfrak{X},\widehat{\Eb})=\varprojlim_j H^{k}(X_j,\Eb/p^{j+1}\Eb)$.
\item[(ii)] For all $k>0$ one has $H^{k}(\mathfrak{X},\widehat{\Eb})=H^{k}(X,\Eb)$.
\item[(iii)] The global section functor $H^{0}(\mathfrak{X},\bullet)$ satisfies $H^0(\mathfrak{X},\widehat{\Eb})=\varprojlim_j H^0(X,\Eb)/p^{j+1}H^0(X,\Eb)$.
\end{itemize}
\end{prop}
\begin{proof}
Let $\Eb_t$ denote the torsion subpresheaf of $\Eb$. As $X$ is a noetherian space and $\Db^{(m)}_{X,\lambda}$ has noetherian rings sections over open affine subsets of $X$ (proposition \ref{prop 3.2.10}), we can conclude that $\Eb_t$ is in fact a coherent $\Db^{(m)}_{X,\lambda}$-module. This is generated by a coherent $\Ob_X$-module which is annihilated by a power $p^c$ of $p$, and so is $\Eb_t$. The quotient $\Gb:=\Eb/\Eb_t$ is again a coherent $\Db^{(m)}_{X,\lambda}$-module and therefore we can assume, after possibly replacing $c$ by a larger number, that $p^c\Eb_t=0$ and $p^cH^k(X,\Eb)=p^cH^k(X,\Gb)=0$ , for all $k>0$. From here on the proof of the proposition follows the same lines of reasoning that in \cite[proposition 3.2]{Huyghe1}.
\end{proof}
\justify
The next proposition is a natural result from lemmas \ref{lem 2.2.1} and \ref{lem 2.2.2}. Except for some technical details, the proof is exactly the same that in \cite[proposition 4.2.2]{HPSS}. Therefore, we will only give a brief sketch of it. 

\begin{prop}\label{prop 3.1.2}
Let $\Ea$ be a coherent $\widehat{\Da}^{(m)}_{\mathfrak{X},\lambda}$-module.
\begin{itemize}
\item[(i)] There exists $r_2=r_2(\Ea)\in\Z$ such that, for all $r\ge r_2$ there is $a\in\Z$ and an epimorphism of $\widehat{\Da}^{(m)}_{\mathfrak{X},\lambda}$-modules
\begin{eqnarray*}
\left(\widehat{\Da}^{(m)}_{\mathfrak{X},\lambda}(-r)\right)^{\oplus a}\rightarrow \Ea\rightarrow 0.
\end{eqnarray*}
\item[(ii)] There exists $r_3=r_3(\Ea)\in\Z$ such that, for all $r\ge r_3$ we have $H^i(\mathfrak{X},\Ea(r))=0$, for all $i>0$.
\end{itemize}
\end{prop}
\begin{proof}
We start the proof of the part $(i)$ by remarking that the torsion subsheaf $\Ea_t$ of $\Ea$ is a coherent $\widehat{\Da}^{(m)}_{\mathfrak{X},\lambda}$-module. As $\mathfrak{X}$ is quasi-compact, there exists $c\in\N$ such that $p^c\Ea_t=0$. Defining $\Ga:=\Ea/\Ea_t$, $\Ga_0:=\Ga/\Ga_t$ and $\Ea_j:=\Ea/p^{j+1}\Ea$, we have for every $j\ge c$ an exact sequence 
\begin{eqnarray*}
0\rightarrow \Ga_0\xrightarrow{p^{j+1}}\Ea_{j+1}\rightarrow\Ea_j\rightarrow 0.
\end{eqnarray*}
Viewing $\mathfrak{X}$ as a closed subset of $X$ and denoting by $\psi$ this topological embedding, we can suppose that $\psi_*\Ga_0$ is a coherent $\Db^{(m)}_{X,\lambda}$-module via the canonical isomorphism of sheaves of rings $\Db^{(m)}_{X,\lambda}/p\Db^{(m)}_{X,\lambda}\simeq\psi_*\left(\widehat{\Da}^{(m)}_{\mathfrak{X},\lambda}/p\widehat{\Da}^{(m)}_{\mathfrak{X},\lambda}\right)$ (similarly, we can consider $\Ea_c$ as a coherent $\Db^{(m)}_{X,\lambda}$-module). By using the fact $\Ga_0$ is also a coherent $\Ab_{X}^{(m)}$-module, lemma \ref{lem 2.2.1} gives us an integers $r'_2(\Ga_0)$ such that the canonical maps 
\begin{eqnarray*}
H^0(\mathfrak{X},\Ea_{j+1}(r'))\rightarrow H^{0}(\mathfrak{X},\Ea_{j}(r'))
\end{eqnarray*} 
are surjective for $r'\ge r'_2(\Ga_0)$ and $j\ge c$. Moreover, lemma \ref{lem 2.2.2} gives another integer $r'_3(\Ea_c)$ such that, for every $r''\ge r'_3(\Ea_c)$ there exists $a \in\N$ and a surjection
\begin{eqnarray*}
\phi: \left(\Db^{(m)}_{X,\lambda}/p^c\Db^{(m)}_{X,\lambda}\right)^{\oplus a}\rightarrow \Ea_c(r'')\rightarrow 0.
\end{eqnarray*} 
From the previous morphisms and proceeding as in \cite[proposition 4.2. (i)]{HPSS}, we can define for every $r\ge r_2:=\text{max}\{r'_2(\Ea_c), r'_3(\Ga_0)\}$ the desired surjective morphism
\begin{eqnarray*}
\left(\widehat{\Da}^{(m)}_{\mathfrak{X},\lambda}\right)^{\oplus a}\rightarrow \Ea(r)\rightarrow 0.
\end{eqnarray*}
To show the part $(ii)$, we remark that if we fix $r_0\in\Z$ such that $H^k(X,\Ob(r))=0$ for every $k>0$ and $r\ge r_0$, exactly as we have done in lemma \ref{lem 2.2.1}, then via the second part of proposition \ref{prop 3.1.2} we also have that  $H^k(\mathfrak{X},\widehat{\Da}^{(m)}_{\mathfrak{X},\lambda}(r))=0$ and the rest of the proof can be deduced exactly as in the proof of lemma \ref{lem 2.2.1}.
\end{proof}
\justify
The same inductive argument exhibited in the second part of proposition $\ref{prop 2.2.3}$ shows
\begin{coro}\label{coro 3.1.3}
Let $\Ea$ be a coherent $\widehat{\Da}^{(m)}_{\mathfrak{X},\lambda}$-module. There exists $c=c(\Ea)\in\N$ such that for all $k>0$ the cohomology group $H^k(\mathfrak{X},\Ea)$ is annihilated by $p^c$.
\end{coro}
\justify
Now, we want to extend the part (i) of proposition \ref{prop 3.1.2} to the sheaves $\widehat{\Da}^{(m)}_{\mathfrak{X},\lambda,\Q}$. To do that, we need to show that the category of coherent $\widehat{\Da}^{(m)}_{\mathfrak{X},\lambda,\Q}$-modules admits integral models (definition \ref{defi I.models}).\\
Let $\text{Coh}(\widehat{\Da}^{(m)}_{\mathfrak{X},\lambda})$ be the category of coherent $\widehat{\Da}^{(m)}_{\mathfrak{X},\lambda}$-modules and let $\text{Coh}(\widehat{\Da}^{(m)}_{\mathfrak{X},\lambda})_\Q$ be the category of coherent $\widehat{\Da}^{(m)}_{\mathfrak{X},\lambda}$-modules  up to isogeny. This means that $\text{Coh}(\widehat{\Da}^{(m)}_{\mathfrak{X},\lambda})_\Q$ has the same class of objects as $\text{Coh}(\widehat{\Da}^{(m)}_{\mathfrak{X},\lambda})$ and, for any two objects $\Mb$ and $\Nb$ in Coh$(\widehat{\Da}^{(m)}_{\mathfrak{X},\lambda})_{\Q}$ one has
\begin{eqnarray*}
\text{Hom}_{\text{Coh}(\widehat{\Da}^{(m)}_{\mathfrak{X},\lambda})_{\Q}}(\Mb,\Nb)=\text{Hom}_{\text{Coh}(\widehat{\Da}^{(m)}_{\mathfrak{X},\lambda})}(\Mb,\Nb)\otimes_{\Z_p}\Q_p.
\end{eqnarray*}

\begin{prop}\label{prop 3.1.4}
The functor $\Mb\mapsto\Mb\otimes_{\Z_p}\Q_p$ induces an equivalence of categories between $\text{Coh}(\widehat{\Da}^{(m)}_{\mathfrak{X},\lambda})_\Q$ and $\text{Coh}(\widehat{\Da}^{(m)}_{\mathfrak{X},\lambda,\Q})$.
\end{prop}
\begin{proof}
By definition, the sheaf $\widehat{\Da}^{(m)}_{\mathfrak{X},\lambda,\Q}$ satisfies \cite[conditions 3.4.1]{Berthelot1} and therefore \cite[proposition 3.4.5]{Berthelot1} allows to conclude the proposition.
\end{proof}
\justify
The proof of the next theorem follows exactly the same lines than in \cite[theorem 4.2.8]{HPSS}. We will reproduce the proof because it is a central result for our goal.

\begin{theo}\label{theo 3.1.5}
Let $\Ea$ be a coherent $\widehat{\Da}^{(m)}_{\mathfrak{X},\lambda,\Q}$-module. 
\begin{itemize}
\item[(i)] There is $r(\Ea)\in\Z$ such that, for every $r\ge r(\Ea)$ there exists $a\in\N$ and an  epimorphism of $\widehat{\Da}^{(m)}_{\mathfrak{X},\lambda,\Q}$-modules
\begin{eqnarray*}
\left(\widehat{\Da}^{(m)}_{\mathfrak{X},\lambda,\Q}(-r)\right)^{\oplus a}\rightarrow\Ea\rightarrow 0.
\end{eqnarray*}
\item[(ii)] For all $i>0$ one has $H^i(\mathfrak{X},\Ea)=0$.
\end{itemize}
\end{theo}
\begin{proof}
By the preceding proposition, there exists a coherent $\widehat{\Da}^{(m)}_{\mathfrak{X},\lambda}$-module $\Fa$ such that $\Fa\otimes_{\Z_p}\Q_p\simeq\Ea$. Therefore, applying proposition \ref{prop 3.1.2} to $\Fa$ gives $(i)$. Moreover,  as $\mathfrak{X}$ is a noetherian space, corollary \ref{coro 3.1.3} allows us to conclude that
\begin{eqnarray*}
H^i(\mathfrak{X},\Ea)=H^i(\mathfrak{X},\Fa)\otimes_{\Z_p}\Q_p=0
\end{eqnarray*}
for every $k>0$ \cite[(3.4.0.1)]{Berthelot1}.
\end{proof}
\subsection{Calculation of global sections}
We recall for the reader that throughout this section $\lambda+\rho\in\mathfrak{t}^*_\Q$ denotes a dominant and regular character, which is induced by $\lambda:\text{Dist}(\T)\rightarrow\Z_p$ via the correspondence (\ref{Iso_chars_Z_p}). In this subsection we propose to calculate the global sections of the sheaf $\widehat{\Da}^{(m)}_{\mathfrak{X},\lambda,\Q}$. Inspired in the arguments exhibited in \cite{HS2},  we will need the following lemma ( \cite[lemma 3.3]{HS2}) whose conclusion is an essential tool for our goal.
\begin{lem}\footnote{The proof as in \cite[lemme 3.3]{HS2}.}\label{G.sections}
Let $A$ be a noetherian $\Z_p$-algebra, $M$, $N$ two $A$-modules of finite type, $\psi: M\rightarrow N$ an $A$-lineal application and $\widehat{\psi}: \widehat{M}\rightarrow\widehat{N}$ the morphism obtained after $p$-adic completion. If $\psi\otimes_{\Z_p}\Q_p: M\otimes_{\Z_p}\Q_p\rightarrow N\otimes_{\Z_p}\Q_p$ is an isomorphism, then $\widehat{\psi}\otimes_{\Z_p}1:\widehat{M}\otimes_{\Z_p}\Q_p\rightarrow \widehat{N}\otimes_{\Z_p}\Q_p$ is an isomorphism as well.  
\end{lem}
\begin{proof}
Let $K$ be the kernel (resp. the cokernel) of $\psi$. Since the $p$-adic completion is an exact functor over the finitely generated $A$-modules \cite[3.2.3 (ii)]{Berthelot1}, the $p$-completion $\hat{K}$ is the kernel (resp. the cokernel) of $\hat{\psi}$. But $\hat{K}=K$ because $K$ is of $p$-torsion, and therefore $\hat{K}\otimes_{\Z_p}\Q_p=K\otimes_{\Z_p}\Q_p=0.$
\end{proof}
\justify
Let us identify the universal enveloping algebra $\Ub(\mathfrak{t}_\Q)$ of the Cartan subalgebra $\mathfrak{t}_\Q$ with the symmetric algebra $S(\mathfrak{t}_\Q)$, and let $Z(\mathfrak{g}_\Q)$ denote the center of the universal enveloping algebra $\Ub(\mathfrak{g}_\Q)$ of $\mathfrak{g}_\Q$. The classical Harish-Chandra isomorphism $Z(\mathfrak{g}_\Q)\simeq S(\mathfrak{t}_\Q)^{W}$ (the subalgebra of Weyl invariants) \cite[theorem 7.4.5]{Dixmier}, allows us to define for every linear form $\lambda\in \mathfrak{t}_\Q^*$ a central character \cite[7.4.6]{Dixmier} 
\begin{eqnarray*}
\chi_\lambda : Z(\mathfrak{g}_\Q)\rightarrow \Q_p
\end{eqnarray*}   
which induces the central reduction $\Ub(\mathfrak{g}_\Q)_{\lambda}:=\Ub(\mathfrak{g}_\Q)\otimes_{Z(\mathfrak{g}_\Q),\chi_\lambda+\rho}\Q_p$. If Ker$(\chi_{\lambda+\rho})_{\Z_p}:=D^{(m)}(\G)\cap \text{Ker}(\chi_{\lambda+\rho})$, we can consider the central redaction
\begin{eqnarray*}
D^{(m)}(\G)_{\lambda}:=D^{(m)}(\G)/D^{(m)}(\G)\text{Ker}(\chi_{\lambda+\rho})_{\Z_p}
\end{eqnarray*}
and its $p$-adic completion $\widehat{D}^{(m)}(\G)_{\lambda}$. It is clear that $D^{(m)}(\G)_{\lambda}$ is an integral model of $\Ub(\mathfrak{g}_\Q)_{\lambda}$.
\begin{theo} \label{C.G sections}
The homomorphism of $\Z_p$-algebras $\Phi^{(m)}_{\lambda}: D^{(m)}(\G)\rightarrow  H^{0}(X,\Db^{(m)}_{X,\lambda})$, defined by taking global sections in (\ref{c. map}), induces an isomorphism of $\Z_p$-algebras
\begin{eqnarray*}
\widehat{D}^{(m)}(\G)_{\lambda}\otimes_{\Z_p}\Q_p \xrightarrow{\simeq} H^0\left(\mathfrak{X},\widehat{\Da}^{(m)}_{\mathfrak{X},\lambda,\Q}\right).
\end{eqnarray*}
\end{theo}
\begin{proof}
The key of the proof of the theorem is the following commutative diagram, which is an immediate consequence of remark \ref{rem 5.1.18}
\begin{eqnarray*}
\begin{tikzcd}
D^{(m)}(\G) \arrow[r, "\Phi^{(m)}_{\lambda}"] \arrow[d,hook ]
& H^0(X,\Db^{(m)}_{X,\lambda}) \arrow[d,hook ] \\
\Ub(\mathfrak{g}_\Q) \arrow[r, "\Phi_{\lambda}"]
& H^{0}(X_\Q,\Db_{\lambda}).
\end{tikzcd}
\end{eqnarray*}
By the classical Beilinson-Bernsein theorem \cite{BB} and the preceding commutative diagram, we have that $\Phi^{(m)}_\lambda$ factors through the morphism $\overline{\Phi}^{(m)}_\lambda: D^{(m)}(\G)_{\lambda}\rightarrow H^{0}(X,\Db^{(m)}_{X,\lambda})$ which becomes an isomorphism after tensoring with $\Q_p$. The preceding lemma implies therefore that $\overline{\Phi}^{(m)}_\lambda$ gives rise to an isomorphism
\begin{eqnarray*}
\widehat{D}^{(m)}(\G)_{\lambda}\otimes_{\Z_p}\Q_p\xrightarrow{\simeq} \widehat{H^0(X,\Db^{(m)}_{X,\lambda})}\otimes_{\Z_p}\Q_p,
\end{eqnarray*}
 and proposition \ref{prop 3.1.1} together with the fact that $\mathfrak{X}$ is in particular a noetherian topological space end the proof of the theorem.
\end{proof}
\subsection{The localization functor}
In this section we will introduce the localization functor. For this, we will first fix the following notation which will make more pleasant the reading of the proof of our principal theorem. We will consider 
\begin{eqnarray*}
\widehat{\mathrm{D}}^{(m)}_{\mathfrak{X},\lambda,\Q}:= H^{0}\left(\mathfrak{X}, \widehat{\Da}^{(m)}_{\mathfrak{X},\lambda,\Q}\right).
\end{eqnarray*}
Now, let $E$ be a finitely generated $\widehat{\mathrm{D}}^{(m)}_{\mathfrak{X},\lambda,\Q}$-module. We define $\La oc^{(m)}_{\mathfrak{X},\lambda}(E)$ as the associated sheaf to the presheaf on $\mathfrak{X}$ defined by
\begin{eqnarray*}
\mathfrak{U}\mapsto \widehat{\Da}^{(m)}_{\mathfrak{X},\lambda,\Q}(\mathfrak{U})\otimes_{\widehat{\mathrm{D}}^{(m)}_{\mathfrak{X},\lambda,\Q}}E.
\end{eqnarray*}
It is clear that $\La oc^{(m)}_{\mathfrak{X},\lambda}$ is a functor from the category of finitely generated $\widehat{\mathrm{D}}^{(m)}_{\mathfrak{X},\lambda,\Q}$-modules to the category of coherent $\widehat{\Da}^{(m)}_{\mathfrak{X},\lambda,\Q}$-modules.

\subsection{The arithmetic Beilinson-Bernstein theorem}

We are finally ready to prove the main result of this paper. To start with, we will enunciate the following proposition. \footnote{This proof is exactly as in \cite[proposition 4.3.1]{HPSS}.}

\begin{prop}\label{prop 3.3.1}
Let $\Ea$ be a coherent $\widehat{\Da}^{(m)}_{\mathfrak{X},\lambda,\Q}$-module. Then $\Ea$ is generated by its global sections as $\widehat{\Da}^{(m)}_{\mathfrak{X},\lambda,\Q}$-module. Furthermore, every coherent $\widehat{\Da}^{(m)}_{\mathfrak{X},\lambda,\Q}$-module admits a resolution by finite free $\widehat{\Da}^{(m)}_{\mathfrak{X},\lambda,\Q}$-modules.
\end{prop}
\begin{proof}
By theorem \ref{theo 3.1.5} we know that $\Ea$ is a quotient of a module $\left(\widehat{\Da}^{(m)}_{\mathfrak{X},\lambda,\Q}(-r)\right)^{a}$ for some $r\in\Z$ and some $a\in\N$. We can therefore assume that $\Ea=\left(\widehat{\Da}^{(m)}_{\mathfrak{X},\lambda,\Q}(-r)\right)$ for some $r\in\Z$. Let $F:=H^0(X, \Db^{(m)}_{X,\lambda}(-r))$, a finitely generated $D^{(m)}(\G)$-module by proposition \ref{prop 1.4.2}. Let us consider the linear map of $ \Db^{(m)}_{X,\lambda}$-modules equals to the composite
\begin{eqnarray}
 \Db^{(m)}_{X,\lambda}\otimes_{D^{(m)}(\G)}F\rightarrow  \Db^{(m)}_{X,\lambda}\otimes_{H^0(X, \Db^{(m)}_{X,\lambda})}F\rightarrow \Db^{(m)}_{X,\lambda}(-r)
\end{eqnarray}
where the first map is the surjection induced by the map $\Phi^{(m)}_{\lambda}$ of theorem \ref{C.G sections}. Let $\Fb$ be the cokernel of the composite map. Since $D^{(m)}(\G)$ is noetherian, the source of this map is a coherent $ \Db^{(m)}_{X,\lambda}$-module and so is $\Fb$. Moreover, this module is of $p$-torsion because $ \Db^{(m)}_{X,\lambda}(-r)\otimes_{\Z_p}\Q_p$ is generated by its global sections \cite{BB}. Now, let us take a linear surjection $\left(D^{(m)}(\G)\right)^{\oplus a}\rightarrow F$. By tensoring with $ \Db^{(m)}_{X,\lambda}$ we obtain the exact sequence of coherent modules
\begin{eqnarray*}
\left( \Db^{(m)}_{X,\lambda}\right)^{\oplus a}\rightarrow  \Db^{(m)}_{X,\lambda}(-r)\rightarrow \Fb\rightarrow 0.
\end{eqnarray*}
Passing to $p$-adic completions (which is exact in our situation \cite[chapter II, proposition 9.1]{Hartshorne1}) and inverting $p$ yields the linear surjection.
\end{proof}
\begin{theo}\label{Affinity}
Let us suppose that $\lambda:\text{Dist}(\T)\rightarrow\Z_p$ is a character of $\text{Dist}(\T)$ such that $\lambda + \rho\in\mathfrak{t}_\Q^*$ is a dominant and regular character of $\mathfrak{t}_\Q$. The functors $\La oc^{(m)}_{\mathfrak{X},\lambda}$ and $H^{0}(\mathfrak{X},\bullet)$ are quasi-inverse equivalence of categories between the abelian categories of finitely generated $\widehat{\mathrm{D}}^{(m)}_{\mathfrak{X},\lambda,\Q}$-modules and coherent  $\widehat{\Da}^{(m)}_{\mathfrak{X},\lambda,\Q}$-modules.
\end{theo}
\begin{proof}
The proof of \cite[Proposition 5.2.1]{Huyghe1}  carries over word by word.
\end{proof}
\justify
Given that any equivalence between abelian categroies is exact, theorems \ref{C.G sections} and \ref{Affinity} clearly imply
\begin{theo}(\textbf{Principal theorem})\label{PT}
Let us suppose that $\lambda:\text{Dist}(\T)\rightarrow\Z_p$ is a character of $\text{Dist}(\T)$ such that $\lambda + \rho\in\mathfrak{t}_\Q^*$ is a dominant and regular character of $\mathfrak{t}_\Q$.
\begin{itemize}
\item[(i)] The functors $\La oc^{(m)}_{\mathfrak{X},\lambda}$ and $H^{0}(\mathfrak{X},\bullet)$ are quasi-inverse equivalence of categories between the abelian categories of finitely generated (left) $\widehat{D}^{(m)}(\G)_{\lambda}\otimes_{\Z_p}\Q_p$-modules and coherent  $\widehat{\Da}^{(m)}_{\mathfrak{X},\lambda,\Q}$-modules.
\item[(ii)] The functor $\La oc^{(m)}_{\mathfrak{X},\lambda}$ is an exact functor.
\end{itemize}
\end{theo}
\section{The sheaves $\Da^{\dag}_{\mathfrak{X},\lambda}$}
\justify
In this section we will study the problem of passing to the inductive limit when $m$ varies. Let us recall that $\xi:\widetilde{X}:=\G/\textbf{N}\rightarrow X:= \G/\B$ is a locally trivial $\T$-torsor (subsection 3.4). For every couple of positive integers $m\le m'$ there exists a canonical homomorphism of sheaves of filtered rings \cite[(2.2.1.5)]{Berthelot1}
\begin{eqnarray} \label{inductive system 0}
\rho_{m',m}:\Db^{(m)}_{\widetilde{X}}\rightarrow\Db^{(m')}_{\widetilde{X}}.
\end{eqnarray}
Let us fix a character $\lambda:\text{Dist}(\T)\rightarrow\Z_p$. As we have remarked for $m\le m'$ we have a commutative diagram
\begin{eqnarray}\label{Trans_Dist}
\begin{tikzcd}[row sep=tiny]
D^{(m)}(\G) \arrow[dr, "\lambda"] \arrow[dd, hook ] &    \\
                               & \Z_p. \\
D^{(m')}(\G) \arrow [ur, "\lambda "]
\end{tikzcd}
\end{eqnarray}
Moreover, by \cite[(1.4.7.1)]{Berthelot1} we dispose of a canonical morphism $\Pb^{n}_{\tilde{X},(m')}\rightarrow \Pb^{n}_{\tilde{X},(m)}$.
\justify
In section 3.3 we have defined a $\T$-equivariant structure $\Phi^{n}_{(m)}:\sigma^*\Pb^{n}_{\tilde{X},n}\rightarrow p_1^*\Pb^{n}_{\tilde{X},n}$ on $\Pb^{n}_{\tilde{X},n}$ (w recall for the reader that $\sigma$ denotes the right action of $\T$ on $\tilde{X}$ and $p_1$ the first projection). By universal property of $\Pb^{n}_{\tilde{X},(m)}$ the preceding $\T$-equivariant structures fit into a commutative diagram   
\begin{eqnarray*}
\begin{tikzcd}
\sigma^{*}\Pb^{n}_{\tilde{X},(m')} \arrow[d] \arrow [r, "\Phi^{n}_{(m')}"]
& p_1^{*}\Pb^{n}_{\tilde{X},(m')} \arrow[d]\\
\sigma^{*}\Pb^{n}_{\tilde{X},(m)} \arrow[r,"\Phi^{n}_{(m)}"]
& p_1^{*}\Pb^{(m)}_{\tilde{X},(m)}.
\end{tikzcd}
\end{eqnarray*}
This implies that the morphisms $\Pb^{n}_{\tilde{X},(m')}\rightarrow \Pb^{n}_{\tilde{X},(m)}$ are $\T$-equivariant and therefore by lemma \ref{Eq.Filtration} and lemma \ref{Eq.Dual}, we can conclude that the canonical maps in (\ref{inductive system 0}) are $\T$-equivariant. In this way, we dispose of morphisms $\widetilde{\Db^{(m)}}\rightarrow\widetilde{\Db^{(m')}}$. The diagram (\ref{Trans_Dist}) implies that we also have maps $\Db^{(m)}_{X,\lambda}\rightarrow\Db^{(m')}_{X,\lambda}$ and therefore an inductive system
\begin{eqnarray}\label{inductive system}
\widehat{\xi_*(\rho_{m',m})^{\T}}:\widehat{\Da}^{(m)}_{\mathfrak{X},\lambda}\rightarrow\widehat{\Da}^{(m')}_{\mathfrak{X},\lambda}.
\end{eqnarray}
\begin{defi}\label{daga}
We will denote by $\Da^{\dag}_{\mathfrak{X},\lambda}$ the limit of the inductive system (\ref{inductive system}), tensored with $\Q_p$
\begin{eqnarray*}
\Da^{\dag}_{\mathfrak{X},\lambda}:=\left(\varinjlim_{m}\widehat{\Da}^{(m)}_{\mathfrak{X},\lambda}\right)\otimes_{\Z_p}\Q_p.
\end{eqnarray*}
\end{defi}

\begin{rem}
Let us suppose that $\lambda\in\text{Hom}(\T,\G_m)$ is an algebraic character and let us denote by $\lambda' : \text{Dist}(\T)\rightarrow \Z_p$ the character of $\text{Dist}(\T)$ induced by the correspondence (\ref{Iso_chars_Z_p}) . Let us denote by $\Da^{\dag}_{\mathfrak{X}}(\lambda)$ the inductive limit of the sheaves defined in \ref{char_alg}. By proposition \ref{HSS} we have $\Da^{\dag}_{\mathfrak{X}}(\lambda) = \Da^{\dag}_{\mathfrak{X},\lambda '}$. 
\end{rem}

\subsection{The localization functor $\mathcal{L}oc^{\dag}_{\mathfrak{X},\lambda}$}
\justify
As in the subsection 4.3 let us denote by $\mathrm{D}^{\dag}_{\mathfrak{X},\lambda}:=H^{0}(\mathfrak{X},\Da^{\dag}_{\mathfrak{X},\lambda})$. In a completely analogous way as we have done in the subsection referenced above, we define the localization functor $\mathcal{L}oc^{\dag}_{\mathfrak{X},\lambda}$ from the category of finitely presented $\mathrm{D}^{\dag}_{\mathfrak{X},\lambda}$-modules to the category of coherent $\Da^{\dag}_{\mathfrak{X},\lambda}$-modules. This is, if $E$ denotes a finitely presented $\mathrm{D}^{\dag}_{\mathfrak{X},\lambda}$-module, then $\mathcal{L}oc^{\dag}_{\mathfrak{X},\lambda}(E)$ denotes the associated sheaf to the presheaf on $\mathfrak{X}$ defined by
\begin{eqnarray*}
\mathfrak{U}\mapsto \Da^{\dag}_{\mathfrak{X},\lambda}(\mathfrak{U})\otimes_{\mathrm{D}^{\dag}_{\mathfrak{X},\lambda}}E.
\end{eqnarray*}
\justify
It is clear that $\La oc^{\dag}_{\mathfrak{X},\lambda}$ is a functor from the category of finitely presented $\mathrm{D}^{\dag}_{\mathfrak{X},\lambda}$-modules to the category of coherent $\Da^{\dag}_{\mathfrak{X},\lambda}$-modules.
\subsection{The arithmetic Beilinson-Bernstein theorem for the sheaves $\Da^{\dag}_{\mathfrak{X},\lambda}$}
\justify
In this subsection we will concentrate our efforts to show the following Beilinson-Bernstein theorem for the sheaf of rings $\Da^{\dag}_{\mathfrak{X},\lambda}$.
\begin{theo}\label{BB for dag}
Let us suppose that $\lambda:\text{Dist}(\T)\rightarrow\Z_p$ is a character of $\text{Dist}(\T)$ such that $\lambda + \rho\in\mathfrak{t}_\Q^*$ is a dominant and regular character of $\mathfrak{t}_\Q$.
\begin{itemize}
\item[(i)] The functors $\mathcal{L}oc^{\dag}_{\mathfrak{X},\lambda}$ and $H^{0}(\mathfrak{X},\bullet)$ are quasi-inverse equivalence of categories between the abelian categories of finitely presented (left) $\mathrm{D}^{\dag}_{\mathfrak{X},\lambda}$-modules and coherent $\Da^{\dag}_{\mathfrak{X},\lambda}$-modules.
\item[(ii)] The functor $\mathcal{L}oc^{\dag}_{\mathfrak{X},\lambda}$ is an exact  functor.
\end{itemize}
\end{theo}
\justify
To do this, we recall the following facts.
\begin{rem}\label{compatibility of m}
\begin{itemize}
\item[(i)]
Let us recall that in remark \ref{Invarian global sections of group} we have stated that $D^{(m)}(\T)$ is isomorphic to the subspace of $\T$-invariants $H^{0}(\T,\Db^{(m)}_{\T})^{\T}$. The isomorphism is in fact induced by the the action of $\T$ on itself by right translations \cite[theorem 4.4.8.3]{HS1} and is compatible with $m$ variable. This means that if $Q_m$ and $Q_m'$ denote those isomorphisms for $m\le m'$, then we have a commutative diagram 
\begin{eqnarray*}
\begin{tikzcd}
D^{(m)}(\T) \arrow{r}{Q_m} \arrow{d}{\phi_{m',m}}
& H^{0}(\T,\Db^{(m)}_{\T})^{\T} \arrow{d}{(\gamma_{m',m})^{\T}} \\
D^{(m')}(\T) \arrow{r}{Q_{m'}}
& H^{0}(\T,\Db^{(m')}_{\T})^{\T}
\end{tikzcd}
\end{eqnarray*}
where the morphisms $\phi_{m',m}$ are obtained by dualizing the canonical morphisms $\psi_{m',m}$ in subsection 2.4 and the morphisms $\gamma_{m',m}$ are defined in (\ref{inductive system 0}).
\item[(ii)] Again by remark \ref{Invarian global sections of group} the isomorphism of proposition \ref{why a TDO} are compatible for varying $m$.
\end{itemize}
\end{rem}
\justify
Let us recall the following proposition.
\begin{prop} \cite[Proposition 3.6.1]{Berthelot1}
Let $Y$ be a topological space and $\{\Db_{i}\}_{i\in J}$ be a filtered inductive system of coherent sheaves of rings on $Y$, such that for any $i\le j$ the morphisms $\Db_i\rightarrow\Db_j$ are flat. Then the sheaf $\Db^{\dag}:=\varinjlim_{i\in J}\Db_{i}$ is a coherent sheaf of rings.
\end{prop}

\begin{prop}\label{coh. dag}
The sheaf of rings $\Da^{\dag}_{\mathfrak{X},\lambda}$ is coherent. 
\end{prop}
\begin{proof}
The previous proposition tells us that we only need to show that the morphisms $\widehat{\xi_{*}(\rho_{m',m})^{\T}}_{\Q}$ are flat. As this is a local property we can take $U\in\Sb$ and to verify this  property over the formal completion $\mathfrak{U}$. In this case, remark \ref{compatibility of m} and the argument used in the proof of the first part of proposition \ref{prop coherence } give us, by functoriality, the following commutative diagram  

\begin{eqnarray*}
\begin{tikzcd}[column sep=15ex]
\widehat{\Da}^{(m)}_{\mathfrak{X},\lambda,\Q}(\mathfrak{U}) \arrow{r}{\widehat{\xi_{*}(\rho_{m',m})^{\T}}_{\Q}( \mathfrak{U})} \arrow[d, "\simeqd"]
&\widehat{ \Da}^{(m')}_{\mathfrak{X},\lambda,\Q}(\mathfrak{U}) \arrow[d, "\simeqd"] \\
\widehat{\Da}^{(m)}_{\mathfrak{X},\Q}(\mathfrak{U}) \arrow{r}{ \widehat{\rho}_{m',m,\Q}( \mathfrak{U})}
& \widehat{\Da}^{(m')}_{\mathfrak{X},\Q}(\mathfrak{U})  
\end{tikzcd} 
\end{eqnarray*}
The flatness theorem \cite[theorem 3.5.3]{Berthelot1} states that the lower morphism is flat and so is the morphism on the top.
\end{proof}
\justify
From now on, we suppose that $\lambda:\text{Dist}(\T)\rightarrow\Z_p$ is a character of $\text{Dist}(\T)$ such that $\lambda + \rho\in\mathfrak{t}_\Q^*$ is a dominant and regular character of $\mathfrak{t}_\Q$.

\begin{lem}\label{lemma 5.2.4}
For every coherent $\Da^{\dag}_{\mathfrak{X},\lambda}$-module $\Ea$ there exists $m\ge 0$, a coherent $\widehat{\Da}^{(m)}_{\mathfrak{X},\lambda,\Q}$-module $\Ea_m$ and an isomorphism of $\Da^{\dag}_{\mathfrak{X},\lambda}$-modules
\begin{eqnarray*}
\tau: \Da^{\dag}_{\mathfrak{X},\lambda}\otimes_{\widehat{\Da}^{(m)}_{\mathfrak{X},\lambda,\Q}}\Ea_{m}\xrightarrow{\simeq}\Ea.
\end{eqnarray*}
Moreover, if $(m',\Ea_{m'},\tau')$ is another such triple, then there exists $l\ge\text{max}\{m,m'\}$ and an isomorphism of $\widehat{\Da}^{(l)}_{\mathfrak{X},\lambda,\Q}$-modules
\begin{eqnarray*}
\tau_{l}:\widehat{\Da}^{(l)}_{\mathfrak{X},\lambda,\Q}\otimes_{\widehat{\Da}_{\mathfrak{X},\lambda,\Q}^{(m)}}\Ea_{m}\xrightarrow{\simeq} \widehat{\Da}^{(l)}_{\mathfrak{X},\lambda,\Q}\otimes_{\widehat{\Da}_{\mathfrak{X},\lambda,\Q}^{(m')}}\Ea_{m'}
\end{eqnarray*}
such that $\tau'\circ\left(id_{\Da^{\dag}_{\mathfrak{X},\lambda}}\otimes\tau_l\right)=\tau$.
\end{lem}
\begin{proof}
This is \cite[proposition 3.6.2 (ii)]{Berthelot1}. We remark that $\mathfrak{X}$ is quasi-compact and separated, and the sheaf $\hat{\Da}^{(m)}_{\mathfrak{X},\lambda,\Q}$ satisfies the conditions in \cite[3.4.1]{Berthelot1}.
\end{proof}

\begin{prop}
Let $\Ea$ be a coherent $\Da^{\dag}_{\mathfrak{X},\lambda}$-module.
\begin{itemize}
\item[(i)] There exists an integer $r(\Ea)$ such that, for all $r\ge r(\Ea)$ there is $a\in\N$ and an epimorphism of $\Da^{\dag}_{\mathfrak{X},\lambda}$-modules
\begin{eqnarray*}
\left(\Da^{\dag}_{\mathfrak{X},\lambda}(-r)\right)^{\oplus a}\rightarrow \Ea\rightarrow 0.
\end{eqnarray*}
\item[(ii)] For all $i>0$ one has $H^{i}(\mathfrak{X},\Ea)=0$.
\end{itemize}
\end{prop}
\begin{proof}\footnote{This is exactly as in \cite[theorem 4.2.8]{HPSS}}
Let $\Ea$ be a coherent $\Da^{\dag}_{\mathfrak{X},\lambda}$-coherent module. The preceding proposition tells us that there exists $m\in\N$, a coherent $\widehat{\Da}^{(m)}_{\mathfrak{X},\lambda,\Q}$-module $\Ea_m$ and an isomorphism of $\Da^{\dag}_{\mathfrak{X},\lambda}$-modules
\begin{eqnarray*}
\tau:\Da^{\dag}_{\mathfrak{X},\lambda}\otimes_{\widehat{\Da}^{(m)}_{\mathfrak{X},\lambda,\Q}}\Ea_m\xrightarrow{\simeq}\Ea.
\end{eqnarray*} 
Now we use theorem \ref{theo 3.1.5} for $\Ea_m$ and we get the desired surjection in (i) after tensoring with $\Da^{\dag}_{\mathfrak{X},\lambda}$. To show (ii) we may use the fact that, as $\mathfrak{X}$ is a noetherian topological space, cohomology commutes with direct limites. Therefore, given that $\widehat{\Da}^{(l)}_{\mathfrak{X},\lambda,\Q}\otimes_{{\widehat{\Da}^{(m)}_{\mathfrak{X},\lambda,\Q}}}\Ea_{m}$ is a coherent $\Da^{(l)}_{\mathfrak{X},\lambda,\Q}$-module for every $l\ge m$, we have 
\begin{eqnarray*}
H^{i}(\mathfrak{X},\Ea)=\varinjlim_{l\ge m}H^{i}\left(\mathfrak{X},\widehat{\Da}^{(l)}_{\mathfrak{X},\lambda,\Q}\otimes_{{\widehat{\Da}^{(m)}_{\mathfrak{X},\lambda,\Q}}}\Ea_{m}\right)=0,
\end{eqnarray*}
for every $i>0$.
\end{proof}

\begin{prop}
Let $\Ea$ be a coherent $\Da^{\dag}_{\mathfrak{X},\lambda}$-module. Then $\Ea$ is generated by its global sections as $\Da^{\dag}_{\mathfrak{X},\lambda}$-module. Moreover, $\Ea$ has a resolution by finite free $\Da^{\dag}_{\mathfrak{X},\lambda}$-modules and $H^0(\mathfrak{X},\Ea)$ is a $D^{\dag}_{\mathfrak{X},\lambda}$-module of finite presentation. 
\end{prop}
\begin{proof}\footnote{This is exactly as in \cite[theorem 5.1]{Huyghe1}}
Theorem \ref{theo 3.1.5} gives us a coherent $\widehat{\Da}^{(m)}_{\mathfrak{X},\lambda,\Q}$-module $\Ea_m$ such that $\Ea\simeq\Da^{\dag}_{\mathfrak{X},\lambda}\otimes_{\widehat{\Da}^{(m)}_{\mathfrak{X},\lambda,\Q}}\Ea_m$. Moreover, $\Ea_m$ has a resolution by finite free $\widehat{\Da}^{(m)}_{\mathfrak{X},\lambda,\Q}$-modules ( proposition \ref{prop 3.3.1}). Both results clearly imply the first and the second part of the lemma. The final part of the lemma is therefore a consequence of the first part and the acyclicity of the the functor $H^{0}(\mathfrak{X},\bullet)$.
\end{proof}

\begin{proof}[Proof of theorem \ref{BB for dag}]
All in all, we can follow the same arguments of \cite[corollary 2.3.7]{Huyghe2}. We start by taking $(\mathrm{D}^{\dag}_{\mathfrak{X},\lambda})^{\oplus a}\rightarrow (\mathrm{D}^{\dag}_{\mathfrak{X},\lambda})^{\oplus b}\rightarrow E\rightarrow 0$ a finitely presented $\mathrm{D}^{\dag}_{\mathfrak{X},\lambda}$-module. By localizing and applying the global sections functor, we obtain a commutative diagram
\begin{eqnarray*}
\begin{tikzcd}
(\mathrm{D}^{\dag}_{\mathfrak{X},\lambda})^{\oplus a} \arrow[d] \arrow[r] 
& (\mathrm{D}^{\dag}_{\mathfrak{X},\lambda})^{\oplus b} \arrow[d] \arrow[r] 
& E \arrow[d] \arrow[r] 
& 0\\
(\mathrm{D}^{\dag}_{\mathfrak{X},\lambda})^{\oplus a}  \arrow[r] 
& (\mathrm{D}^{\dag}_{\mathfrak{X},\lambda})^{\oplus b} \arrow[r] 
& H^0(\mathfrak{X},\mathscr Loc_{\mathfrak{X},\lambda}^{\dag}(E)) \arrow[r] 
& 0.
\end{tikzcd} 
\end{eqnarray*}
which tells us that $E\rightarrow H^{0}(\mathfrak{X},\La oc^{\dag}_{\mathfrak{X},\lambda}(E))$ is an isomorphism. To show that if $\Ea$ is coherent $\Da^{\dag}_{\mathfrak{X},\lambda}$-module then the canonical morphism $\Da^{\dag}_{\mathfrak{X},\lambda}\otimes_{\mathrm{D}^{\dag}_{\mathfrak{X},\lambda}}H^{0}(\mathfrak{X},\Ea)\rightarrow\Ea$  is an isomorphism the reader can follow the same argument as before. As we have remarked, the second assertion follows because any equivalence between abelian categories is exact. 
\end{proof}

\subsubsection{Calculation of global sections}
\justify
Let us recall that in the subsection 4.2 we have used the fact that  associated to the  linear form $\lambda\in\mathfrak{t}^*_\Q$ there exists a central character $\chi_{\lambda}: Z(\mathfrak{g}_\Q)\rightarrow \Q_p$, where $Z(\mathfrak{g}_\Q)$ denotes the center of the universal enveloping algebra $\Ub(\mathfrak{g}_\Q)$. In this case, if Ker$(\chi_{\lambda+\rho})_{\Z_p}:=D^{(m)}(\G)\cap \text{Ker}(\chi_{\lambda+\rho})$, we can consider the central redaction
\begin{eqnarray*}
D^{(m)}(\G)_{\lambda}:=D^{(m)}(\G)/D^{(m)}(\G)\text{Ker}(\chi_{\lambda+\rho})_{\Z_p}
\end{eqnarray*}
and its $p$-adic completion $\widehat{D}^{(m)}(\G)_{\lambda}$. We have $D^{(m)}(\G)_{\lambda}\otimes_{\Z_p}\Q_p= \Ub(\mathfrak{g}_\Q)_{\lambda}$. We have shown that there exists a canonical isomorphism of $\Z_p$-algebras $\widehat{D}^{(m)}(\G)_{\lambda}\otimes_{\Z_p}\Q_p \xrightarrow{\simeq} H^0\left(\mathfrak{X},\widehat{\Da}^{(m)}_{\mathfrak{X},\lambda,\Q}\right)$. Taking the inductive limit we can conclude that if $D^{\dag}(\G)_{\lambda}:=\varinjlim_{m}\widehat{D}^{(m)}(\G)_{\lambda}$, then we also have a canonical isomorphism of $\Z_p$-algebras $D^{\dag}(\G)_{\lambda}\xrightarrow{\simeq} H^{0}(\mathfrak{X},\Da^{\dag}_{\mathfrak{X},\lambda})$. Theorem \ref{BB for dag} and this calculation complete the Beilinson-Bernstein correspondence.



\end{document}